\documentclass[11pt]{article}
\usepackage[dvipsnames]{xcolor}
\usepackage{amsfonts}
\usepackage{amsmath} 
\usepackage{amssymb}
\usepackage{graphicx}
\usepackage{amsthm}
\usepackage{mathtools}
\usepackage{bbm}
\usepackage{mathrsfs}
\usepackage{dirtytalk}
\usepackage{fullpage}
\usepackage{alphabeta}
\usepackage{tikz-cd}
\usepackage{tikz}

\usepackage{shuffle}
\usepackage{euscript}
\usepackage{ebgaramond}
\usepackage{comment}
\usetikzlibrary{calc}
\usepackage{wrapfig}
\usepackage{lipsum}
\usepackage{scalerel}

\usepackage[T1]{fontenc}  

\tikzset{every path/.append style={line cap=round, line width=1pt}}

\makeatletter
\pgfkeys{
	/dotrows/.is family, /dotrows,
	dotsize/.store in=\DR@dotsize, dotsize=2.5pt, 
	colsep/.store in=\DR@colsep,   colsep=6mm,    
	rowsep/.store in=\DR@rowsep,   rowsep=,       
	prefix/.store in=\DR@prefix,   prefix=p,      
	style/.code=\def\DR@style{#1}, style=         
}

\newcommand{\dotrows}[2][]{%
	\pgfkeys{/dotrows, #1}%
	\def\DR@rowsepFinal{\ifx\DR@rowsep\@empty \DR@colsep\else \DR@rowsep\fi}%
	\begin{scope}
		\foreach[count=\i] \N in {#2}{%
			\foreach \j in {1,...,\N}{%
				\node[circle, inner sep=0pt, minimum size=\DR@dotsize, fill,
				style/.expanded=\DR@style]
				(\DR@prefix-\i-\j)
				at ({(\j-1)*\DR@colsep},{-(\i-1)*\DR@rowsepFinal}) {};
			}%
		}%
	\end{scope}%
}
\makeatother

\makeatletter
\newcommand{\dotcoord}[1]{\expandafter\dotcoordpar@split#1,\@nil}
\def\dotcoordpar@split#1,#2,\@nil{p-#1-#2.center}

\newcommand{\dotcoordpref}[2]{\expandafter\dotcoordprefpar@split#1,\@nil{#2}}
\def\dotcoordprefpar@split#1,#2,\@nil#3{#3-#1-#2.center}
\makeatother

\numberwithin{equation}{section}
\usepackage{thmtools}
\declaretheorem[numberwithin=section]{theorem}
\declaretheorem[sibling=theorem]{proposition}
\declaretheorem[sibling=theorem]{lemma}
\declaretheorem[sibling=theorem]{corollary}

\theoremstyle{definition}

\declaretheorem[sibling=theorem]{example}
\declaretheorem[sibling=theorem]{definition}
\declaretheorem[sibling=theorem]{conditions}
\theoremstyle{remark}
\declaretheorem[sibling=theorem]{remark}

\def\cH{\mathcal{H}}

\def\bR{\mathbb{R}}

\def\Rd{\bR^d}

\def\D{\mathrm{D}}
\def\E{\mathbb{E}}
\def\I{\EuScript{I}}
\def\A{\EuScript{A}}
\def\M{\EuScript{M}}
\def\K{\EuScript{K}}

\def\P{\mathcal{P}}
\def\calD{\mathcal{D}}

\def\calH{\mathcal{H}}
\def\ka{\kappa}
\def\dia{\diamond}

\makeatletter
\newcommand{\bigboxdot}{\DOTSB\mathop{\mathpalette\bigboxdot@{}}\displaylimits}
\newcommand{\bigboxdot@}[2]{%
\vcenter{\hbox{%
		\sbox\z@{$#1\bigotimes$}%
		\resizebox{!}{\ht\z@}{$#1\boxdot$}%
}}%
}
\makeatother

\DeclareFontFamily{U}{bigshuffle}{}
\DeclareFontShape{U}{bigshuffle}{m}{n}{
<5-8> s*[1.6] shuffle7
<8->  s*[1.4] shuffle10
}{}
\DeclareSymbolFont{BigShuffle}{U}{bigshuffle}{m}{n}
\DeclareMathSymbol\bigshuffle{\mathop}{BigShuffle}{"001}
\DeclareMathSymbol\bigcshuffle{\mathop}{BigShuffle}{"002}

\def\dif{\mathrm{d}}

\usepackage{hyperref}

\usepackage{iftex}
\iftutex
\usepackage{unicode-math}
\setmathfontface\altgrfont{GFS Artemisia Italic}[Scale=MatchLowercase]

\else
\usepackage[T1]{fontenc}
\DeclareSymbolFont{altgr}{OML}{antt}{m}{it}
\DeclareMathSymbol{\sko }{\mathord}{altgr}{"0E}
\fi

\usepackage{upgreek}
\usepackage{euscript}

\title{\textsc{Wick integrals}}
\author{
	Carlo Bellingeri\thanks{IRIMAS, Université de Haute-Alsace, France. Email: \href{mailto:carlo.bellingeri@uha.fr}{carlo.bellingeri@uha.fr}}
	\and
	Emilio Ferrucci\thanks{SISSA, Trieste, Italy. Email: \href{mailto:emilio.ferrucci@sissa.it}{emilio.ferrucci@sissa.it}}
}
\date{July 27, 2026}

\begin{document}
\maketitle

\begin{abstract}
	\noindent We introduce the \emph{Wick integral} 
	\[
	\int_s^t f(X_u) \dia \dif X_u
	\]
	for a class of stochastic processes $X$ of bounded $2 > p$-variation which are not necessarily Gaussian. The integral is defined for a class of entire functions $f$ depending on the process. In the case of $1/2 < H$-fractional Brownian motion, the Wick integral agrees with the divergence operator in Malliavin calculus. It satisfies a correction formula with the Young integral $\int \hspace{-0.15em} f(X)\dif X$ and an It\^o formula which have infinitely many correction terms, given by integration against the cumulant functions of $X$, and reduce to familiar identities in the Gaussian case. These results are obtained by developing diagram formulae for Appell polynomials w.r.t.\ a linear span of reference random variables and extending them to series via absolute convergence in $L^1$. Our theory applies to processes taking values in the second Wiener chaos, such as the Rosenblatt process.
\end{abstract}



\section*{Introduction}	
\addcontentsline{toc}{section}{Introduction}

In \cite{Appell_1880} Appell introduced sequences of polynomials that are centred under given a measure and behave like ordinary monomials under differentiation. This idea was later reintroduced by Wick \cite{wick}, whose notation is still used in quantum field theory:
\[
\partial_i : \hspace{-0.2em} X_1^{k_1} \cdots X_d^{k_d}  \hspace{-0.2em} : \ = \ k_i :  \hspace{-0.2em} X_1^{k_1} \cdots X_i^{k_i - 1} \cdots X_d^{k_d}  \hspace{-0.2em} :, \qquad \mathbb E[:  \hspace{-0.2em}X_1^{k_1} \cdots X_d^{k_d} \hspace{-0.2em} :] = 0.
\]
We refer to \cite{simon} for a comprehensive treatment. These polynomials---we shall refer to them as Appell polynomials, while attributing the product to Wick---entered the contemporary literature in probability theory in the '80s thanks to their relevance to limit theorems \cite{Avram_1987}. For Gaussian distributions they are much better known as the Hermite polynomials, and are orthogonal. Wick products in the more general sense continue to generate interest in the context of the renormalisation, see for example the use in \cite{HaiPhi43} for non-Gaussian approximations to space-time white noise.

The Wick product, which generates a sequence of Appell polynomials starting from the probability measure, is not an associative operation, as the probability law of the Appell polynomials is not the same as that of its factors. One can however impose associativity of the Wick product w.r.t.\ a set of reference random variables, and obtain an associative product. This is, again, mainly considered when the reference measure is Gaussian \cite{Janson_1997}, in which case it mirrors (through the Wiener chaos isometry) the symmetric product on Fock space. This product is closely related to the divergence operator in Malliavin calculus, which can, in many cases, be approximated with Riemann-Stieltjes sums in which integrand and integrator are Wick-multiplied \cite{alos, NT06}. The idea of making Wick products associative works more abstractly for general reference measures admitting finite moments, and was recently introduced by \cite{niko} in the context of Hopf-algebraic renormalisation. Following \cite{Janson_1997}, we continue to call it the Wick product.

The aim of this work is to put the associative Wick product to use in the construction of a stochastic integration theory for processes that are not required to be Gaussian. Our contributions are as follows. In \autoref{sec:product}, after introducing some groundwork to the Wick product, we develop diagram formulae to compute the product of Appell polynomials and the Wick product of ordinary monomials, extending the moment and cumulant formulae in \cite{Giraitis_1986}. A \say{change of chaos} formula is developed for Appell polynomials of Appell polynomials, making precise the aforementioned failure in associativity. In \autoref{sec:series} we extend some of these results to entire functions, requiring the cumulant-generating function to have a positive radius of analyticity at the origin, and the function to have exponential type bounded by this radius. We show how classical identities from Malliavin calculus, such as Stroock's formula, carry over to our case, and by passing to the expectation in one such identity we recover an integration-by-parts relation of Barbour \cite{barbourAsymp} under more general hypotheses on the function. The material in this section rests on the theory of entire functions of exponential type \cite{boasEntire, BB}, and their series expansions in Appell polynomials: we have collected some auxiliary material (which did not seem to be immediately available for several  variables, but is nevertheless a straightforward extension of the single variable case) in \autoref{app:proof}. In \autoref{sec:ints} we define the \emph{Wick integral} \autoref{def:wickint} of a one-form (satisfying the conditions of the previous section) against a multidimensional stochastic process $X$ with sample paths of bounded $2 > p$-variation and with a sufficiently well-behaved cumulant-generating function. We do this in terms of Riemann-Stieltjes-Wick sums, similarly to \cite{alos, NT06}, but for laws that need not be Gaussian. In \autoref{thm:itostrat} we prove that such sums converge to the sum of a Young integral and a series of Stieltjes integrals against the cumulant functions of the process. We deduce a corresponding change-of-variables formula \autoref{thm:ito}, which in the scalar case reads
\begingroup
\newcommand{\AboveBraceGap}{6pt}
\newcommand{\BelowBraceGap}{5pt}
\makeatletter
\def\overbrace#1{\mathop{\vbox{\m@th\ialign{##\crcr\noalign{\kern1\p@}
	\downbracefill\crcr\noalign{\kern\AboveBraceGap\nointerlineskip}
	$\hfil\displaystyle{#1}\hfil$\crcr}}}\limits}
\def\underbrace#1{\mathop{\vtop{\m@th\ialign{##\crcr
	$\hfil\displaystyle{#1}\hfil$\crcr\noalign{\kern\BelowBraceGap\nointerlineskip}
	\upbracefill\crcr\noalign{\kern1\p@}}}}\limits}
\makeatother
\[
\newcommand{\Aa}{\int_s^t f'(X_u) \dia \dif X_u}
\newcommand{\Bb}{\frac 12 \int_s^t f''(X_u) \kappa_2(\dif u)}
\newcommand{\Cc}{\sum_{n = 3}^\infty \frac{1}{n!}
	\int_s^t f^{(n)}(X_u) \kappa_n(\dif u)}
f(X_t) - f(X_s)
=
\underbrace{
	\overbrace{\Aa}^{\text{centred stochastic integral}}
	+
	\mathrlap{
		\overbrace{\hphantom{\Bb + \Cc}\vphantom{\Aa}}^{\text{locally deterministic corrections to centering}}
	}
	\Bb\vphantom{\Cc}
}_{\text{Gaussian case}}
+
\underbrace{\Cc}_{\text{non-Gaussian terms}}.
\]
\endgroup
For Gaussian processes, restricting to our class of functions of $X_t$, we recover It\^o-Stratonovich and It\^o formulae in the literature \cite{D99, DHP00, alos, bender, biagini} (and we note that their formulae hold for much more general integrands than the ones we consider here, thanks to orthogonality of Wiener chaos and the very well developed frameworks of Malliavin calculus and white noise analysis). Such formulae have been substantially extended, still in the Gaussian case, but in varying regularity regimes and for different types of integrands, e.g.\ \cite{HJT13, CL19, Ben20} among many authors. For non-Gaussian measures our formulae have arbitrarily many correction terms, an expression of the famous characterisation of Gaussian laws as the only ones with finitely many non-zero cumulants \cite{marcin}. In \autoref{sec:wiener}, we specialize our general theory to random variables in finite Wiener chaos, with particular emphasis on the second chaos, the only chaos beyond the zeroth and first whose random variables have moment-generating function which is analytic in $0$. Utilising the underlying Gaussian structure, we obtain explicit diagrammatic and operator representations for cumulants and for changes of Wick product (from the ordinary Gaussian to the intrinsic one), and verify the conditions that guarantee Appell expansions. In \autoref{sec:2chaos_processes}, we provide sufficient conditions ensuring that second-chaos processes satisfy the conditions necessary for Wick integration, and derive the resulting form of the Wick integral and It\^o-type formulae. As a main example, we consider the Rosenblatt process $X^H$ \cite{Taqqu1975}, for which the required assumptions can be checked explicitly and the correction terms take a particularly simple form due to the stationarity and self-similarity of the process. We then compare the Wick integral for the Rosenblatt process with the white-noise approach of \cite{Arras15}: while the latter leverages the underlying Gaussian noise and the machinery of Hida distributions \cite{Kuo96}, it also introduces auxiliary objects that are not intrinsic to $X^H$ itself. By contrast, our construction is formulated directly in terms of the law of the process, and our change-of-chaos identity allows us to recover Arras's It\^o formula for polynomials.

\subsection*{Comparison with other previous approaches, and outlook}

In addition to \cite{Arras15} (which we consider to be the work that is most relevant to our approach), there is a rich literature on Malliavin calculus for processes that are not Gaussian. These usually work by specifying the measure in advance and considering the analogue of Wiener chaos given by orthogonal polynomials, see for example \cite[Ch.\ 10, 11]{NuaNua}. These approaches have the benefits of orthogonality---the fact that arbitrary $L^2$ functions can be expanded---but must give up many nice identities true in the Gaussian case, like compatibility of the derivative with the chaos polynomials. Unlike these approaches, the one considered here is, for the time being, essentially only finite dimensional and makes strong requirements of the functions being considered, but it is arguably more faithful to the Gaussian case from the algebraic point of view.

Another line of work originates in \cite{ADKS96}. Here the measure is required to have a certain log-derivative in $L^1$, which gives rise to an integration-by-parts relation. These rather restrictive conditions (which would rule out most examples considered here) are loosened in \cite{KSWY98} to ones which enable the treatment of many interesting measures, such as the Poisson law, but the theory for the divergence operator is not as well developed. These approaches are based on considering a biorthogonal system, with Appell polynomials $P$ on the \say{derivative} side and a different $Q$ system (which in general is distributional) on the adjoint \say{divergence} side. In other words, the aspect of Malliavin calculus that is extended here is the adjoint property of derivative and divergence, while integrating $P_n$ does not yield $P_{n+1}$. This is different from our approach, in which derivative and Wick integration are not considered as being adjoint but behave well w.r.t.\ the same system of Appell polynomials, as in the Gaussian case (for which both properties hold, since $P = Q$). In the future, we believe it would be interesting to conduct a more thorough comparison between the two approaches.

We deem several extensions of our work to be worthy of consideration. On the one hand, one could make the theory work for infinite-dimensional factors, allowing for path-dependent and even anticipating integrands, which is the norm in Malliavin calculus. We also believe a rough extension of our theory is within reach; this would cover the most important case of Brownian motion, and more generally $\frac 12 \geq H$-fBm (with $H > \frac 14$ in the multidimensional case). On the other hand, the rather strong restriction to entire functions of exponential type seems to be the best achievable with the methods considered; it is nevertheless conceivable that resummation techniques (e.g.\ Borel summation) could be employed to extend this aspect too.

\paragraph{Acknowledgements}

We are very grateful to Nikolas Tapia for introducing us to his work on Wick products and for helpful discussions on several topics related to those of this paper.

During the first phase of this project, CB was employed at the Université de Lorraine and supported by the ERC Starting Grant \emph{Low Regularity Dynamics via Decorated Trees (LoRDeT)}, and EF was employed at the University of Oxford and supported through the EPSRC programme grant [EP/S026347/1]. EF's research is currently supported through the ERC grant SQGT [101116964].

\section{The Wick product, diagram formulae, and change of chaos}\label{sec:product}

We consider a collection of real-valued random variables $\mathcal{X}=(X^i,i \in \mathcal I)$ defined on a probability space $(\Omega, \mathcal{F}, \mathbb{P})$ and indexed by a set $\mathcal I$, with $\mathcal F = \sigma(X^i \colon i \in \mathcal I)$. Let $x^i, i \in \mathcal I$ be indeterminates indexed by the same set; we will always distinguish random variables from indeterminates by using upper case letters for the former and lower case for the latter. Denote
\[
x^I \coloneqq x^{i_1} \cdots x^{i_n}, \qquad X^I \coloneqq X^{i_1} \cdots X^{i_n}
\]
for $I= \{i_1,\ldots,i_n\}$ a finite submultiset of  $\mathcal I$ (a multiset---repeated entries are allowed and order does not matter---with $i_k \in \mathcal I$ for $k = 1,\ldots, n$, even though $\mathcal I$ itself is a set). We assume finiteness of all mixed moments, $\E[|X^I|] < \infty$ for all such $I$, so that $X^I \in L^p(\Omega)$ for any $p$, while $x^I$ is an element of the polynomial algebra $\mathbb R[x^\mathcal{I}]$.

We use $\kappa[X^I] \coloneqq \ka[X^{i_1},\ldots,X^{i_n}]$ to denote joint cumulants, and recall the Leonov-Shiraev relations: denoting by $\P(I)$ the set of partitions of any generic multiset $I$, and using the notation
\[
\E[X]^\pi = \prod_{J \in \pi} \E[X^J],\qquad \ka[X]^\pi = \prod_{J \in \pi} \ka[X^J]
\]
for any $\pi \in \P(I)$, we have the identities
\begin{equation}
\begin{split}
\E[X^I] &= \sum_{\pi \in \P(I)} \ka[X]^\pi\\
\ka[X^I] &= \sum_{\pi \in \P(I)} (|\pi| - 1)!(-1)^{|\pi| - 1} \E[X]^\pi\text{ .}
\end{split}
\end{equation}
We shall adopt the conventions $\E[X]^\varnothing = 1 = \kappa[X]^\varnothing$. Note that while $\mathbb E[X^{I_1}\cdots X^{I_m}] = \mathbb E[X^I]$ with $I = I_1 \sqcup \ldots \sqcup I_m$, in general $\kappa[X^{I_1},\ldots, X^{I_m}] \neq \kappa[X^I]$.

Partial differentiation acts on monomials by
\begin{equation}\label{eq:multinomial}
\partial_J x^I = \sum_{J \subset I}x^{I \setminus J}, 
\end{equation}
where we are summing over distinct multiset embeddings $J \subset I$. Appell polynomials \cite{Appell_1880,wick} are polynomials, each associated to a multiset, that continue to satisfy this property while being centred. We denote these $x^{\dia I} \in \mathbb R[x^\mathcal{I}]$ and
\begin{equation}\label{eq:wickX}
	X^{\dia I} = x^{\dia I}|_{x^i = X^i, \ i \in I} \in L^2(\Omega) \text{ .}
\end{equation}
In the following proposition we recall their definition in four equivalent ways. We refer to \cite[Sections 1-3]{MR3535310} for the proofs of the following results.

\begin{proposition}[Appell polynomials associated to a measure]\label{prop:wick}
The following are equivalent definitions of the \emph{Appell polynomials} $x^{\dia I}$, defined for $I \subset \mathcal I$ finite.
\begin{enumerate}
\item (Recursive definition)
\[
x^{\diamond I} = x^I- \sum_{\varnothing\neq J\subset I}\mathbb{E}[X^J]x^{\diamond I\setminus J},\quad x^{\dia \varnothing} = 1 \text{ ,}
\]
from which the inverse expression follows:
\begin{equation}\label{eq:inverse}
x^I = \sum_{J \subset I} x^{\dia J} \sum_{\pi \in \P(I \setminus J)} \ka[X]^\pi \text{ .}
\end{equation}
\item (Zero-mean Appell sequence.) For any $j \in \mathcal I$, $\partial_j x^{\dia I} = 0$ if $j \not \in I$ and otherwise
\begin{equation}\label{eq_zero_average}
\partial_j x^{\dia I} = \sum_{j \in I}x^{\dia I \setminus j}, \quad \E X^{\dia I} = 0 \quad \text{for } I \neq \varnothing, \quad x^{\dia \varnothing} = 1 \text{ .}
\end{equation}
	
\item (Closed-form expression.) 
\begin{equation}\label{eq:wick}
x^{\diamond I} = \sum_{J \subset I} x^J \sum_{\pi \in \P(I \setminus J)} (-1)^{|\pi|} \ka[X]^\pi, \quad x^{\dia \varnothing} = 1 \text{ .}
\end{equation}
\item (Generating series.) If the moment-generating function of finite collections of random variables indexed by $\mathcal I$ is analytic in a neighbourhood of $0$,
\[
x^{\dia I} = \partial_{\theta_I} \exp \Big(\sum_{i} \theta_i x^i - \EuScript{K}(\theta)\Big)\Big|_{\theta = 0} \text{ ,}
\]
where $\EuScript{K}(\theta) \coloneqq \log \E \exp(\sum_i \theta_i X^i)$ is the cumulant generating function, with the sum ranging finitely over distinct indices in $I$.
\end{enumerate}
\end{proposition}

Assuming $\mathbb E X = 0$, the first four Appell polynomials in terms of moments are, for $i,j,h,k \in \mathcal I$:
\begin{equation}\label{ex:1234}
\begin{split}
x^{\dia i} ={} &x^i, \qquad x^{\dia ij} = x^{ij} - \E[X^{ij}]\\
x^{\dia ijh} ={} &x^{ijh} - \E[X^{ij}]x^h - \E[X^{ih}]x^j - \E[X^{jh}]x^i - \E[X^{ijh}] \\
x^{\dia ijhk} ={} &x^{ijhk} - \E[X^{ij}]x^{hk} - \E[X^{ih}]x^{jk} - \E[X^{ik}]x^{jh} - \E[X^{jh}]x^{ik} - \E[X^{jk}]x^{ih} - \E[X^{hk}]x^{ij} \\
&- \E[X^{ijh}]x^k - \E[X^{ijk}]x^h - \E[X^{ihk}]x^j - \E[X^{jhk}]x^i\\
&- \kappa[X^{ijhk}] + \E[X^{ij}]\E[X^{hk}] + \E[X^{ih}]\E[X^{jk}] +\E[X^{ik}]\E[X^{jh}] \text{ ,}
\end{split}
\end{equation}
where we have used/may further substitute the expressions for joint cumulants of centred random variables
\begin{align*}
\ka[X^i] &= \E[X^i] = 0, \quad \ka[X^{ij}] = \E[X^{ij}], \quad \ka[X^{ijh}] = \E[X^{ijh}], \\
\ka[X^{ijhk}] &= \E[X^{ijhk}] - \E[X^{ij}]\E[X^{hk}] - \E[X^{ih}]\E[X^{jk}] - \E[X^{ik}]\E[X^{jh}] \text{ .}
\end{align*}

Once the base random variables $X$ are fixed, we consider the \emph{Wick product} defined in \cite[Theorem 5.3]{niko} by
\begin{equation}\label{eq:wickprod}
x^{\diamond I} \diamond x^{\diamond J} \coloneqq x^{\diamond I \sqcup J }\text{ ,}
\end{equation}
and extended to arbitrary polynomials in $X$ thanks to the fact that Appell polynomials are monic \eqref{eq:inverse} and therefore linearly generate the polynomial algebra. This product is the one used in \cite{Janson_1997} (denoted $\odot$) and \cite[Definition 2.4]{the_wick} in the case of $X$ Gaussian; in \cite{niko} it is viewed through the lens of Hopf-algebraic deformation. The operation $\dia$ is a commutative and associative product on $\mathbb R[x^\mathcal{I}]$.

\begin{remark}\label{rem:centred}
	Little generality is lost by assuming $X$ to be centred. This is because, calling $\widetilde X \coloneqq X - \mu$ with $\mu = \mathbb E X$ and corresponding indeterminates $\widetilde x = x - \mu$, we have $\widetilde x^{\dia I} = x^{\dia I}$. Indeed, by multilinearity of cumulants and calling $\mathcal P_{\geq 2}$ the set of partitions of a multiset into multisets of cardinality at least $2$, by the closed-form expression in \autoref{prop:wick}
	\begin{align*}
		\widetilde x^{\dia I} &= \sum_{J \subset I} (x - \mu)^J \sum_{\pi \in \P(I \setminus J)} (-1)^{|\pi|} \ka[X - \mu]^\pi \\
		&= \sum_{J \subset I} \sum_{k = 0}^{|J|} (-1)^k \sum_{\substack{K \subset J \\ |K| = k}} \mathbb E [X]^K x^{J \setminus K} \sum_{\pi \in \P_{\geq 2}(I \setminus J)} (-1)^{|\pi|} \ka[X]^\pi \\
		&= x^{\dia I}
	\end{align*}
	thanks to the same formula, allowing for singletons. Notice that when $X$ is not centred, $p(x) \dia x^j$ is not generally centred (when evaluated at $x = X$) for a polynomial $p$, instead 
	\[
	p(x) \dia x^j = p(x) \dia x^{\dia j} + p(x) \mu^j
	\]
	where $x^{\dia j} =  \widetilde x^j = x^j - \mu^j$ and $p(x) \dia x^{\dia j}$ is centred; instead $\E[p(X)\dia X^j] = \E[p(X)]\mu^j$. For example
	\[
	x \dia x = (x^{\dia 1} + \mu) \dia (x^{\dia 1} + \mu) = x^{\dia 2} + 2\mu x^{\dia 1} + \mu^2,
	\]
	which is not the same as $x^{\dia 2}$, and has mean $\mu^2$.
\end{remark}

\begin{remark}[Lack of associativity]\label{rem:ass}
The polynomials $x^{\dia I}$ and the product $\dia$ are only defined w.r.t.\ the base collection of random variables $X$: we will sometimes emphasise this by writing $\dia_X$. Let $Y^k$ be random variables that are polynomial in $X$, $Y^k = p^k(X)$. Then, unless the $p^k$'s are of degree $1$, there is a distinction between performing the Wick product $\dia_Y$ on polynomials in the indeterminates $y$ and substituting $y^k = p^k(x)$, and performing the Wick product $\dia_X$ on polynomials in $p^k(x)$. For example, letting $y^1 = x^{\dia I}$ and $y^2 = x^{\dia J}$, unless $I$ and $J$ are singletons, generally
\[
y^1 \dia_Y y^2 \big|_{y^1 = x^{\dia I}, y^2 = x^{\dia J}} = x^{\dia I}x^{\dia J} - \mathbb E[X^{\dia I}X^{\dia J}] \neq x^{\dia I \sqcup J} = x^{\dia I} \dia_X x^{\dia J} \text{ ,}
\]
and therefore generally $Y^1 \dia_Y Y^2 \neq Y^1 \dia_X Y^2$, i.e.\ $\dia_X$ and $\dia_Y$ define different products. The full relationship between $\dia_X$ and $\dia_Y$ will be explored in \autoref{thm:change} below. Sometimes the Wick product is simply defined as the Appell polynomial w.r.t.\ its factors, without any reference to a base set of random variables, denoted $: Y^1 \cdots Y^n :$ in the theoretical physics literature. See \cite[\S 2d)]{the_wick} for a discussion about the distinction between \say{the Wick product in physics} and \say{the Wick product in stochastic analysis} in the Gaussian case, the latter being the one of interest here, denoted $\dia$. The exception that proves the rule is when the random variables $Y$ are obtained as a linear combination of the ones $X$, then $\dia_Y = \dia_X$. More precisely, let $Y^j = \sum_i \lambda^j_i X^i$ with the sum finite, then for a multiset $J = j_1 \ldots j_n$, it follows trivially from \autoref{prop:wick} and homogeneity of monomials and cumulants that
\begin{align*}
y^{\dia_Y J}\big|_{y^j = \sum_i \lambda^j_i x^i} &= \Big( \sum_i \lambda^{j_1}_i x^i \Big) \dia_X \cdots \dia_X \Big( \sum_i \lambda^{j_n}_i x^i \Big), \qquad \text{and therefore} \quad
Y^{\dia_Y J} = Y^{\dia_X J} \text{ .}
\end{align*}
For this reason, from now on we will often identify $\mathcal X$ with its linear span.
\end{remark}

\begin{remark}[Well defined product]\label{rem:wd}
We are also interested in the Wick product as a product of polynomial random variables, not abstract polynomials. This is how it is defined in \cite{Janson_1997} in the Gaussian case. For this to be well-defined, however, it must hold that if $p, q \in \mathbb R[x^\mathcal{I}]$ with $p(X) = 0$ a.s.\ then $p(X) \dia q(X) = 0$ a.s. The implication can be guaranteed by requiring that the condition be trivial, i.e.\ that no polynomial vanishes on $X$; this ensures that the polynomial $p$ is identifiable from the law of $p(X)$ and that of $X$. This in turn is implied by requiring that for each finite subset $\{i_1,\ldots,i_d\} \subset \mathcal I$, the law $\mu$ of $(X^{i_1},\ldots,X^{i_d})$ admits a positive density on an open set of $\mathbb R^d$. This condition is also useful for guaranteeing well-definedness in the case in which the factors are entire functions, treated below in \autoref{sec:series}, since a power series that vanishes on an open set must vanish. In fact, it is enough for $\mu$ to have a density on an open set $U \subset \mathbb R^d$ with $\mu[U] > 0$ (without specifying that the density must be positive on $U$), since the zero-set of a non-zero analytic function in several variables must have lebesgue measure zero \cite{Mit20}. In this case the Wick product defines a commutative and associative product on polynomial random variables in $X$
\[
\dia \colon \mathbb R[X^\mathcal{I}] \odot \mathbb R[X^\mathcal{I}] \to \mathbb R[X^\mathcal{I}] \text{ .}
\]
We now give a simple example of how $\dia$ can fail to be well defined as a product on random variables.
\end{remark}

\begin{example}[Lack of well-definedness]\label{ex:wd}
	This example discusses a specific instance of the issues described in \autoref{rem:ass} and \autoref{rem:wd}. Let $(X^1,X^2) \sim \mathcal N(0, I_2)$. Let 
	\[
	Y^1 = (X^1)^2 - 1, \quad Y^2 = (X^2)^2 - 1,\quad Y^3 = Y^4 = X^1X^2 \text{ .}
	\]
	Now consider the polynomial
	\[
	p(y^1, y^2, y^3) = (y^3)^2 - (y^1 + 1)(y^2 + 1) \text{ .}
	\]
	We will show that, even though $p((x^1)^2 - 1, (x^2)^2 - 1, x^1x^2) = 0$,
	\[
	p(y^1, y^2, y^3) \dia_Y y^4\big|_{y^1 = (x^1)^2 - 1, y^2 = (x^2)^2 - 1, y^3 = y^4 = x^1x^2} \neq 0.
	\]
	This example shows that substitution of variables does not commute with Wick products. More importantly, it shows that the Wick product is not well defined as a product of random variables unless these are free of polynomial relations: indeed, since $p(Y^1, Y^2, Y^3) \dia_Y Y^4$ is now a non-trivial polynomial in a bivariate Gaussian, which is free of polynomial relations, it does not vanish a.s.\ despite $p(Y^1, Y^2, Y^3)$ vanishing identically. We compute
	\begin{align*}
		&p(y^1, y^2, y^3) \dia_Y y^4\\
		={}&(y^3)^2 \dia y^4 - y^1y^2 \dia y^4 - y^1 \dia y^4 - y^2 \dia y^4 - y^4 \\
		={}&(y^3)^2 y^4 - 2 \kappa[Y^3, Y^4]y^3 - \kappa[Y^3, Y^3, Y^4] - y^1y^2y^4 + \kappa[Y^1,Y^4]y^2 + \kappa[Y^2,Y^4]y^1 + \kappa[Y^1, Y^2, Y^4] \\
		&-y^1y^4 + \kappa[Y^1,Y^4] - y^2y^4 + \kappa[Y^2, Y^4] - y^4\\
		=&-2y^3 \text{ ,}
	\end{align*}
	where the last identity is obtained by computing the cumulants with \autoref{prop:EkW} (using the fact that the $Y^k$'s are Hermite polynomials---indeed, the whole calculation can alternatively been performed by writing $p(y^1, y^2, y^3)$ in the $\dia_Y$-basis and applying \autoref{thm:change}, with the same result), and simplifying. 
\end{example}

By far the best known example of Appell polynomials is the case in which the random variables $X^i$ are Gaussian, see \cite{Janson_1997}, in which case the Appell polynomials are the Hermite polynomials. In this case, and only in this case, the Appell polynomials block-orthogonal between degrees, leading to the very useful Wiener chaos decomposition. In general, however, the orthogonal polynomials and the Appell polynomials associated to a measure are different.

\begin{example}[Poisson]\label{ex:pois}
Take $\mathcal X$ to contain a single random variable $X \sim \mathrm{Pois}(\lambda)$. Note that this is a discrete random variable and thus does not satisfy the requirement of \autoref{rem:wd}; the Appell polynomials are nevertheless still perfectly well defined. We have $\kappa_n[X] = \lambda$ for all $n \geq 1$ (in particular $X$ is not centred) and therefore, denoting $\mathcal P_m$ the set of partitions of the set with $m$ elements,
\begin{align*}
x^{\dia n} = \sum_{k = 0}^n {n \choose k} x^k \sum_{\pi \in \P_{n - k}} (-1)^{|\pi|} \ka[X]^\pi = \sum_{k = 0}^n {n \choose k} x^k \sum_{h = 0}^{n-k} (-\lambda)^{h}\Big\{ {n-k \atop h} \Big\} = \sum_{k = 0}^n {n \choose k} T_{n-k}(-\lambda)  x^k \text{ .}
\end{align*}
We have denoted $\{ {m \atop h} \}$ the Stirling numbers of the second kind, i.e.\ the number of ways of partitioning a set with $m$ elements into $h$ non-empty subsets (so $\{ {m \atop 0} \} = 0$ for $m > 0$ and $\{ {0 \atop 0} \} = 1$), and $T_m(x) = \sum_{h = 1}^m \{ {m \atop h} \}x^h$, $T_0 = 1$ the Touchard polynomials \cite[\S 2.4]{PT}. We compute the first few:
\begin{align*}
x^{\dia 0} &= 1\\
x^{\dia 1} &= x - \lambda  \\
x^{\dia 2} &= x^2 - 2\lambda x + \lambda^2 - \lambda \\
x^{\dia 3} &= x^3 - 3\lambda x^2 + 3(\lambda^2 - \lambda)x - \lambda^3 + 3\lambda^2 - \lambda \text{ .}
\end{align*}
Compare these with the Charlier polynomials, the sequence of orthogonal polynomials, normalised to be monic:
\begin{align*}
c^0_\lambda(x) &= 1 \\
c^1_\lambda(x) &= x - \lambda \\
c^2_\lambda(x) &= x^2 - (2\lambda +1) x + \lambda^2 \\
c^3_\lambda(x) &= x^3 - (3\lambda + 3) x^2 + (3\lambda^2  +3\lambda + 2) x - \lambda^3 \text{ .}
\end{align*}
We note that $c^n_\lambda(x) \neq x^{\dia n}$ for $n \geq 2$; this can also be seen from their generating functions, from which it can be seen that Charlier polynomials are a Sheffer sequence that is not Appell. This is just another manifestation of the fact that Appell polynomials are generally not orthogonal, while orthogonal polynomials are not generally preserved under the usual differentiation rule. See \cite{Anshelevich_2004} for a more complete comparison between Appell, orthogonal and Kailath-Segall polynomials. It is usually Charlier polynomials that are used in analogues for Malliavin calculus for jumps, e.g.\ \cite[Ch.\ 10]{NuaNua}, in which however the derivative operator does not correspond to usual differentiation and does not satisfy properties such as being a local operator; in \cite[Ch.\ 11]{NuaNua} the authors instead take the derivative operator to be the usual one, which enables the study of densities, but without dualising it. Appell polynomials instead preserve differentiation, at the price of lack orthogonality.
\end{example}

We now introduce diagram notation for dealing with identities involving Appell polynomials. This notation is a slight extension of that in \cite{Giraitis_1986}, and similar notation has been used prominently in neighbouring contexts, e.g.\ for Gaussian Wick products \cite{Janson_1997} or chaos expansions \cite{PT}. A \emph{partial diagram}, or simply \emph{diagram}, is a row of nodes, each representing an index $i \in \mathcal I$, together with a partition of a subset of nodes, which is denoted by connecting nodes that lie in the same set of the partition with an \emph{edge}. A diagram is called \emph{total} if all of its nodes lie on an edge---these are the diagrams already considered in \cite{Giraitis_1986}. Given finite multisets $I_1,\ldots, I_m \subset \mathcal I$ indexing the rows of the diagram, $I \coloneqq \bigsqcup_j I_j$, we will write $\mathcal D(I_1,\ldots,I_m)$ to denote the set of all diagrams over those indices, and $D = (\pi, K) \in \mathcal D(I_1,\ldots,I_m)$ where $\pi \in \mathcal P(I \setminus K)$ are the edges; we will also denote $\mathcal D$ the set of all diagrams without reference to a specific indexing. We will similarly use $\mathcal P$ for total diagrams, preserving the notation for partitions. We will call the set $K$ in the diagram $(\pi, K) = D$ the \emph{residual set}, and denote $\overline \pi \coloneqq \pi \sqcup \{K\} \in \mathcal P(I)$ the \emph{total partition} of $D$.

\begin{figure}[ht!]
\centering
\begin{tikzpicture}
	\dotrows{2,6,6,3,4}
	\draw (\dotcoord{1,1}) to (\dotcoord{2,2});
	\draw (\dotcoord{2,2}) to (\dotcoord{2,4});
	\draw (\dotcoord{2,4}) to (\dotcoord{3,4});
	\draw (\dotcoord{3,4}) to (\dotcoord{3,2});
	\draw (\dotcoord{3,6}) to (\dotcoord{4,3});
	\draw (\dotcoord{4,3}) to (\dotcoord{4,2});
	\draw (\dotcoord{4,1}) to (\dotcoord{5,1});
\end{tikzpicture}
\hspace{5em}
\begin{tikzpicture}
	\dotrows{6,3,3,3}
	\draw (\dotcoord{1,1}) to (\dotcoord{2,1});
	\draw (\dotcoord{1,3}) to (\dotcoord{1,5});
	\draw (\dotcoord{3,1}) to (\dotcoord{4,1});
	\draw (\dotcoord{3,2}) to (\dotcoord{4,3});
	\draw (\dotcoord{4,1}) to (\dotcoord{4,2});
\end{tikzpicture}
\hspace{5em}
\begin{tikzpicture}
	\dotrows{3,3,3,5,1}
	\draw (\dotcoord{1,1}) to (\dotcoord{2,2});
	\draw (\dotcoord{1,2}) to (\dotcoord{2,3});
	\draw (\dotcoord{2,3}) to (\dotcoord{3,1});
	\draw (\dotcoord{2,1}) to (\dotcoord{3,2});
	\draw (\dotcoord{1,3}) to[bend left = 30] (\dotcoord{3,3});
	\draw (\dotcoord{4,1}) to (\dotcoord{5,1});
	\draw (\dotcoord{4,1}) to (\dotcoord{4,2});
\end{tikzpicture}
\caption{Examples of diagrams. The first and third diagrams are non-flat, while the second has a flat edge in the first row. The first two diagrams are connected (the second thanks to the residual set), but not the last.}
\label{diag:ex}
\end{figure}

Since we are working in the commutative setting, the rows as well as the elements inside each row should be considered as unordered. An edge is called \emph{flat} if it lies within a single row, and a diagram is called \emph{non-flat} if none of its edges are flat (but the residual set of a non-flat diagram is allowed to lie in a single row). Given $\rho \in \mathcal P(I_1,\ldots,I_m)$ we introduce an equivalence relation $\sim_{\rho}$ on the rows $\{1,\ldots,m\}$, given by $i \sim_{\rho} j$ if there exist $H_1,\ldots,H_l \in \rho$ such that $H_1$ intersects row $i$, $H_l$ intersects row $j$, and for $h = 1,\ldots,l-1$, $H_k$ and $H_{k+1}$ intersect a common row. This is an equivalence relation whose equivalence classes we call \emph{connected components} of $(\pi, K) \in \mathcal D$ by taking $\rho = \overline \pi$, and we say a diagram is connected if it has a single connected component. We will use subscripts $\mathrm{NF},\mathrm{C}$ to refer respectively to the subset of non-flat and connected diagrams, combining these if necessary. We comment that the graphical notation is slightly lacking in the case in which the random variables $X$ are not centred, since in that case an edge can also be a singleton, indistinguishable from a single node.

\begin{remark}
These diagrams are similar to those used by \cite{PT}, but different in their specifics. Their notion of non-flatness is stronger than the one used here, i.e.\ edges may not have more than one node per row, and the notion of contraction is different. These differences are the result of different objectives: they do not study Wick products, rather they focus on the completely random measures of \cite{RW97} and their iterated Wiener-type integrals. These measures exclude processes with correlated increments, which we will be able to treat.
\end{remark}

We now associate polynomials to diagrams. Given $D = (\pi, K) \in \mathcal D$, we set
\begin{equation}\label{eq:eval}x^D \coloneqq x^K \ka[X]^\pi, \qquad x^{\dia D} \coloneqq x^{\dia K} \ka[X]^\pi \text{ .}
\end{equation}
This definition coincides with the one in \cite{Giraitis_1986} for total diagrams, in which case $X^D = X^{\dia D} = \ka[X]^D \in \mathbb R$. We recall their main result on combinatorics of diagrams.
\begin{proposition}[{\cite[Theorem 4]{Giraitis_1986}}]\label{prop:EkW}
	\begin{alignat*}{2}
		\E[X^{I_1} \cdots X^{I_m}]
		&={} \sum_{D \in \mathcal P(I_1,\ldots, I_m)} \ka[X]^D
		&\qquad
		\E[X^{\dia I_1} \cdots X^{\dia I_m}]
		&={} \sum_{D \in \mathcal P_\mathrm{NF}(I_1,\ldots, I_m)} \ka[X]^D
		\\[0.5em]
		\ka[X^{I_1}, \ldots ,X^{I_m}]
		&={} \sum_{D \in \mathcal P_\mathrm{C}(I_1,\ldots, I_m)} \ka[X]^D
		&\qquad
		\ka[X^{\dia I_1}, \ldots ,X^{\dia I_m}]
		&={} \sum_{D \in \mathcal P_{\mathrm{NF},\mathrm{C}}(I_1,\ldots, I_m)} \ka[X]^D
	\end{alignat*}
\end{proposition}

The reason to introduce partial diagrams is to prove product formulae. The following formula generalises the well-known one for Hermite polynomials \cite[Exercise 1.7.1]{NP} (often stated more generally for Wiener integrals \cite[Theorem 2.7.10]{NP}), both in the direction of more general measures, and for more than two factors. While we were not able to find it in the literature, many similar results exist, e.g.\ \cite[Proposition 17]{Anshelevich_2004} in free probability. Note how it implies the second identity in \autoref{prop:EkW} by taking expectations and using the zero-mean property of Appell polynomials.
\begin{theorem}[Product formula]\label{prop:prod}
\[
x^{\dia I_1} \cdots x^{\dia I_m} = \sum_{D \in \mathcal D_\mathrm{NF}(I_1,\ldots,I_m)} x^{\dia D}
\]
\end{theorem}
\begin{proof}
We prove this by induction on $m$. The cases $m = 0, 1$ are trivial. We start with $m = 2$ and call $I$ and $J$ the two rows. We compute, using \autoref{prop:wick} (see \autoref{fig:partition-drawing})

\begin{equation*}
\begin{split}
&x^{\diamond I}x^{\diamond J}\\
={}&\sum_{\substack{L \subset I \\ M \subset J}} x^{L\sqcup M} \sum_{\substack{\pi \in \P(I \setminus L) \\ \rho \in \P(J \setminus M)}} (-1)^{|\pi| + |\rho|} \ka[X]^{\pi \sqcup \rho} \\
={}&\sum_{N \subset I \sqcup J} \sum_{K \subset N} x^{\dia K} \sum_{\sigma \in \P(N \setminus K)} \ka[X]^\sigma \sum_{\substack{\pi \in \P(I \setminus N) \\ \rho \in \P(J \setminus N) \\ \tau \coloneqq \pi \sqcup \rho}}  (-1)^{|\tau|} \ka[X]^{\tau}
\end{split}
\end{equation*}

\begin{figure}[h!]
\centering 
\begin{tikzpicture}
\dotrows{20,23}
\draw[Cerulean] (\dotcoord{1,7}) to (\dotcoord{1,9});
\draw[Cerulean] (\dotcoord{1,10}) to (\dotcoord{1,11});
\draw[Cerulean] (\dotcoord{1,12}) to (\dotcoord{1,13});
\draw[Cerulean] (\dotcoord{2,6}) to (\dotcoord{2,10});
\draw[Cerulean] (\dotcoord{2,11}) to (\dotcoord{2,14});
\draw[OrangeRed] (\dotcoord{1,14}) to (\dotcoord{1,17});
\draw[OrangeRed] (\dotcoord{2,15}) to (\dotcoord{2,16});
\draw[OrangeRed] (\dotcoord{1,18}) to (\dotcoord{2,18});
\draw[OrangeRed] (\dotcoord{2,17}) to (\dotcoord{2,19});
\draw[OrangeRed] (\dotcoord{1,19}) to (\dotcoord{2,20});
\draw[OrangeRed] (\dotcoord{1,20}) to (\dotcoord{2,21});
\draw[OrangeRed] (\dotcoord{2,21}) to (\dotcoord{2,23});
\end{tikzpicture}
\caption{Drawing representing the above partitions. 
The $\tau$ partition (in blue) is flat, while generally only part of the sets in $\sigma$ (in red) are flat. In the next part of the proof, we will group $\tau$ with the sets in $\sigma$ which are flat.}
\label{fig:partition-drawing}
\end{figure}
We use the subscript $\mathrm{F}$ to mean \say{flat} and write $\mathcal P$ without necessarily specifying the set of which we are taking the partition (it is necessarily just the union of its sets). Writing $\P(N \setminus K) \ni \sigma = \alpha \sqcup \beta$ with $\alpha \in \P_\mathrm{F}$, $\beta \in \P_\mathrm{NF}$, we continue

\begin{equation*}
\begin{split}
={}&\sum_{N \subset I \sqcup J} \sum_{K \subset N} x^{\dia K} \sum_{\substack{\alpha \in \P_\mathrm{F} \\ \beta \in \P_\mathrm{NF} \\ \alpha \sqcup \beta \in \P(N \setminus K)}} \ka[X]^{\alpha \sqcup \beta} \sum_{\tau \in \P_\mathrm{F}((I \sqcup J) \setminus N)}  (-1)^{|\tau|} \ka[X]^{\tau} \\
={}&\sum_{K \subset I \sqcup J} x^{\dia K} \sum_{\substack{\gamma \in \P_\mathrm{F} \\ \beta \in \P_\mathrm{NF} \\ \gamma \sqcup \beta \in \P((I \sqcup J) \setminus K)}} \ka[X]^{\gamma \sqcup \beta} \sum_{\gamma = \alpha \sqcup \tau}  (-1)^{|\tau|}\text{ .}
\end{split}
\end{equation*}
The last sum is over all ways of splitting the partition $\gamma$ into two partitions $\alpha$ and $\tau$, and if $|\gamma| > 0$ we have
\[
\sum_{\gamma = \alpha \sqcup \tau}  (-1)^{|\tau|} = 2^{|\gamma| - 1}  - 2^{|\gamma| - 1} = 0\text{ .}
\]
The only case that contributes to the sum is therefore when $\gamma$ is the empty partition, in which case $\beta \in \P_\mathrm{NF}((I \sqcup J) \setminus K)$, reducing the sum to 
\[
\sum_{K \subset I \sqcup J}\sum_{\pi \in \P_\mathrm{NF}((I \sqcup J) \setminus K)} \ka[X]^\pi x^{\diamond K} =\sum_{D \in \mathcal D_\mathrm{NF}(I,J)}x^{\dia D} \text{ .}
\]
For the induction step, we apply the case above and obtain
\[
x^{\dia I_1} \cdots x^{\dia I_{m+1}} = \sum_{(\pi, K) \in \mathcal D_\mathrm{NF}(I_1,\ldots,I_m)} \ka[X]^\pi x^{\dia K} x^{\dia I_{m+1}} = \sum_{(\pi, K) \in \mathcal D_\mathrm{NF}(I_1,\ldots,I_m)} \ka[X]^\pi \sum_{D \in \mathcal D_\mathrm{NF}(K,I_{m+1})} x^{\dia D}
\]
The conclusion now follows by observing that any non-flat diagram on rows $I_1,\ldots,I_{m+1}$ is uniquely obtained by adding non-flat edges that join free nodes in a non-flat diagram on rows $I_1,\ldots,I_m$ and the row $I_{m+1}$.
\end{proof}

\begin{theorem}[Reverse product formula]\label{thm:reverse}
\[
x^{I_1} \dia \cdots \dia x^{I_m} = \sum_{(\pi,K) \in \mathcal D_\mathrm{NF}(I_1,\ldots,I_m)} (-1)^{|\pi|} x^{(\pi,K)}
\]
\end{theorem}
\begin{proof}
The proof of the first statement is completely analogous to that of \autoref{prop:prod}, using \eqref{eq:inverse} and \eqref{eq:wick} in reverse order to prove the case $m=2$, and the general case follows as before.
\end{proof}

We now return to the lack of associativity discussed in \autoref{rem:ass} and qualify it more precisely.

\begin{theorem}[Change of chaos]\label{thm:change}
Let $Y^k = X^{\dia_X I_k}$ for $k = 1,\ldots,m$.
\[
y^1 \dia_Y \cdots \dia_Y y^m \big|_{y^j = x^{\dia I_{\scaleto{j}{3.5pt}}}} = \sum_{\substack{(\pi, K) \in \mathcal D_{\mathrm{NF}, \mathrm{C}}(I_1,\ldots,I_m) \\ |K| > 0}} x^{\dia_X (\pi, K)}
\]
\end{theorem}
\begin{proof}
Observe that for any $J_1,\ldots,J_n \subset \mathcal I$ and any $D \in \mathcal D(J_1,\ldots,J_n)$ exactly one of its connected components will contain the residual set if this is non-empty, while the rest will belong to $\mathcal P$ (i.e.\ be a full diagram). We write $D_1,\ldots,D_h, E \in \mathrm{cc}(D)$ to mean that $D_1,\ldots,D_h, E$ are the connected components of $D$ with $D_1,\ldots,D_h \in \mathcal P$ and $E \in \mathcal D$ (possibly empty if $h > 0$). Moreover, we write $[m]$ for the set with $m$ elements which we view as rows, each containing a single element: therefore $(\rho,R) \in \mathcal D([m])$ simply denotes a partition of $[m]$ with a distinguished set $R$. Using the definition of Wick product (3.\ in \autoref{prop:wick}) \autoref{prop:prod} and the fourth identity in \autoref{prop:EkW} we compute (see \autoref{fig:coc})
\begin{align*}
&y^1 \dia_Y \cdots \dia_Y y^m \big|_{y^j = x^{\dia I_{\scaleto{j}{3.5pt}}}} \\&= \sum_{(\rho,R) \in \mathcal D([m])} (-1)^{|\rho|} \prod_{r \in R}x^{\dia I_r} \prod_{P \in \rho} \ka[X^{\dia I_p} : p \in P] \\
&= \sum_{(\rho,R) \in \mathcal D([m])} (-1)^{|\rho|} \Big(\sum_{D \in \mathcal D_\mathrm{NF}(I_r : r \in R)} x^{\dia D} \Big) \prod_{P \in \rho} \sum_{C \in \mathcal P_{\mathrm{NF},\mathrm{C}}(I_p : p \in P)} \ka[X]^C \\
&= \sum_{(\rho,R) \in \mathcal D([m])} (-1)^{|\rho|} \Big(\sum_{\substack{D \in \mathcal D_\mathrm{NF}(I_r : r \in R) \\ D_1,\ldots,D_s,E \in \mathrm{cc}(D)}} \prod_{k = 1}^s\ka[D_k] \cdot x^{\dia E} \Big) \prod_{P \in \rho} \sum_{C \in \mathcal P_{\mathrm{NF},\mathrm{C}}(I_p : p \in P)} \ka[X]^C \\
&= \sum_{\substack{F \in \mathcal D_\mathrm{NF}(I_1,\ldots,I_m) \\ G_1,\ldots,G_q,E \in \mathrm{cc}(F)}} x^{\dia E}  \prod_{k = 1}^q \ka[X]^{G_k}\sum_{[q] = A \sqcup B} (-1)^{|A|} \text{ .}
\end{align*}
As before, the last sum cancels for any $q > 0$, leaving the desired expression.
\end{proof}
\begin{figure}[h!] \centering \begin{tikzpicture}[rotate = 90]
\dotrows{2,4,4,4,4,3,4,4,3,4,3} \draw[dashed, line width = 0.5pt] ($ (p-2-1)!0.5!(p-3-1) + (-3mm,0) $) -- ($ (p-2-4)!0.5!(p-3-4) + (3mm,0) $); \draw[line width = 0.5pt] ($ (p-4-1)!0.5!(p-5-1) + (-3mm,0) $) -- ($ (p-4-4)!0.5!(p-5-4) + (3mm,0) $); \draw[dashed, line width = 0.5pt] ($ (p-7-1)!0.5!(p-8-1) + (-3mm,0) $) -- ($ (p-7-4)!0.5!(p-8-4) + (3mm,0) $); \draw[OrangeRed] (\dotcoord{11,1}) to (\dotcoord{11,2}); \draw[OrangeRed] (\dotcoord{11,1}) to (\dotcoord{10,1}); 
\draw[OrangeRed] (\dotcoord{9,1}) to (\dotcoord{9,2}); \draw[OrangeRed] (\dotcoord{9,2}) to (\dotcoord{8,2}); \draw[OrangeRed] (\dotcoord{8,1}) to (\dotcoord{8,4}); 
\draw[OrangeRed] (\dotcoord{5,1}) to (\dotcoord{7,1}); 
\draw[OrangeRed] (\dotcoord{7,1}) to (\dotcoord{7,4});
\draw[OrangeRed] (\dotcoord{5,4}) to (\dotcoord{5,2});
\draw[OrangeRed] (\dotcoord{5,2}) to (\dotcoord{6,2});
\draw[OrangeRed] (\dotcoord{6,2}) to (\dotcoord{6,3});
\draw[Cerulean] (\dotcoord{1,2}) to (\dotcoord{2,2});
\draw[Cerulean] (\dotcoord{1,1}) to (\dotcoord{2,1});
\draw[Cerulean] (\dotcoord{2,2}) to (\dotcoord{2,4});
\draw[Cerulean] (\dotcoord{3,1}) to (\dotcoord{4,2});
\draw[Cerulean] (\dotcoord{4,2}) to (\dotcoord{4,4});
\draw[Cerulean] (\dotcoord{4,1}) to (\dotcoord{3,2});
\draw[Cerulean] (\dotcoord{3,2}) to (\dotcoord{3,4});
\end{tikzpicture}
\caption{The diagram $F$, rotated by 90$^\circ$ counterclockwise to improve the layout, so that rows have become columns. The dividing lines represent connected components, with the blue terms those coming from $\rho$ and the red terms coming from $R$. The proof consists of considering together connected components without a residual set, in the picture the first three starting from the right. In the proof, edges in the diagram with a single node per row $(\rho, R) \in \mathcal D([m])$ correspond to connected components in $F$.}
\label{fig:coc}
\end{figure}

\section{Series in Appell polynomials}\label{sec:series}

In this section we take the indexing set $\mathcal I = [d] = \{1,\ldots,d\}$ so that $X$ is an $\mathbb R^d$-valued random variable, whose law we denote $\mu$. The factors of our Wick product are entire functions $f \colon \mathbb R^d \to \mathbb R$ which we implicitly extend to $\mathbb C^d$. While for purely combinatorial/diagrammatic arguments we find that it is more natural to work with multisets (since it is often irrelevant how many repetitions of an index there are), when working with functions in $d$ variables it is better to work with multiindices, i.e.\ arrays of $d$ non-negative integers, each entry of which counts the multiplicity of the corresponding index. There is a one-to-one correspondence between submultisets of $[d]$ and multiindices with $d$ entries. We will therefore continue using the same notation for both, i.e.\ letter $I, J, K, \ldots$ can interchangeably be viewed as a submultisets of $[d]$ or tuple with $d$ entries. Interpreted as the latter, we consider the multiindex factorial
\[
K! \coloneqq K_1! \cdots K_d!, \qquad K = (K_1,\ldots,K_d).
\]
With this notation, it holds that for $J \subset I$ (in which the inclusion takes multiplicity into account)
\[
\partial_J x^I = \frac{I!}{(I \setminus J)!} x^{I \setminus J} = J! \binom{I}{J} x^{I \setminus J},
\]
The differentiation rule for $x^{\dia I}$ is the same. Here multiset difference corresponds to entrywise subtraction in the multiindex; similarly, disjoint union of multisets corresponds to addition of multiindices. Notice that while the correspondence between multisets and multiindices is one-to-one, there are $\binom{I}{J}$ distinct ways in which to fit the multiset $J$ inside of $I$.

We will consider the moment-generating function of $\mu$ as function of several complex variables $\theta \in \mathbb C^d$,
\begin{equation}
\M(\theta) = \E \exp(\theta \cdot X) = \int_{\Rd} \exp(\theta \cdot x) \mu(\dif x).
\end{equation}
This is a holomorphic function as long as it is finite. More precisely, let $R = (R_1,\ldots,R_d) \in  (0,\infty]^d$ be a polyradius with $\E e^{q \cdot |X|} < \infty$ for all $q < R$---here and below we write $q < R$ to mean $q_i < R_i$ for all $i = 1,\ldots, d$, and use similar entrywise notation (e.g.\ $|X|$ will sometimes mean $(|X^1|, \ldots, |X^d|)$ as can be deduced from context). Then $\M$ is holomorphic on the polydisc 
\begin{equation}\label{eq:polydisc}
D_R \coloneqq \{|\theta| < R\} = \{\theta \in \mathbb C^d : |\theta_i| < R_i, \ i = 1,\ldots, d\}.
\end{equation}
In fact, since $|e^{\theta \cdot X}| = e^{\Re \theta \cdot X}$, $\M$ is actually holomorphic on the tube domain $\{\Re\theta < R\}$, but it is more natural to work with polydiscs since that is where $\M$ can be expanded in a power series around the origin.

For an entire function $f$, its (multivariate) \emph{Borel transform} \cite{boasEntire} is defined by the Laurent series
\[
\EuScript{B}f(z) = \sum_{K} \frac{\partial_K f(0)}{z^{K+1}}
\]
where here and below $K + 1$ is the multiindex in which each entry in $K$ is increased by $1$, whenever the series is convergent, and we sum over all multiindices $K$, including the zero multiindex corresponding to the empty multiset (unless otherwise specified). Then the definition of Appell polynomials in terms of generating functions \autoref{prop:wick} can be written as
\begin{equation}\label{eq:wickGenerating}
	\frac{\exp({\theta \cdot z})}{\EuScript{M}(\theta)} = \sum_K \frac{z^{\dia K}}{K!}\theta^K.
\end{equation}

The first part of the next theorem is the $d$-variable analogue of a special case of the results in \cite{mar36} (see \cite[Theorem 9.2]{BB}). Since the extension to several complex variables is a more or less standard consequence of iterated single variable Cauchy estimates, we have moved the proof to \autoref{app:proof}, which also contains an important estimate on the growth of derivatives of entire functions of exponential type. In the second part of the theorem, we prove, using similar techniques, that it is possible to define the Wick product of two entire functions with small enough exponential types compared to the radius of analyticity of $\M$ and $1/\M$.

\begin{theorem}\label{thm:martin}
	Let $R \in (0, \infty]^d$ be a polyradius such that $\M$ is defined and non-zero on the open polydisc $D_R$. Let $f \colon \mathbb C^d \to \mathbb C$ be entire of exponential type strictly bounded by $R$, i.e.\ there exists some $r < R$ and $C$ such that
	\begin{equation}\label{eq:exp_growth}
	|f(z)| \leq C \exp( r \cdot |z|).
	\end{equation}
	Then $\EuScript{B}f$ is defined and holomorphic on 
	\(
	E_r \coloneqq \{\theta \in \mathbb C^d : |\theta_k| > r_k \}
	\)
	and $f$ can be expanded in a series of Appell polynomials as follows:
	\begin{align*}
		f(z) &= \sum_K L_K(f) \frac{z^{\dia K}}{K!}, \qquad \text{with} \\
		L_K(f) &= \frac{1}{(2\pi i)^d} \int_{|\theta_1| = \rho_1} \cdots \int_{|\theta_d| = \rho_d} \theta^K \EuScript{M}(\theta) \EuScript{B}f(\theta) \dif \theta_d \cdots \dif \theta_1
	\end{align*}
	for any $\rho$ with $r < \rho < R$ (with value of $L_K(f)$ independent of $\rho$), and the series converges absolutely, both locally uniformly on compacts and in $L^1(\mu)$.
	
	Moreover, let $g$ be another entire function satisfying the same hypotheses above with exponential type bounded by $r_g$, let $r_f = r$, assume that $r_f + r_g < R$, and let
	\[
	f(x) = \sum_I a_I \frac{x^{\dia I}}{I!}, \qquad g(x) = \sum_J b_J \frac{x^{\dia J}}{J!}.
	\]
	Then the \emph{Wick product} of $f$ and $g$
	\[
	(f \dia g)(x) \coloneqq \sum_{I, J} a_I b_J \frac{x^{\dia I \sqcup J}}{I! J!}
	\]
	is defined as an absolutely convergent series, both locally uniformly and in $L^1$.
\end{theorem}

The non-vanishing condition on the moment-generating function is less standard to check than exponential integrability. Of course, by continuity, it always holds in a small enough neighbourhood of $0$. For certain classes of distributions, however, it holds in the full domain of analyticity (see also \autoref{prop_suff_cond_second_chaos} below).

\begin{proposition}\label{prop_suff_cond}
Let $\mu$ be an infinitely divisible distribution on $\mathbb{R}^d$ such that $\EuScript{M}(\theta) < \infty$ for $\Re \theta < R$. Then, for $\Re \theta < R$, it also holds that $\EuScript{M}(\theta) \neq 0$.
\end{proposition}

\begin{proof}
By infinite divisibility, there exists a L\'evy process $(X_t)_{t\geq0}$
such that $X_1\sim\mu$; denoting by $\mu_t$ the law of $X_t$,
$(\mu_t)_{t\geq0}$ is a weakly continuous convolution semigroup with
$\mu_0=\delta_0$ and $\mu_1=\mu$ \cite[Theorem~7.10]{Sato13}. Set
\[
\EuScript M_t(\theta)
\coloneqq \int_{\Rd} e^{\theta\cdot x}\,\mu_t(\dif x).
\]
Fix $\theta$ with $\Re\theta<R$, and choose $p>1$ sufficiently close to $1$ so that $p\Re\theta<R$. Since $\mu_{1/n}\Rightarrow\delta_0$, we have
$X_{1/n}\to0$ in probability, and hence $e^{\theta\cdot X_{1/n}}\to 1$ in probability. Moreover, by the semigroup property,
\[
\E\bigl[|e^{\theta\cdot X_{1/n}}|^p\bigr]
=\EuScript M_{1/n}(p\Re\theta)
=\EuScript M_1(p\Re\theta)^{1/n},
\]
which is uniformly bounded in $n$. Hence
$\{e^{\theta\cdot X_{1/n}}\}_{n\geq1}$ is uniformly integrable, and therefore
\(
\EuScript M_{1/n}(\theta) =\E[e^{\theta\cdot X_{1/n}}] \to 1
\).
Thus, for all sufficiently large $n$, $\EuScript M_{1/n}(\theta)\neq0$, and
the semigroup property yields
\[
\EuScript M(\theta)
=\EuScript M_1(\theta)
=\EuScript M_{1/n}(\theta)^n
\neq0.
\qedhere
\]
\end{proof}

The next lemma will be used to bound cumulants jointly in their order and taking into account the \say{smallness} of a random variable, which in the next section will be given by the increments of a stochastic process inside a Riemann-Stieltjes sum. We write $\lesssim_a$ to mean boundedness up to a multiplicative constant which depends on $a$.

\begin{lemma}\label{lem:cumulant_estimate}
	Let $(X,Y) = (X^1,\ldots,X^d,Y)$ be an $\bR^{d+1}$-valued random variable whose law $\mu$ satisfies the hypotheses of \autoref{thm:martin} with $(d+1)$-dimensional polyradius $R = (R_X, R_Y)$ (with $R_X \in (0, +\infty]^d$). Assume that for every $\sigma < R_X$
	\[
	\sup_{\xi \in D_\sigma}\big|\partial_\eta \EuScript{K}(\xi,\eta)|_{\eta = 0} \big| \lesssim_\sigma \varepsilon
	\]
	for some $\varepsilon > 0$, where $\EuScript{K}(\xi,\eta) = \log \M(\xi,\eta)$ (with $\M(\xi,\eta) = \E \exp({\xi \cdot X + \eta Y})$) is the cumulant-generating function of $(X,Y)$, considered for $\xi \in \mathbb C^d$, $\eta \in \mathbb C$. Then for every $\sigma < R_X$
	\[
	|\kappa[X^I, Y]| \lesssim_\sigma I! \sigma^{-I} \varepsilon
	\]
	where $\kappa[X^I, Y] \coloneqq \kappa[X^{i_1},\ldots, X^{i_n},Y]$ for a submultiset $I = \{i_1,\ldots,i_n\}$ of $[d]$.
\end{lemma}

\begin{proof}
	Since $\M$ does not vanish on $D_R$, after choosing a branch of the logarithm, the cumulant-generating function $\EuScript{K}(\xi,\eta) = \log \M(\xi,\eta)$ is defined and holomorphic on $D_R$ and setting $F(\xi) \coloneqq \partial_\eta \EuScript{K}(\xi,\eta)\big|_{\eta = 0}$, we have $\kappa[X^I, Y] = \partial_I F(0)$. Then a Cauchy estimate (similar to the ones performed in \autoref{app:proof}) yields
	\[
	|\kappa[X^I, Y]| = |\partial_I F(0)| \leq I! \sigma^{-I} \sup_{|\xi| \leq \sigma} |F(\xi)| \lesssim_\sigma I! \sigma^{-I}\varepsilon. \qedhere
	\]
\end{proof}

The next proposition will be the main ingredient when taking Riemann-Stieltjes sums of Wick products involving stochastic processes. It proves a conversion formula between the ordinary and Wick products of an entire function $f$ with a linear coordinate function $y$. Notice how, passing to the expectation in our conversion formula, and using that $\E[f(X) \dia Y] = \E[f(X)]\E[Y]$ (corresponding to the summand $I = \varnothing$) recovers the integration-by-parts formula of Barbour \cite[Corollary 1]{barbourAsymp}
\[
\E[f(X)Y] = \sum_{I} \frac{\E \partial_I f(X)}{I!} \kappa[X^I, Y],
\]
which we prove in the multivariate case and without assuming boundedness of derivatives, only exponential type (although the rate of growth of derivatives assumed therein is consistent with that obtained in \autoref{lem:cumulant_estimate}).

\begin{proposition}[Wick product with small linear random variable]\label{thm:wickfxy}
	Let $(X,Y)$ be as in \autoref{lem:cumulant_estimate} with polyradius $R = (R_X, R_Y)$ and bound $\varepsilon$, $f$ be an entire function in $d$ variables as in \autoref{thm:martin} with polyradius $r < R_X$, let $y$ denote the $(d+1)$-th coordinate function in $\mathbb R^{d+1}$. Then the Wick product $f \dia y$ is defined according to \autoref{thm:martin}, and 
	\begin{equation}\label{eq:reverse_series_lin}
	f(x) \dia y = f(x)y - \sum_{I \neq \varnothing} \frac{\partial_I f(x)}{I!} \kappa[X^I, Y]
	\end{equation}
	where the series converges absolutely, both locally uniformly and in $L^1$; more specifically, for any $r < \rho < \sigma < R_X$ there exists a constant $C_{\rho,\sigma}$ with
	\[
	\bigg\|\frac{\partial_I f(X)}{I!} \kappa[X^I, Y]\bigg\|_{L^1}  \leq C_{\rho,\sigma} \bigg(\frac{\rho}{\sigma}\bigg)^I \varepsilon.
	\]
\end{proposition}
\begin{proof}
	We first consider the reverse product formula for two factors, in multiindex notation. Take $p(x) = x^A/A!$ and $q(x) = x^B/B!$. Then by \autoref{thm:reverse} 
	\begin{align*}
		(p \dia q)(x) &= p(x)q(x) + \frac{1}{A!B!}\sum_{K \subsetneq A \sqcup B} x^K \sum_{\pi \in \mathcal P_\mathrm{NF}(A \setminus K, B \setminus K)} (-1)^{|\pi|} \kappa[X]^\pi \\
		&= p(x)q(x) + \frac{1}{A!B!}\sum_{\substack{I \leq A \\ J \leq B \\ I, J \neq \varnothing}} \binom{A}{I} \binom{B}{J} x^{A - I} x^{B - J} \sum_{\pi \in \mathcal P_\mathrm{NF}(I, J)} (-1)^{|\pi|} \kappa[X]^\pi \\
		&= p(x)q(x) + \sum_{I, J \neq \varnothing} \frac{\partial_I p(x) \partial_J q(x)}{I!J!} \sum_{\pi \in \mathcal P_\mathrm{NF}(I, J)} (-1)^{|\pi|} \kappa[X]^\pi .
	\end{align*}
	By linearity, this extends to arbitrary polynomials $p,q$.
	
	Now introduce the $(d+1)$-th variable $y$ and take $g(y) = y$. The formula becomes the one in the statement: every set in the non-flat partition must contain the single variable $y$, and therefore can only have one set (see \autoref{fig:Ij}).
	\begin{figure}[h!]
		\centering 
		\begin{tikzpicture}
			\dotrows{5,1}
			\node[anchor=east] at (-4mm,0) {$I$};
			\node[anchor=east] at (-4mm,-6mm) {$j$};
			\draw[Cerulean] (\dotcoord{1,1}) to (\dotcoord{2,1});
			\draw[Cerulean] (\dotcoord{1,1}) to (\dotcoord{1,3});
		\end{tikzpicture}
		\caption{A diagram with two rows and in which the second row only has one node: there is room for at most a single non-flat edge.}\label{fig:Ij}
	\end{figure}
	By \autoref{thm:martin} $f$ can be expanded in Appell polynomials
	\[
	f(x) = \sum_J a_J x^{\dia J}/J! = \lim_{N \to \infty} \underbrace{\sum_{|J| \leq N} a_J x^{\dia J}/J!}_{\displaystyle\eqqcolon p_N}
	\]
	For each $N \in \mathbb N$ we then have
	\[
	p_N(x) \dia y = p_N(x)y - \sum_{I \neq \varnothing} \frac{\partial_I p_N(x)}{I!} \kappa[X^I, Y]
	\]
	As $N \to \infty$, $p_N \dia y \to f \dia y$ by the definition in \autoref{thm:martin} given that the function $g(y) = y$ has arbitrarily small exponential type. Choosing conjugate exponents $p,q>1$ with $p\sigma<R_X$, exponential integrability gives $Y\in L^q$, while the same Appell-series estimates give $p_N(X)\to f(X)$ in $L^p$, implying $p_N(X)Y\to f(X)Y$ in $L^1$ by H\"older. An additional argument is needed to pass to the limit in the sum over $I$. We have
	\[
	\partial_I p_N(x) = \sum_{|J| \leq N} a_J \partial_I \frac{x^{\dia J}}{J!} = \sum_{|J| \leq N} a_J \frac{x^{\dia (J-I)}}{(J-I)!} = \sum_{|K + I| \leq N} a_{K+I} \frac{x^{\dia K}}{K!}.
	\]
	From the proof of \autoref{thm:martin}, for any $r < \rho <R_X$, $|a_J| \lesssim \rho^J$
	\[
	|\partial_I p_N(x)| \lesssim \rho^I \sum_K \rho^K \bigg| \frac{x^{\dia K}}{K!} \bigg| \lesssim_{x,\rho} \rho^I
	\]
	since $\sum_K \rho^K x^{\dia K}/K! = \exp({x \cdot \rho})/\M(\rho)$ and the series converges absolutely. Combining with \autoref{lem:cumulant_estimate} we have, for any $\rho < \sigma < R$, the bound
	\[
	\bigg| \frac{\partial_I p_N(x)}{I!} \kappa[X^I, Y] \bigg| \lesssim_{x,\sigma,\rho} \bigg(\frac{\rho}{\sigma}\bigg)^I \varepsilon
	\]
	which is uniform in $N$ and summable in $I$. Dominated convergence for series justifies the first identity in
	\[
	\lim_{N \to \infty} \sum_{I \neq \varnothing} \frac{\partial_I p_N(x)}{I!} \kappa[X^I, Y] = \sum_{I \neq \varnothing} \frac{\lim_{N \to \infty}\partial_I p_N(x)}{I!} \kappa[X^I, Y] = \sum_{I \neq \varnothing} \frac{\partial_I f(x)}{I!} \kappa[X^I, Y],
	\]
	and in the second identity we have used that Appell expansions may be differentiated termwise (a consequence of absolute uniform-on-compacts convergence). Finally, the $L^1$ bound follows by a similar application of \autoref{lem:cumulant_estimate} in conjunction with the more explicit estimate \autoref{lem:growth} for the derivatives of $f$, and integrating the exponential.
\end{proof}

We observe that a couple of identities that are well known in the Gaussian case \cite[Corollary 2.7.8, Proposition 1.3.8]{NP} also hold in our non-Gaussian setting.

\begin{corollary}\label{cor_stroock}Let $f = \sum_I a_I x^{\dia I}/I!$ be as in \autoref{thm:martin}.
	\begin{itemize}
		\item The coefficients of the Appell expansion of $f$ are given by averaging the derivatives of $f$:
		\[
		a_K = \mathbb E \partial_K f(X).
		\]
		\item The Heisenberg commutation relationship  $
		[\partial_j,-\dia x^j]=\mathrm{id}$ holds.
	\end{itemize}
\end{corollary}
\begin{proof}
	Differentiating under the summation sign by uniform convergence on compacts, differentiating Appell polynomials and reindexing, we obtain
	\[
	\partial_K \sum_I a_I  \frac{x^{\dia I}}{I!}= \sum_{I \supset K} a_I  \frac{1}{(I \setminus K)!} x^{\dia I \setminus K}=\sum_J a_{J \sqcup K}  \frac{1}{J!} x^{\dia J}.
	\]
	By \autoref{lem:growth} the differentiated series converges in $L^1(\mu)$. Therefore, taking expectation under the summation sign and by the zero-mean property, the only term with $J=\varnothing$ survives, obtaining the statement.
	
	For the second claim, by \autoref{thm:wickfxy}
	\[
	f(x)\dia x^j
	=
	f(x)x^j
	-
	\sum_{I \neq \varnothing}
	\kappa[X^{I\sqcup j}]
	\frac{\partial_I f}{I!}.
	\]
	The conclusion follows from the identities
	\[
	\left[
	\partial_j,
	\sum_{I \neq \varnothing}
	\kappa[X^{I\sqcup j}]
	\frac{\partial_I}{I!}
	\right]
	=0\,, \qquad[\partial_j, - \cdot x^j]=\mathrm{id}. \qedhere
	\]
\end{proof}

\section{Wick integrals of entire one-forms in the Young regime}\label{sec:ints}

Let $X = (X^1,\ldots,X^d) \colon [0,T] \times \Omega \to \mathbb R^d$ be a stochastic process on $[0,T]$. We denote $\triangle_T \coloneqq\{(s,t) \in [0,T] : s < t\}$. Throughout this section we will denote $X_{s,t} = X_t - X_s$ and similar. We assume that $X$ satisfies the following

\begin{conditions}\label{conditions}\ \\[-3ex]
	\begin{enumerate}
		\item\label{item:1} $X$ is a.s.\ continuous and of bounded $p$-variation for some $p<2$;
		\item\label{item:2} Denote the cumulant-generating function of $(X_s, X_t)$ by
		\[
		\K_{s,t}(z, w) \coloneqq \log \M_{s,t}(z,w),\qquad \text{with}\quad\M_{s,t}(z,w) \coloneqq \E e^{z \cdot X_s + w \cdot X_t}.
		\]
		We assume that there exists a polyradius $R > 0$ such that $\K$ is holomorphic (i.e.\ $\M_{s,t}$ defined and non-vanishing) in $D_R \times D_R$, and uniformly $C^1$ w.r.t.\ the time variables in the sense that $\partial_t \K_{s,t}$ exists for $t \in [0,T]$, is continuous in all four variables and
		\[
		\sup_{\substack{s,t \in [0,T] \\ z,w \in D_R}} |\partial_t \K_{s,t}(z,w)| < \infty;
		\]
		\item\label{item:3} For any $t \in (0,T]$, there exists $U \subset \Rd$ open with $\mathbb P[X_t \in U] > 0$ such that $X_t$ has a density on $U$.
	\end{enumerate}
\end{conditions}

\begin{definition}\label{def:wickint}
	Under the conditions above, let $0 \leq s < t \leq T$ and $f = (f_1,\ldots, f_d) \colon \Rd \to \bR^d$ be such that, for all $k$, $f_k$ satisfies the hypotheses in \eqref{eq:exp_growth} w.r.t.\ the polyradius $R$ in \autoref{conditions}. We define the \emph{Wick integral} as the limit, a.s.\ and/or in $L^1$
	\[
	\int_s^t f_k(X_r) \dia \dif X^k_r \coloneqq \lim_{m \to \infty}\sum_{[u,v] \in \pi_m} f_k(X_u) \dia X^k_{u,v} \text{ ,}
	\]
	taken along an arbitrary sequence of partitions $\pi_m$ on $[s,t]$ with vanishing mesh size. Here and below we use the Einstein convention on up-down indices, here over $k$.
\end{definition}

\autoref{conditions}\autoref{item:3} is simply in place to guarantee that the Wick products $f_k(X_u) \dia X^k_{u,v}$ are well defined as a product on random variables, as explained in \autoref{rem:wd}; no density requirement is necessary for the linear second factor, since if $X^k_{u,v}$ vanishes the Wick product does too.

\begin{remark}
	This integral has already been studied in detail in the case in which $X$ is a Gaussian process, see e.g. \cite{D99, DHP00, CCM03, alos} and many others. In that case, it can be defined for much more general (path-dependent, anticipative) integrands, and agrees with the Skorokhod integral, i.e.\ the divergence operator in Malliavin calculus, adjoint to the Malliavin derivative. It can also be defined in rougher regimes we do not consider here, e.g.\ fractional Brownian motion with $H \leq \frac 12$, see e.g.\ \cite{NT06} who define it using the same Riemann-Stieltjes sum approximations as \autoref{def:wickint}. We compute the CGF of the $2$-time marginals of $H$-fBm:
	\begin{align*}
		\K_{s,t}(z,w) &= \frac 12 z^2 s^{2H} + \frac 12 w^2 t^{2H} + \frac 12 zw \big( s^{2H} + t^{2H} - |t-s|^{2H} \big) \\
		\partial_t\K_{s,t}(z,w) &= H w^2 t^{2H - 1} + Hzw \big( t^{2H-1} - \mathrm{sgn}(t-s)|t-s|^{2H-1} \big).
	\end{align*}
For $H > \frac 12$ it is clear from this that \autoref{conditions} hold. 
\end{remark}

\begin{remark}\label{rem:notexp}
We comment that many of the identities in this and subsequent sections hold for polynomials under less restrictive conditions on $X$ than those in \autoref{conditions} (e.g.\ no radius of analyticity of $\M$ is required). We do not comment further on this, and instead restrict to measure that allow treatment of the theory for our more general class of integrands.
\end{remark}

For a tuple $(i_1, \ldots, i_n) \in [d]^n$ with $I = \{i_1, \ldots, i_n\}$ the associated multiset, we will define
\begin{equation}\label{eq:kappanotation}
\begin{split}
	\kappa^{i_1 \ldots i_n}(u_1,\ldots,u_n) &\coloneqq \kappa[X^{i_1}_{u_1},\ldots,X^{i_n}_{u_n}], \quad \kappa^{I,j}(u, \ldots, u, \dif u) \coloneqq \partial_{n+1}\kappa^{i_1 \ldots i_n j}(u,\ldots, u, u) \dif u,\\
\kappa^I(\dif u) &\coloneqq \frac{\dif}{\dif u} \kappa^I(u)\dif u, \qquad \kappa^I(u) \coloneqq \kappa[X^{i_1}_u,\ldots, X^{i_n}_u].
\end{split}
\end{equation}
where the second notation is justified by the fact that $\partial_{n+1}\kappa^{i_1 \ldots i_n j}(u,\ldots, u, u)$ is symmetric in $i_1,\ldots,i_n$. The following can be considered the main theorem of this paper.

\begin{theorem}[Wick integral and It\^o-Stratonovich formula]\label{thm:itostrat}\ \\
	Under the stated hypotheses on $X$ and $f$, the limit in \autoref{def:wickint} exists a.s.\ and for $(s,t) \in \triangle_T$
	\begin{equation}\label{eq:itostrat}
		\int_s^t f_j(X_u) \dia \dif X^j_u = \int_s^t f_j(X_u) \dif X^j_u - \sum_{I \neq \varnothing} \frac{1}{I !} \int_s^t \partial_I f_j(X_u) \kappa^{Ij}(u,\ldots, u, \dif u) \text{ ,}
	\end{equation}
	where the first integral is Young and the second is Riemann-Stieltjes according to \eqref{eq:kappanotation}, and the series is absolutely summable in $L^1$. Moreover if $e^{\sigma \|X\|_\infty} \in L^1$ for some $\sigma > r$ and if $\| X \|_{p\text{-}\mathrm{var}}$ has moments of all orders, then the integral exists as a limit in $L^1$. If additionally $\E X \equiv 0$
	\begin{align*}
		\E \int_s^t f_j(X) \dia \dif X^j &= 0, \\
		\mathbb E \int_s^t f_j(X) \dif X^j &= \sum_{I \neq \varnothing}\frac{1}{I !} \int_s^t \mathbb E[\partial_I f_j(X_u)] \kappa^{I,j}(u,\ldots, u, \dif u) \text{ .}
	\end{align*}
\end{theorem}

\begin{proof}
	Let $\widetilde\K_{u,v}(\xi,\eta) \coloneqq \K_{u,v}(\xi - \eta e_j, \eta e_j)$ be the cumulant generating function of $(X_u, X_{u,v}^j)$. Since $\widetilde \K_{u,u}(\xi,\eta) = \K_{u,u}(\xi-\eta e_j,\eta e_j) = \K_u(\xi)$, we have $\partial_\eta \widetilde \K_{u,u}(\xi,0)=0$. Hence, by the $C^1$-regularity in $v$ and a Cauchy estimate in the $w_j$-variable on a smaller polydisc, the hypotheses of \autoref{thm:wickfxy} are satisfied with $\varepsilon = v - u$ (so that \autoref{lem:cumulant_estimate} is applicable), and thus
	\[
	\sum_{[u,v] \in \pi_m} f_j(X_u) \dia X^j_{u,v} = \sum_{[u,v] \in \pi_m} f_j(X_u) X^j_{u,v} - \sum_{[u,v] \in \pi_m} \sum_{I \neq \varnothing} \frac{\partial_I f_j(X_u)}{I!} \kappa^{I,j}(u,\ldots,u, \Delta(u,v))
	\]
	where $\kappa^{I,j}(u,\ldots,u, \Delta(u,v)) = \kappa^{I,j}(u,\ldots,u, v) - \kappa^{I,j}(u,\ldots,u, u)$ and for each $[u,v] \in \pi_m$ the series is absolutely summable in $L^1$. We show convergence of the two summands separately. For the first, by \autoref{conditions}\autoref{item:1} and retracing the proof of existence of the Young integral\cite{young} (see \cite[Theorem 1.16]{LCL07}), we have
	\begin{align*}
		\bigg|\int_s^t f_j(X) \dif X^j - \sum_{[u,v] \in \pi_m} f_j(X_u) X^j_{u,v} \bigg|
		&\leq \sum_{[u,v] \in \pi_m}  \bigg| \int_u^v f_j(X) \dif X^j - f_j(X_u) X^j_{u,v} \bigg| \\
		&\lesssim_p \sum_{[u,v] \in \pi_m}\| f(X) \|_{p\text{-}\mathrm{var},[u,v]}  \| X \|_{p\text{-}\mathrm{var},[u,v]} \\
		&\leq \| \D f\|_{\infty, B(0,\max X)} \max_{[u,v] \in \pi_m} \| X \|_{p\text{-}\mathrm{var},[u,v]}^{2-p} \sum_{[u,v] \in \pi_m}  \| X \|_{p\text{-}\mathrm{var},[u,v]}^p \\
		&\leq \| \D f\|_{\infty, B(0,\max X)}  \| X \|_{p\text{-}\mathrm{var},[0,T]}^p \max_{[u,v] \in \pi_m} \| X \|_{p\text{-}\mathrm{var},[u,v]}^{2-p},
	\end{align*}
	and the maximum vanishes in the limit since $(u,v) \mapsto \| X \|_{p\text{-}\mathrm{var},[u,v]}$ is continuous. Moreover, the right hand side is bounded by $\|\D f\|_{\infty, B(0,\|X\|_\infty)} \| X \|^2_{p\text{-}\mathrm{var},[0,T]}$: $L^1$ convergence will follow if we show this quantity is integrable. Indeed, if $e^{\sigma \|X\|_\infty} \in L^1$ for some $\sigma > r$ and $\| X \|_{p\text{-}\mathrm{var}} \in \bigcap_{s \geq 1} L^s$ implies this by the fact that $\D f$ has the same exponential type as $f$, and the H\"older inequality, by setting $\sigma =qr$ with $q > 1$.
	
	For the second term, we have a.s.\ for each fixed $I$
	\begin{equation}\label{eq:RSlimitI}
		\sum_{[u,v] \in \pi_m} \frac{\partial_I f_j(X_u)}{I!} \kappa^{I,j}(u,\ldots,u, \Delta(u,v)) \xrightarrow{m \to \infty} \frac{1}{I !} \int_s^t \partial_I f_j(X_u) \kappa^{I,j}(u,\ldots, u, \dif u)
	\end{equation}
	by $C^1$ regularity of $(u,v) \mapsto \kappa^{I,j}(u,\ldots,u, v)$
	\[
	\partial_{n+1}\kappa^{i_1 \ldots i_n j}(u,\ldots,u,u) = \partial_v \partial_{z_I} \partial_{w_j} \K_{u,v}(0,0) = \partial_{z_I} \partial_{w_j} \partial_v \K_{u,v}(z,w)\big|_{v = u; z = 0 = w},
	\]
	since \autoref{conditions}\autoref{item:2} implies $(u,v) \mapsto \K_{u,v}$ is a $C^1$ map into the space of holomorphic functions under the topology of local uniform convergence, and therefore it is possible to pass to the derivative in the $v$ parameter inside Cauchy's integral formula. Then by \autoref{lem:cumulant_estimate} in conjunction with \autoref{conditions} and \autoref{lem:growth}, picking $q > 1$ small enough so that $qr < R$ and for some $a < 1$ (entrywise), by Jensen we have
	\begin{align*}
		\bigg| \frac{1}{I!} \partial_I f_j(X_u) \kappa^{I,j}(u,\ldots,u, \Delta(u,v))  \bigg| &\lesssim_a e^{r|X_u|} a^I (v-u) \\
		\implies \E\bigg| \sum_{[u,v] \in \pi_m} \frac{1}{I!} \partial_I f_j(X_u) \kappa^{I,j}(u,\ldots,u, \Delta(u,v))  \bigg|^q &\lesssim_a (t-s)^{q-1} a^I \sum_{[u,v] \in \pi_m} \E e^{qr|X_u|} (v-u) \\
		&\leq (t-s)^q a^I \sup_{u \in [0,T]} \E e^{qr|X_u|} 
	\end{align*}
	which is finite by \autoref{conditions}\autoref{item:2}. Boundedness in $L^q$ implies uniform integrability of the Riemann-Stieltjes sums of a fixed-$I$ summand and we conclude that the limit \eqref{eq:RSlimitI} is in fact in $L^1$. Note that by Fatou this also implies the $I$-integral is bounded in $L^1$ by the same bound. Then, exchanging sums thanks to absolute $L^1$-summability
	\begin{align*}
		&\bigg\| \sum_{[u,v] \in \pi_m}\sum_{I\neq\varnothing} \frac{\partial_I f(X_u)}{I!} \kappa^{I,j}(u,\ldots,u, \Delta(u,v)) - \sum_{I \neq \varnothing} \frac{1}{I !} \int_s^t \partial_I f_j(X_u) \kappa^{I,j}(u,\ldots, u, \dif u) \bigg\|_{L^1} \\
		\leq{}&\sum_{I \neq \varnothing} \bigg\| \sum_{[u,v] \in \pi_m} \frac{\partial_I f(X_u)}{I!} \kappa^{I,j}(u,\ldots,u, \Delta(u,v)) - \frac{1}{I !} \int_s^t \partial_I f_j(X_u) \kappa^{I,j}(u,\ldots, u, \dif u) \bigg\|_{L^1} \\
		\lesssim_a{} &\sum_{I \neq \varnothing}a^I < \infty
	\end{align*}
	with each $I$-summand convergent to $0$. We can therefore conclude by dominated convergence.
	
	By the zero-mean property of Appell polynomials, $\sum_{[u,v] \in \pi_m} f_j(X_u) \dia X^j_{u,v}$ vanish in mean as long as $X$ does, and so does their $L^1$-limit. To pass to the limit inside the expectation we use Fubini, which is justified by
	\[
	\sum_{I \neq \varnothing} \frac{1}{I!} \int_s^t  \E | \partial_I f_j(X_u) | |\kappa^{I,j}|(u,\ldots, u, \dif u) \lesssim_a \sum_{I \neq \varnothing} a^I \sup_u \E e^{r|X_u|} 
	\]
	where we have used that
	\begin{align*}
		|\partial_{n+1} \kappa(u,\ldots,u, u)| = \lim_{v \searrow u}  \bigg|\frac{ \kappa(u,\ldots,u, \Delta(u,v))}{v-u}\bigg|
	\end{align*}
	thanks to \autoref{conditions}\autoref{item:2} and subsequently applying \autoref{lem:cumulant_estimate} once again.
\end{proof}

\begin{remark}[Non-centred case]\label{rem:non-centred}
	The formulae in this section all work in the non-centred case as well. In order to obtain a centred integral in that case, set $\mu^j(t) \coloneqq \kappa^j(t) = \E X^j_t$ (assumed $C^1$) and $\widetilde X{}^j \coloneqq X^{\dia j}_t = X^j_t - \mu^j(t)$, so it holds that
	\[
	\int_s^t f_j(X) \dia \dif X^j = \int_s^t f_j(X) \dia \dif \widetilde X{}^j + \int_s^t f_j(X_u) \mu^j(\dif u)
	\]
	with $\int\! f(X) \dia \dif \widetilde X$ centred. In this case $f(X_t) = f(\mu(t) + \widetilde X_t)$ can be viewed as a time-dependent function of the new variables $\widetilde X_t$, see \eqref{rem:timeito} below; note that the linear change of variables does not affect the exponential type of $f$. The centred Wick integral therefore admits a similar expression to \eqref{eq:itostrat} but including $I = \varnothing$ in the sum.
\end{remark}

By making the integrand exact, we can leverage symmetry to obtain a change-of-variables formula analogous to the It\^o formula, with higher-order corrections, corresponding to the infinitely many non-zero cumulants (if $X$ is non-Gaussian). Recall the notation \eqref{eq:kappanotation}.
\begin{theorem}[It\^o series formula]\label{thm:ito}
	Under the same hypotheses on $X$ and $f\colon \mathbb{R}^d \to \mathbb{R}$ as in \autoref{thm:itostrat},
	\[
	f(X_t)-f(X_s) = \int_s^t \D f(X) \dia \dif X + \sum_{|I| \geq 2} \frac{1}{I!} \int_s^t \partial_I f(X) \kappa^I(\dif u) \text{ .}
	\]
\end{theorem}
\begin{proof}
	By the previous result,
	\begin{align*}
		\int_s^t \partial_j f(X) \dia \dif X^j &= \int_s^t \partial_j f(X) \dif X^j - \sum_{I \neq \varnothing} \frac{1}{I !} \int_s^t \partial_{Ij} f(X_u) \kappa^{I j}(u,\ldots, u, \dif u) \text{ .}
	\end{align*}
	The first summand on the right is equal to $f(X_t) - f(X_s)$ by the ordinary change-of-variables formula for Young integrals. Now, for a multiset $H$
	\[
	\frac{\dif}{\dif u} \kappa^H(u) = \sum_{i = 1}^{|H|} \partial_i \kappa^H(u,\ldots,u) \text{ .}
	\]
	Therefore, letting $H$ correspond to the multi-index $(H_1,\ldots,H_d)$,
	\begin{align*}
		&\sum_{I \neq \varnothing} \frac{1}{I !} \int_s^t \partial_{Ij} f(X_u) \kappa^{I j}(u,\ldots, u, \dif u) \\
		={}&\sum_{|H| \geq 2}\sum_{\substack{j = 1 \\ j \in H}}^d \frac{1}{(H \setminus j)!}\int_s^t \partial_H f(X_u) \kappa^{H}(u,\ldots, u, \dif u) \\
		={}& \sum_{|H| \geq 2} \frac{1}{H!}\Big(\sum_{\substack{j = 1 \\ j \in H}}^d H_j\Big)\int_s^t \partial_H f(X_u) \frac{1}{|H|} \sum_{i = 1}^{|H|} \kappa^{H}(u,\ldots, \dif u, \ldots , u) \\
		={}& \sum_{|H| \geq 2}\frac{1}{H!}\int_s^t \partial_H f(X_u) \kappa^{H}(\dif u).\qedhere
	\end{align*}
\end{proof}

\begin{remark}[It\^o formula with time dependence]\label{rem:timeito}
	The It\^o-Stratonovich formula \autoref{thm:itostrat} is unaffected if the integrand $f(X_u, u)$ is time-dependent, as long as it is jointly continuous with $f( \, \cdot \, , u)$ satisfying the hypotheses of the theorem uniformly in $u$, with a common exponential type. Assuming additionally that $f$ is smooth in time,
	\[
	f(X_u, u)\big|_{u = s}^t = \int_s^t \D_x f(X_u, u) \dif X_u + \int_s^t \partial_u f(X_u, u) \dif u
	\]
	and thus the above proof can immediately be readapted to imply
	\[
	f(X)_{s,t} = \int_s^t \D_x f(X_u,u) \dia \dif X + \int_s^t \partial_u f(X_u, u) \dif u + \sum_{|I| \geq 2} \frac{1}{I!} \int_s^t \partial_I f(X) \kappa^I(\dif u) \text{ .}
	\]
	Note that if $X$ is non-centred \emph{and} $f$ is time dependent, the Wick integral in the It\^o formula above can be further expressed as sum of two terms, one of which centred, as explained in \autoref{rem:non-centred}, and the other which can be incorporated in the sum by letting it start at $|I| = 1$.
\end{remark}

We note how the formula simplifies to the well-known It\^o formula for centred Gaussian processes, thanks to the vanishing of cumulants of order other than $2$. The following result is immediate, and in the literature on Malliavin calculus and white noise analysis, it has been extended to much more general functions than those considered here (indeed path-dependent and even anticipative).

\begin{corollary}[The Gaussian case]
	If $X$ is a centred Gaussian process (still satisfying \autoref{conditions})
	\begin{align*}
		\int_s^t f_j(X) \dia \dif X^j &= \int_s^t f_j(X) \dif X^j - \int_s^t \partial_i f_j(X_u) \kappa^{i j}(u,\dif u) \\
		f(X)_{s,t} &= \int_s^t \partial_j f(X_u) \dia \dif X^j_u + \frac{1}{2} \int_s^t \partial_{i j} f(X_u) \kappa^{i j}(\dif u) \text{ .}
	\end{align*}
\end{corollary}

\autoref{thm:itostrat} may sometimes be simplified in a similar way even for non-exact polynomial integrands, thanks to symmetries in the law of $X$. Recall that a random vector is called exchangeable if its law is invariant under permutation of the components; for a Gaussian vector this is easily seen to happen if and only if there is a common value for the variances and a common value for the off-diagonal covariances. The proof of the following result is similar to that of the previous ones, and omitted.

\begin{corollary}[Symmetries in the It\^o-Stratonovich formula]\label{eq:symmetries}\ \\
	If the components of $X$ are independent, \autoref{thm:itostrat} simplifies to
	\[
	\int_s^t f_j(X) \dia \dif X^j = \int_s^t f_j(X) \dif X^j - \sum_{n \geq 2} \frac{1}{n!} \int_s^t \partial_{j^{n-1}} f_j(X_u) \kappa_n^j(\dif u)
	\]
	where $j^k$ is the multiset containing $k$ copies of $j$ and $\kappa_n^j(u) \coloneqq \kappa^{j^n}(u)$.
	
	If $X$ is a Gaussian process with exchangeable components with variance function $\kappa^{ii}(u) = \sigma^2(u)$ and off-diagonal covariance function $\kappa^{ij}(u) = \rho(u)$ (any $i \neq j$), then
	\[
	\int_s^t f_j(X) \dia \dif X^j = \int_s^t f_j(X) \dif X^j - \frac 12 \int_s^t \Big(\sum_{i \neq j}\partial_i  f_j(X_u) \Big) \rho(\dif u) - \frac 12 \int_s^t \Big( \sum_j \partial_j f_j(X_u) \Big) \sigma^2(\dif u) \text{ .}
	\]
\end{corollary}

We record, for ease of reference, the formulae in the one-dimensional case, along with an identity for integrals of Wick powers which extends the well-known corresponding identity for Hermite polynomials in the Gaussian case.
\begin{proposition}[The scalar case]\label{prop:1d_Ito}
	If $X$ is one-dimensional
	\begin{equation*}
		\begin{split}
			\int_s^t f(X_u) \dia \dif X_u &= \int_s^t f(X_u) \dif X_u - \sum_{n = 2}^\infty \frac{1}{n!} \int_s^t f^{(n-1)}(X_u) \kappa_n(\dif u)\\
			f(X)_{s,t} &= \int_s^t f'(X_u) \dia \dif X_u + \sum_{n = 2}^\infty \frac{1}{n!} \int_s^t f^{(n)}(X_u) \kappa_n(\dif u)
		\end{split}
	\end{equation*}
	Furthermore, if $X$ is centred, for $n \geq 0$ (assuming the shifted process $(X_{s,u})_{u \geq s}$ satisfies \autoref{conditions})
	\begin{equation*}
		\int_s^t X^{\dia n} \dia \dif X = \frac{(X^{\dia (n+1)})_{s,t}}{n+1}, \qquad \int_s^t (X_{s,u})^{\dia n} \dia \dif X =  \frac{(X_{s,t})^{\dia (n+1)}}{n+1}\,\text{ .}
	\end{equation*}
\end{proposition}
\begin{proof}
	The first two are immediate corollaries of \autoref{eq:symmetries}. Let $w_n(t,x)$ denote the $n^\text{th}$ Appell polynomial of $X_t$, i.e.\ $w_n(t, X_t) = X_t^{\dia n}$. Then $\partial^j_x w_n(t,x) = \frac{n!}{(n-j)!} w_{n-j}(t,x)$ and letting $\EuScript{K}(t,\theta) = \sum_{n = 1}^\infty \kappa_n(t)\frac{\theta^n}{n!}$ be the cumulant-generating function of $X_t$, by \autoref{prop:wick}
	\begin{align*}
		\partial_t w_n(t, x) &= \partial_t \partial_\theta^n \exp(\theta x - \EuScript{K}(t, \theta))|_{\theta = 0} = -\partial_\theta^n \big[ \exp(\theta x - \EuScript{K}(t, \theta)) \partial_t \EuScript{K}(t, \theta)\big]_{\theta = 0} \\
		&= -\sum_{j = 0}^n \binom{n}{j} \partial^{n-j}_\theta \exp(\theta x - \EuScript{K}(t, \theta)) \partial^j_\theta \partial_t \EuScript{K}(t,\theta)|_{\theta = 0} \\
		&= -\sum_{j = 0}^n \binom{n}{j} w_{n-j}(t, x) \partial_t \kappa_j(t) = -\sum_{j = 2}^n \frac{1}{j!} \partial_x^j w_n(t, x) \partial_t \kappa_j(t)
	\end{align*}
	Then by the It\^o formula with time dependence \autoref{rem:timeito}
	\begin{align*}
		&X^{\dia(n+1)}_t - X^{\dia(n+1)}_s \\
		={} &(n+1) \int_s^t w_n(u,X_u) \dia \dif X_u + \int_s^t \partial_u w_{n+1}(u,X_u) \dif u + \sum_{j = 2}^{n+1}  \frac{1}{j!} \int_s^t \partial_x^j w_{n+1}(u, X_u) \partial_u \kappa_j(u) \dif u \\
		={} &(n+1) \int_s^t X_u^{\dia n} \dia \dif X_u \text{ .}
	\end{align*}
	The second identity follows from the first, applied to the shifted stochastic process.
\end{proof}

\begin{example}\label{1d_exp_int}
	We compute Wick integrals of exponentials and Wick exponentials. Using that $\kappa_1 \equiv 0$ and for $r > 0$,
	\begin{align*}
		\int_s^t e^{rX_u} \dia \dif X_u &= \int_s^t e^{rX_u} \dif X_u - \sum_{n = 2}^\infty \int_s^t \frac{r^{n-1}}{n!} e^{rX_u} \kappa_n(\dif u) \\
		&= \frac{e^{rX_t} - e^{rX_s}}{r} - \int_s^t e^{rX_u} \sum_{n = 1}^\infty\frac{r^{n-1}}{n!}  \kappa_n(\dif u) = \frac{1}{r} \bigg[ e^{rX_t} - e^{rX_s} - \int_s^t e^{rX_u} \dif_u \K_u(r) \bigg]
	\end{align*}
	where $\K_u(r) = \log \E e^{rX_u}$. Integrating the Wick exponential instead yields a more familiar identity (here we argue directly, but the same could have been obtained by applying the time-dependent It\^o formula \autoref{rem:timeito}):
	\begin{align*}
		\int_s^t \exp_\dia(rX_u) \dia \dif X_u &= \int_s^t \exp(rX_u - \K_u(r)) \dia \dif X_u \\
		&= \int_s^t \exp(rX_u - \K_u(r)) \dif X_u - \frac 1r \int_s^t \exp(rX_u - \K_u(r)) \dif_u\K_u(r) \\
		&= \frac 1r\int_s^t \dif_u\exp(rX_u - \K_u(r)) = \frac{\exp_\dia(rX_t) - \exp_\dia(rX_s)}{r}.
	\end{align*}
	Here the Wick exponential can be viewed as extending the stochastic exponential in It\^o calculus.
\end{example}

\begin{example}[Gaussian process $\times$ independent random variable]
	Let \(	X_t = Y W_t	\)	where $Y$ is a (static) $\mathbb R^{d \times d}$-valued random matrix and $W$ is a centred Gaussian process with independent components, independent of $Y$, with covariance function that is constant across components:
	\[
	\E[W^i_s W^i_t] = R(s,t), \quad R(t) \coloneqq R(t,t)
	\]
	with $R \in C^1([0,T]^2)$ and first derivative $0 < \alpha$-H\"older. Also assume $R(t,t) > 0$ for $t > 0$ and $\det Y \neq 0$ a.s.
	
	Conditionally on $Y$, the process $X$ is Gaussian, and
	\begin{align*}
		\M_{s,t}(z,w) = \E \E[\M_{s,t}(z,w) | Y]={} \E\exp\Bigg[\frac12\Big(
		R(s) z^\intercal Y Y^\intercal z + 2R(s,t) z^\intercal Y Y^\intercal w + R(t) w^\intercal Y Y^\intercal w
		\Big)\Bigg].
	\end{align*}
	The argument of the exponential is bounded by $(|z|^2 + |w|^2)|Y|^2$ ($|Y|$ the operator norm of the matrix) up to a multiplicative constant, uniformly in $s,t$. Thus, if $\E e^{\lambda |Y|^2} < \infty$ for some $\lambda > 0$, $\K_{s,t}$ is analytic in a sufficiently small polydisc around the origin, and differentiating under the expectation shows $\M$ is $C^1$ in the time variables and thus so is $\K$, since $\partial_t \K = \partial_t \M/\M$. The path regularity is entirely inherited from $W$, and an application of Kolmogorov's continuity argument implies that this is $(1 + \alpha)/2 > \beta$-H\"older and thus of bounded $2/(1+\alpha) < p$-variation. The required exponential integrability (using Euclidean, not coordinatewise, norms, for simplicity) follows by observing that by Fernique there exists $a > 0$ such that $\E e^{a \| W \|_\infty^2} < \infty$, and thus bounding $\sigma |Y| \|W\|_\infty \leq \frac a2 \|W\|^2_\infty + \frac{\sigma^2}{2a}|Y|^2$ by Young's inequality, and taking $\sigma^2 < 2a\lambda$. By conditioning and a change of variables, $X_t$ has density given a.e.\ by $p_{X_t}(x) = \E p_{W_t}(Y^{-1}x)/|\det Y|$ which is strictly positive for all $x \in \mathbb R^d$ and $t > 0$.

	If the goal is to re-center integrals w.r.t.\ $X$, a viable option (available as long as $Y$ is defined on the same Gaussian Hilbert space) is to view $f(YW_t)$ as a function of $W_t$ and to use Skorokhod integration 
	\begin{align*}
		\int_s^t f_j(X_u) \dif X^j_u &= Y^j_i\int_s^t f_j(X_u) \dif W_u^i = Y^j_i \int_s^t f_j(X_u) \sko W_u^i +  \sum_{i = 1}^d \frac{Y^j_i Y^h_i}{2} \int_s^t \partial_h f_j(X_u) R(\dif u) .
	\end{align*}
	Here $\sko W$ should be read as \say{$\dia_W \dif W$}, i.e.\ where the Wick product is taken w.r.t.\ $W$. This approach is valid but has the drawback of the right hand side of the expression requiring one to split $X$ into its \say{non-observable} constituents $Y$ and $W$. Moreover, this approach would not even yield a centred integrand in the more general case in which $Y$ and $W$ are not independent. A more intrinsic option (which however involves infinite series and is restricted to entire integrands of exponential type) is instead to use the formulae in this section. In order to apply these, one must be able to compute joint cumulants of $X$. This can be done by using Malyshev's formula for cumulants of products of random variables, see \cite[Proposition 3.2.1]{PT}. Assume for simplicity that $Y$ is diagonal: we compute for an even number of terms (for an odd number the cumulant vanishes)
	\begin{align*}
		\kappa[X^{i_1}_u, \ldots, X^{i_{2n-1}}_u, X^j_v] = R(u)^{n-1} R(u,v) \underbrace{\sum_{\substack{\pi \times \rho \, \in \, \mathcal P([2n]) \times_\mathrm{C} \mathcal P_{\mathrm{G}}([2n]) \\ k_\beta = 1,\ldots, d \ : \ \beta \in \rho }} \prod_{\alpha \in \pi}\kappa[Y^{i_r}_{k_{\rho(r)}} : r \in \alpha]}_{\eqqcolon C^{I,j}}.
	\end{align*}
	where $\rho(r)$ is the pair in $\rho$ to which $r$ belongs, we are taking products of partitions by independence, and the modifier $\mathrm{C}$ to the cartesian product means that the two partitions viewed together must be a connected diagram. An example of a diagram considered is given in \autoref{fig:malyshev}.
	\begin{figure}[ht!]
		\centering
		\begin{tikzpicture}[rotate = 90]
			\dotrows{2,2,2,2,2,2}
			\draw[Cerulean] (\dotcoord{1,1}) to (\dotcoord{2,1});
			\draw[Cerulean] (\dotcoord{2,1}) to[bend right = 30] (\dotcoord{4,1});
			\draw[Cerulean] (\dotcoord{3,1}) to[bend right = 30] (\dotcoord{6,1});
			\draw[Cerulean, line width=1.75] (\dotcoord{5,1}) to[bend right = 30] (\dotcoord{5,1});
			\draw[OrangeRed] (\dotcoord{1,2}) to[bend left = 30] (\dotcoord{5,2});
			\draw[OrangeRed] (\dotcoord{2,2}) to (\dotcoord{3,2});
			\draw[OrangeRed] (\dotcoord{4,2}) to[bend left = 30] (\dotcoord{6,2});
		\end{tikzpicture}
		\caption{An example of a diagram considered in the above sum. The picture is rotated by 90$^\circ$ counterclockwise, to improve the layout, so that rows have become columns. The blue edges represent those in $\pi$, and notice that singletons are allowed since $Y$ is not assumed centred. The red edges are those in $\rho$, and only consist of pairings since $W$ is Gaussian. By independence the edges do not cross over into the other column (displayed as a row), but taken together the diagram cannot be separated into two sub-diagrams (connectedness), so the blue and red diagrams must be summed over jointly.}
		\label{fig:malyshev}
	\end{figure}
	
	Differentiating w.r.t.\ $v$ and evaluating at $v = u$ we obtain
	\begin{align*}
		&\partial_{2n}\kappa^{I,j}(u, \ldots, u, u) = \frac 12 R(u)^{n-1}R'(u) C^{I,j}, \qquad \frac{\dif}{\dif u} \kappa^{Ij}(u) = n R(u)^{n-1}R'(u) C^{I,j} \\
		\implies{} &\kappa^{I,j}(u,\ldots,u, \dif u) = \frac{1}{2n}\kappa^{I,j}(\dif u).
	\end{align*}
	An example of such a process that is considered in the literature as a model of anomalous diffusion is \emph{fractional grey noise} \cite{Sch92, MP08} (with $H > 1/2$), cf.\ \cite{GJ16,BGO23} for related forms of Wick calculus. In this case, taking $d = 1$, we have $Y = \sqrt{S_\beta}$ where $S_\beta$ with $0 < \beta \leq 1$ Mittag-Leffler distributed, i.e.\ with MGF given by
	\(
	\E e^{z S_\beta} = \sum_{n = 0}^\infty \frac{z^n}{\Gamma(\beta n + 1)}.
	\)
	This is an entire function but one must still restrict to a finite radius since it may have complex zeros for $\beta < 1$.
\end{example}

\section{Non-Gaussian Wick product in Wiener chaos}\label{sec:wiener}
Let $\mathcal H$ be a separable real Hilbert space with inner product
$\langle \,\cdot\,,\,\cdot\,\rangle$, and let
$W \colon \mathcal H \to L^2(\Omega)$ be a centred isonormal Gaussian process,
as defined, for example, in \cite[Ch.~2]{NP}. We recall the Wiener--It\^o chaos
isomorphism
\[
\I = \bigoplus_{n=0}^\infty \I^n
\colon
{\textstyle\bigodot}(\mathcal H)
\xrightarrow{\cong}
\bigoplus_{n=0}^\infty \mathcal W^n
=
L^2(\Omega).
\]
Here \({\textstyle\bigodot}(\mathcal H)= \bigoplus_{n=0}^\infty \mathcal H^{\odot n}\) denotes the Hilbert completion of the algebraic symmetric algebra over
$\mathcal H$, i.e.\ the symmetric Fock space, and
$\mathcal W^n = \I^n(\mathcal H^{\odot n})$ is the $n^\text{th}$ homogeneous
Wiener chaos. We use the convention that $\mathcal H^{\odot n}$ is equipped with the norm $
\sqrt{n!}\,\Vert \cdot \Vert_{\mathcal H^{\otimes n}}$. This differs from the convention adopted in some parts of the literature, as in \cite{NP}, where the symmetric tensor powers are equipped with the inherited tensor-product norm and the multiple Wiener integrals are correspondingly normalised by a factor $(\sqrt{n!})^{-1}$. With our choice of norm, the Wiener--It\^o isomorphism is instead normalised so that
\[
\I^n(h_1 \odot \cdots \odot h_n)
=
W(h_1) \dia_W \cdots \dia_W W(h_n),
\qquad
h_1,\ldots,h_n \in \mathcal H,
\]
without additional factorial factors. Here $\dia_W$ denotes the Wick product with respect to the Gaussian family $W$, and the right-hand side is the multivariate
Hermite polynomial of total degree $n$ associated with the Gaussian vector $(W(h_1),\ldots,W(h_n))$. We use the symbol $\dia_W$ to distinguish it from $\dia$, the Wick product introduced below for specified random variables belonging
to higher Wiener chaoses.

 Since for centred Gaussian random variables, cumulants of order different from $2$ vanish, the only diagrams that need to be considered when $\mathcal I = W(\mathcal H) = \mathcal W^1$ are ones in which all edges connect exactly two nodes; we call such diagrams \emph{Gaussian} and denote them by using the subscript $\mathrm{G}$. For these, edges represent inner products $\E[W(h)W(g)] = \langle h, g \rangle$. Given $(\pi,F) = D \in \calD_\mathrm{G}$ we set
\begin{equation}\label{def_on_pure}
\I(D) \coloneqq W^{\dia_W D} = \I(F) \prod_{\{h,g\} \in \pi}\langle h,g\rangle \text{ .}
\end{equation}
We now extend this notation to rows decorated by arbitrary symmetric tensors via partial contractions, adapting the conventions of \cite[Appendix B4]{NP} to the normalization adopted above.

\begin{proposition}\label{prop:Iextension} Definition \eqref{def_on_pure} extends uniquely to non-flat Gaussian diagrams whose rows are decorated by arbitrary tensors in $\bigodot(\mathcal H)$, with the degree of each tensor equal to the length of the corresponding row. \end{proposition}
\begin{proof}
Let $f_1 \in \mathcal{H}^{\odot n_1}, \ldots, f_m \in \mathcal{H}^{\odot n_m}$,
let $(\pi,F)=D\in\mathcal{D}_{\mathrm{G}}(f_1,\ldots,f_m)$ be non-flat, and let
$\{e_r\}_{r=1}^\infty$ be an orthonormal basis of $\mathcal H$. Set
$N=n_1+\cdots+n_m$, enumerate the dots of $D$ from $1$ to $N$, and denote
the unpaired dots by $r_1,\ldots,r_k$, where $k=N-2|\pi|$. Define the
unrenormalised partial contraction
\begin{equation}\label{eq:general_contraction}
\begin{split}
C_\pi(f_1,\ldots,f_m)
\coloneqq{}&
\sum_{i_1,\ldots,i_N\geq 1}
\langle
f_1\otimes\cdots\otimes f_m,\,
e_{i_1}\otimes\cdots\otimes e_{i_N}
\rangle_{\mathcal H^{\otimes N}}
\bigg(\prod_{\{r,s\}\in\pi}\delta_{i_r i_s}\bigg)
e_{i_{r_1}}\otimes\cdots\otimes e_{i_{r_k}}\,.
\end{split}
\end{equation}
The series is understood as the limit in $\mathcal H^{\otimes k}$ of the
finite-dimensional truncations obtained by restricting all indices to
$1,\ldots,M$. We first verify that this limit exists and is independent of
the choice of orthonormal basis. For elementary tensors
$f_\ell=h_{\ell,1}\otimes\cdots\otimes h_{\ell,n_\ell}$, let
$h_1,\ldots,h_N$ denote the corresponding tensor factors, ordered according
to the enumeration of the dots of $D$. By the completeness relations
\[
\sum_{i=1}^\infty
\langle h,e_i\rangle\langle g,e_i\rangle
=
\langle h,g\rangle,
\qquad
\sum_{i=1}^\infty
\langle h,e_i\rangle e_i
=
h,
\]
the expression in \eqref{eq:general_contraction} reduces to the intrinsic
operation of contracting the tensor components corresponding to the pairs in
$\pi$ through the inner product of $\mathcal H$ and retaining the components
corresponding to the unpaired dots. In particular, it is independent of the
choice of orthonormal basis. The non-flatness assumption is essential for extending this operation to
arbitrary Hilbert tensors. Indeed, since no pair in $\pi$ joins two dots
belonging to the same row, every contracted index occurs in two distinct row
tensors. Repeated applications of the Cauchy--Schwarz inequality to the contracted indices therefore yield
\[
\Vert C_\pi(f_1,\ldots,f_m)\Vert_{\mathcal H^{\otimes k}}
\leq
\prod_{\ell=1}^m
\Vert f_\ell\Vert_{\mathcal H^{\otimes n_\ell}}
\]
for algebraic tensors. Hence the contraction extends uniquely to a continuous
multilinear map on
\(
\mathcal H^{\otimes n_1}\times\cdots\times\mathcal H^{\otimes n_m},
\)
and therefore, by symmetrisation, on
\(
\mathcal H^{\odot n_1}\times\cdots\times\mathcal H^{\odot n_m}.
\) Indeed, if $P_M$ denotes the orthogonal projection onto
$\operatorname{span}\{e_1,\ldots,e_M\}$, then the $M$-th truncation of
\eqref{eq:general_contraction} is precisely the contraction obtained by
replacing each $f_\ell$ with $P_M^{\otimes n_\ell}f_\ell$. Since
\[
P_M^{\otimes n_\ell}f_\ell\longrightarrow f_\ell
\qquad\text{in }\mathcal H^{\otimes n_\ell},
\]
the above continuity estimate shows that these truncations converge in
$\mathcal H^{\otimes k}$. Their limit is the continuous extension of the
intrinsic contraction described above and is therefore independent of the
orthonormal basis.

We now set
\begin{equation}\label{eq:random_diagram}
\I(D)
\coloneqq
\I^k\bigl(\mathrm{Sym}_k C_\pi(f_1,\ldots,f_m)\bigr),
\end{equation}
where $\mathrm{Sym}_k \colon\mathcal H^{\otimes k}\to\mathcal H^{\odot k}$
denotes symmetrisation. Since symmetrisation is contractive with respect to
the norm of $\mathcal H^{\otimes k}$ and, under our normalisation, $\I^k$ is
an isometry from $\mathcal H^{\odot k}$ into $L^2(\Omega)$, we obtain
\[
\begin{aligned}
\Vert \I(D)\Vert_{L^2(\Omega)}
&=
\sqrt{k!}\,
\Vert \mathrm{Sym}_k C_\pi(f_1,\ldots,f_m)
\Vert_{\mathcal H^{\otimes k}}\leq
\sqrt{k!}\,
\prod_{\ell=1}^m
\Vert f_\ell\Vert_{\mathcal H^{\otimes n_\ell}}
=
\sqrt{\frac{k!}{n_1!\cdots n_m!}}\,
\prod_{\ell=1}^m
\Vert f_\ell\Vert_{\mathcal H^{\odot n_\ell}}.
\end{aligned}
\]
Thus the map is continuous in the tensors decorating the rows. For elementary
symmetric tensors, \eqref{eq:random_diagram} agrees with Definition
\eqref{def_on_pure}. Since the linear span of elementary symmetric tensors is
dense in each $\mathcal H^{\odot n_\ell}$ and the map defined above is
continuous, the extension is unique.
\end{proof}

For $f_1 \in \calH^{\odot n_1}, \ldots, f_m \in \calH^{\odot n_m}$ we will then write $(\pi, F) = D \in \calD_{\mathrm{G}}(f_1,\ldots,f_m)$, and $\I(D) \in \mathcal W^{n_1 + \ldots + n_m - 2|\pi|}$. If $D$ is a total diagram, $\I^0(D)$ is a real number. Using this notation and \autoref{prop:EkW} we recover the following known result (see e.g.\ \cite{NP10} and \cite[Ch.8]{NP}).
\begin{corollary}\label{prop_moment_cumulant}
With notation as above,
\begin{align*}
	\mathbb E[\I^{n_1}(f_1) \cdots \I^{n_m}(f_m)] &= \sum_{D \in \mathcal P_{\mathrm{G}, \mathrm{NF}}(f_1,\ldots,f_m)} \I^0(D), \\
	\kappa[\I^{n_1}(f_1), \ldots , \I^{n_m}(f_m)] &= \sum_{D \in \mathcal P_{\mathrm{G}, \mathrm{NF}, \mathrm{C}}(f_1,\ldots,f_m)} \I^0(D) \text{ .}
\end{align*}
\end{corollary}
\begin{remark}\label{L2remark}
In the particular case in which $\mathcal H = L^2(T,\mathcal{B},\mu)$ for a sigma-finite atomless measure space $(T,\mathcal{B},\mu)$,
the operator $\I^0$ takes a concrete form.
For symmetric functions $f_i\in L^2(T^{n_i}, \mu^{\otimes n_i})$, $i=1,\ldots,m$ and a partition $\pi \in \mathcal{P}_{\mathrm{G}}(f_1,\ldots, f_m)$,
by setting $N = n_1+\ldots +n_m$ one has the identity
\begin{equation}\label{eq_contractions}
\begin{split}
\I^0(\pi) ={} &\int_{T^{N/2}}
f_1(t_1^{\pi},\ldots,t_{n_1}^{\pi})
\cdots f_m(t_{N-n_m+1}^{\pi},\ldots,t_{N}^{\pi})
\,\mathrm d\mu^{\otimes N/2}\,,
\end{split}
\end{equation}
 where the variables $t_i^{\pi}$ are identified according to the pairings
of the partition $\pi$. Diagram formulae in this setting can be found in \cite{PT}. More generally, for any diagram $ D=(\pi,F)\in \mathcal D_{\mathrm G}(f_1,\ldots, f_m)$,
we consider the function
\begin{equation}\label{eq_partial_contractions}
\begin{split}
f_D (x_1,\ldots ,x_{|F|})
:={} &
\int_{T^{|\pi|}}
f_1(t_1^{D},\ldots,t_{n_1}^{D})
\cdots f_m(t_{N-n_m+1}^{D},\ldots,t_{N}^{D})
\,\mathrm d\mu^{\otimes |\pi|},
\end{split}
\end{equation}
where the variables $t_i^{D}$ are identified according to the diagram
$\pi$, while the free vertices correspond to the variables
$x_1,\ldots,x_{|F|}$. One then has $\I(D)= \I^{|F|}\bigl(\mathrm{Sym} (f_D)\bigr)$. In this sense, the operator $\I(D)$ replaces the usual tensor contraction in $\mathcal H$, see \cite[Appendix B]{NP}.
\end{remark}

In addition to standard moment-cumulant computations, \autoref{prop:Iextension} allows us to reformulate \autoref{thm:change} in this setting. Let $\mathcal I$ be an index set, $n_i \geq 1$, $f_i \in \calH^{\odot n_i}$ and we consider the subspace of $L^2(\Omega)$
\begin{equation}\label{eq:XWiener}
\mathcal X \coloneqq \mathrm{span}\{ \I^{n_i}(f_i) : i \in \mathcal I \}
\end{equation}
We use $\dia$ to denote the Wick product and Appell polynomials associated to $\mathcal X$; this is the same product as $\dia_W$ only when all the factors are in the first Wiener chaos.

\begin{corollary}\label{cor:change_of_chaos}
	Let $f_i \in \mathcal H^{\odot n_i}$ for $i=1,\ldots,m$. Then
	\[
	y^1 \dia \cdots \dia y^m \big|_{y^j = \I^{n_j}(f_j)} = \sum_{\substack{(\pi, K) \in \mathcal D_{\mathrm{G},\mathrm{NF}, \mathrm{C}}(f_1,\ldots,f_m) \\ |K| > 0}}  \I^{|K|}(\pi,K).
	\]
\end{corollary}

As explained in \autoref{rem:wd}, the Wick product $\dia$ does not in general descend to the level of random variables, and in order for this to happen over a linear space of kernels spanned by a collection $\{ f_i \in \calH^{\odot n_i} \mid i \in \mathcal I \}$ it is sufficient for all finite-dimensional marginals to admit densities; in fact \cite{NNP13} show that this is equivalent to the absence of polynomial relations. Shigekawa \cite{Sh80} proved that every non-constant random variable of the form $F = \sum_{k=1}^{N}\I^{n_k}(f_{k})$ admits a density, by showing $\D F \neq 0$ a.s., where $\D$ here denotes the Malliavin derivative. For vectors of the form $(\I^n(f),\I^n(g))$, Nourdin and Tudor \cite{NT17} showed that the existence of a density is equivalent to the linear independence of $f$ and $g$. Another case is given by considering linearly independent $\mathcal H_1,\ldots,\mathcal H_m\subset\mathcal H$ (by which we mean $h_1 + \ldots + h_m = 0$ with $h_k \in \calH_k$ implies $h_k = 0$ for all $k$), and let $f_k\in\mathcal H_k^{\odot n_k}$ be non-zero for $k=1,\ldots,m$. Then the random variables $\D\I^{n_k}(f_k) \in \cH_k$ are linearly independent. Combined with the fact that each is non-zero, this shows that the Malliavin matrix $(\langle \D\I^{n_i}(f_i), \D\I^{n_j}(f_j)  \rangle_{\calH})_{i,j}$ is non-degenerate and thus that the random vector $(\I^{n_1}(f_1),\ldots,\I^{n_m}(f_m))$ admits a density with respect to Lebesgue measure on $\mathbb R^m$.\\

We now take $n_i \equiv 2$ in \eqref{eq:XWiener} and consider Appell polynomials of second chaos variables. A further simplification of \autoref{prop_moment_cumulant} and \autoref{cor:change_of_chaos} occurs in this case. In this case the operator $\I(D)$ can be expressed in terms of a single operation $\star$ and simpler combinatorial objects, which we explain now.

 For $f\in\mathcal H^{\otimes p}$ and $g\in\mathcal H^{\otimes q}$,  $f\star g\in\mathcal H^{\otimes(p+q-2)}$ is defined by
\begin{align*}
&f\star g\coloneqq{}\\&
\sum_{k,i_2,\ldots,i_p,j_2,\ldots,j_q\geq1}
\bigl\langle f, e_{i_2}\otimes\cdots\otimes e_{i_p}\otimes e_k\bigr\rangle
\bigl\langle g,e_k\otimes e_{j_2}\otimes\cdots\otimes e_{j_q}\bigr\rangle e_{i_2}\otimes\cdots\otimes e_{i_p}\otimes
e_{j_2}\otimes\cdots\otimes e_{j_q},
\end{align*}
for any orthonormal basis $(e_k)_{k\geq1}$ of $\mathcal H$.  As in the preceding proposition, these definitions are independent of the chosen orthonormal basis, and the map $(f,g)\mapsto f\star g$ is bilinear and continuous. Moreover when $f\in\mathcal H^{\odot p}$ and $g\in\mathcal H^{\odot q}$, the expression does not depend on the position where we place $e_k$ and it is  a standard exercise to show that in that case $\star$ coincides with the one contraction $\otimes_1$, see \cite[Section 2.7.4]{NP}. The operation \(\star\) is particularly convenient on
\(\mathcal H^{\otimes2}\). For \(f\in\mathcal H^{\otimes2}\), define
\begin{equation}\label{eq:contract_op}
\A_f\colon\mathcal H\to\mathcal H,
\qquad
\A_fh\coloneqq f\star h.
\end{equation}
Then \(\A_f\) is a Hilbert--Schmidt operator and \(
\|\A_f\|_{\mathrm{HS}}
=
\|f\|_{\mathcal H^{\otimes2}}.
\) Moreover, for every
\(f_1,\ldots,f_m\in\mathcal H^{\otimes2}\),
\[
\A_{(\cdots((f_1\star f_2)\star f_3)\cdots)\star f_m}
=
\A_{f_1}\cdots\A_{f_m}.
\]
Since the correspondence \(f\mapsto\A_f\) is injective, the
associativity of operator composition implies that \(\star\) is an associative, and in general non-commutative, product on
\(\mathcal H^{\otimes2}\). We shall therefore omit parentheses in
iterated \(\star\)-products.Finally, in case $f\in  \mathcal H^{\odot 2}$ the resulting operator $\A_f$ becomes also self-adjoint and one has $\Vert \A_f\Vert_{\mathrm{HS}}
=
\frac{1}{\sqrt{2}}\Vert f\Vert_{\mathcal H^{\odot2}} $ .  For \(u\in\mathcal H^{\otimes2}\), let \(u^\top\) denote the tensor
obtained by exchanging its two variables. Since
\(\A_{u^\top}=\A_u^*\), for
\(f_{j_1},\ldots,f_{j_n}\in\mathcal H^{\odot2}\) one has
\[
\left(f_{j_1}\star\cdots\star f_{j_n}\right)^\top
=
f_{j_n}\star\cdots\star f_{j_1}.
\]
Therefore, exactly as for the symmetrization of a matrix, one has
\[
\operatorname{Sym}_2
\bigl(f_{j_1}\star\cdots\star f_{j_n}\bigr)
=
\frac12\left(
f_{j_1}\star\cdots\star f_{j_n}
+
f_{j_n}\star\cdots\star f_{j_1}
\right).
\]
In particular, the resulting symmetric contraction is unchanged when the orientation of the corresponding path is reversed.

When $f_1,\ldots,f_m\in\mathcal H^{\odot 2}$, every element
$D\in\mathcal D_{\mathrm G,\mathrm{NF},\mathrm C}(f_1,\ldots,f_m)$
can be encoded by a simpler combinatorial object, cf.\ \cite[Ch.\ 4]{PT}. Since each row contains
exactly two nodes, we collapse each row into a single vertex labelled by
$f_i$, so that contractions induce a multigraph with vertex set
$\{f_1,\ldots,f_m\}$; see
\autoref{diagram_to_multigraph} for some examples. In fact, the only multigraph that is not a graph is given by two nodes connected by two edges, but we continue to use this terminology since it is standard and would extend to higher chaoses. In particular, if
$D\in\mathcal P_{\mathrm G,\mathrm{NF},\mathrm C}(f_1,\ldots,f_m)$,
then every node is paired and connectedness implies that the associated
multigraph is a single cycle, with every vertex of degree $2$; we call
such multigraphs cycles over $\{f_1,\ldots,f_m\}$. More generally, if
$D\in\mathcal D_{\mathrm G,\mathrm{NF},\mathrm C}(f_1,\ldots,f_m)$ is
not total, then the associated multigraph has degree at most $2$ at every
vertex and is a disjoint union of path components, including trivial
paths consisting of a single vertex of degree $0$. Each non-trivial path has exactly two vertices of degree $1$, while the remaining vertices have degree $2$. We call such multigraphs linear forests  over
$\{f_1,\ldots,f_m\}$.  We denote by $\mathcal M_{\mathrm C}(f_1,\ldots,f_m)$ the collection of cycles over $\{f_1,\ldots,f_m\}$, by
$\mathcal M_{\mathrm F}(f_1,\ldots,f_m)$  the collection of linear forests  over $\{f_1,\ldots,f_m\}$. The empty forest is included in $\mathcal M_{\mathrm F}(f_1,\ldots,f_m)$. 
\begin{figure}[!ht]
\centering
\tikzset{every picture/.append style={baseline={([yshift=-\the\fontdimen22\textfont2]current bounding box.center)}}}
\begin{align*}
	\begin{tikzpicture}
    \dotrows[rowsep = 5mm, colsep = 4mm]{2,2,2,2}
    \draw (\dotcoord{1,1}) -- (\dotcoord{2,1});
    \draw (\dotcoord{1,2}) -- (\dotcoord{3,1});
    \draw (\dotcoord{2,2}) -- (\dotcoord{4,1});
    \draw (\dotcoord{3,2}) -- (\dotcoord{4,2});
\end{tikzpicture}\quad\xrightarrow \quad \begin{tikzpicture}
		\dotrows[rowsep = 5mm, colsep = 4mm]{1,1,1,1}
		\draw (\dotcoord{1,1}) to (\dotcoord{2,1});
		\draw (\dotcoord{2,1}) to[bend left=50] (\dotcoord{4,1});
		\draw (\dotcoord{3,1}) to (\dotcoord{4,1});
		\draw (\dotcoord{1,1}) to[bend right=50] (\dotcoord{3,1});
	\end{tikzpicture}\qquad  	\qquad \begin{tikzpicture}
    \dotrows[rowsep = 5mm, colsep = 4mm]{2,2,2,2}
    \draw (\dotcoord{1,1}) -- (\dotcoord{2,1});
    \draw (\dotcoord{1,2}) -- (\dotcoord{3,1});
    \draw (\dotcoord{2,2}) -- (\dotcoord{4,1});
\end{tikzpicture}\quad\xrightarrow \quad  \begin{tikzpicture}
    \dotrows[rowsep = 5mm, colsep = 4mm]{1,1,1,1}
    \draw (\dotcoord{1,1}) -- (\dotcoord{2,1});
    \draw (\dotcoord{1,1}) to[bend right=40] (\dotcoord{3,1});
    \draw (\dotcoord{2,1}) to[bend left=40] (\dotcoord{4,1});
\end{tikzpicture} \qquad \qquad \begin{tikzpicture}
    \dotrows[rowsep = 5mm, colsep = 4mm]{2,2,2,2}
      \draw (\dotcoord{1,1}) -- (\dotcoord{2,1});
    \draw (\dotcoord{3,1}) -- (\dotcoord{4,1});
\end{tikzpicture}\quad\xrightarrow \quad \quad \begin{tikzpicture}
    \dotrows[rowsep = 5mm, colsep = 4mm]{1,1,1,1}
     \draw (\dotcoord{1,1}) -- (\dotcoord{2,1});
    \draw (\dotcoord{3,1}) -- (\dotcoord{4,1});
\end{tikzpicture}\qquad \qquad \begin{tikzpicture}
    \dotrows[rowsep = 5mm, colsep = 4mm]{2,2,2,2}
    \draw (\dotcoord{3,1}) -- (\dotcoord{4,1});
\end{tikzpicture}\quad\xrightarrow \quad \quad \begin{tikzpicture}
    \dotrows[rowsep = 5mm, colsep = 4mm]{1,1,1,1}
    \draw (\dotcoord{3,1}) -- (\dotcoord{4,1});
\end{tikzpicture}
\end{align*}
\caption{Contraction of diagrams into multigraphs. The first example is a partition  with its associated cycle and  the other  examples are  diagrams contracting to a closed path and linear forests.}
\label{diagram_to_multigraph}
\end{figure}

The multigraph representation introduced above allows us to encode compositions of operators and contractions in a compact combinatorial way.

\begin{definition}\label{def_conventions_contractions}
Let \(m\geq 2\) and let
\(f_1,\ldots,f_m\in\mathcal H^{\odot 2}\). For each
\(i=1,\ldots,m\), let \(\A_{f_i}\) be the Hilbert--Schmidt operator
associated with \(f_i\) through \eqref{eq:contract_op}. For any \(\Gamma\in\mathcal M_{\mathrm C}(f_1,\ldots,f_m)\), written as \(
\Gamma
=
(f_m\to f_{i_1}\to\cdots\to f_{i_{m-1}}\to f_m)\), we define
\[
\A_\Gamma
:=
\A_{f_{i_{m-1}}}\cdots\A_{f_{i_1}}\A_{f_m}\,.
\]
Now let
\(\Gamma\in\mathcal M_{\mathrm F}(f_1,\ldots,f_m)\), and write its
connected components as \(
\Gamma=\Gamma_1\cup\cdots\cup\Gamma_k.
\)
Orient each non-trivial component from the endpoint lying in the lowest
row towards the other endpoint, and write $\Gamma_l= (f_{j^l_1}\to\cdots\to f_{j^l_{n_l}}),$ for $l=1,\ldots,k$. We then define
\[
f_\Gamma
:=
\bigodot_{l=1}^k
\operatorname{Sym}_2
\bigl(
f_{j^l_1}\star\cdots\star f_{j^l_{n_l}}
\bigr)=
\frac{1}{2^k}\bigodot_{l=1}^k
\bigl(
f_{j^l_1}\star\cdots\star f_{j^l_{n_l}}+f_{j^l_{n_l}}\star\cdots\star f_{j^l_1}
\bigr)
\in\mathcal H^{\odot 2k}\,,
\]
where for an isolated component
\(\Gamma_l=(f_{j^l_1})\), the corresponding factor is simply
\(f_{j^l_1}\).
\end{definition}

These notations allow us to rewrite \autoref{prop_moment_cumulant} and  \autoref{cor:change_of_chaos} in a more compact form, obtaining the following extension of \cite[Proposition 2.7.13]{NP}.
\begin{proposition}\label{prop:second-chaos-cumulants}
Let $m\geq 2$ and $f_1,\ldots,f_m \in \mathcal H^{\odot 2}$. Using the conventions in \autoref{def_conventions_contractions}, one has the identity
\begin{equation}\label{eq:cumulant_second}
\begin{split}
\kappa\big[\I^{2}(f_1),\ldots, \I^{2}(f_m)\big]
&= 2^{m-1} \omega_m\sum_{\Gamma\in \mathcal{M}_{\mathrm{C}}(f_1,\ldots, f_m)} \operatorname{Tr}(\A_\Gamma), \quad \text{with} \;\;\omega_m=\begin{cases}
1, & m=2,\\
2, & m\geq 3.
\end{cases}
\end{split}
\end{equation}
 with the trace $\operatorname{Tr}$ is given in terms of operators over $H$. Moreover, one has \begin{equation}\label{eq:1D_cumulant}
\kappa_m[\I^{2}(f)]= 2^{m-1} (m-1)! \operatorname{Tr}\big(\A_{f}^m\big)
\end{equation}
 and we have the following change of chaos formula:
\begin{equation}\label{eq:change_second}
\begin{split}
y^1 \dia \cdots \dia y^m
\big|_{y^j=\I^{2}(f_j)}
={}&
\sum_{\Gamma\in\mathcal M_{\mathrm F}(f_1,\ldots,f_m)}
2^{|V(\Gamma)|}
\I^{2m-2|E(\Gamma)|}(f_\Gamma),
\end{split}
\end{equation}
where $V(\Gamma)$ and $E(\Gamma)$ denote, respectively, the set of non isolated vertices and edges of $\Gamma$.
\end{proposition}
\begin{remark}
The trace appearing in \eqref{eq:cumulant_second} is well-defined. Indeed, each $\A_{f_i}$ is Hilbert--Schmidt, hence the product of any two of them is trace class, and any further composition with the remaining bounded operators remains trace class. Moreover, $\operatorname{Tr}(\A_\Gamma)$ depends only on the underlying cycle $\Gamma\in\mathcal M_{\mathrm C}(f_1,\ldots,f_m)$. Cyclic invariance of the trace implies that it is independent of the choice of starting vertex, while self-adjointness of the operators $\A_i$ implies invariance under reversal of the orientation. Thus, although the definition of $\A_\Gamma$ requires choosing an orientation of $\Gamma$, its trace is independent of this choice. Similarly, even though $f_{\Gamma}$ depends on the  orientation, final result \eqref{eq:change_second}. The distinction between the cases $m=2$ and $m\geq3$ in the definition of $\omega_m$ comes from the number of diagrams associated with a fixed cycle. For $m\geq3$, each of the $m$ vertices gives two independent choices for attaching its two incident edges to the two nodes of the corresponding row, giving $2^m$ diagrams. When $m=2$, the cycle consists instead of two parallel edges joining the same two vertices, and there are only two distinct pairings of the corresponding nodes. These multiplicities are respectively $2^{m-1}\omega_m$. We can expect more general multigraphs, consisting of arbitrary collections of single nodes, cycles and linear forest, when expressing product formulae using multigraphs instead of diagrams.
\end{remark}

\begin{proof}
We first prove \eqref{eq:cumulant_second}. By
\autoref{prop_moment_cumulant},
\[
\kappa\big[\I^{2}(f_1),\ldots,\I^{2}(f_m)\big]
=
\sum_{D\in
\mathcal P_{\mathrm G,\mathrm{NF},\mathrm C}(f_1,\ldots,f_m)}
\I^0(D).
\]
Every
$D\in\mathcal P_{\mathrm G,\mathrm{NF},\mathrm C}(f_1,\ldots,f_m)$
collapses to a cycle
$\Gamma(D)\in\mathcal M_{\mathrm C}(f_1,\ldots,f_m)$.
Using the definition of $\I^0(D)$ in \eqref{eq:random_diagram}, together with the definition of trace and $\star$, one obtains
\[
\I^0(D)=\operatorname{Tr}\big(\A_{\Gamma(D)}\big).
\]
Now fix a cycle
$\Gamma\in\mathcal M_{\mathrm C}(f_1,\ldots,f_m)$.
If $m\geq3$, at each vertex there are two possible ways of attaching
its two incident edges to the two nodes of the corresponding row. These
choices are independent, so that $\Gamma$ is associated with exactly
$2^m$ diagrams. When $m=2$, the cycle consists of two parallel edges,
and there are exactly two possible pairings of the nodes of the two rows.
Thus, in both cases, the number of diagrams associated with $\Gamma$ is
$2^{m-1}\omega_m$. Summing over
$\Gamma\in\mathcal M_{\mathrm C}(f_1,\ldots,f_m)$ gives
\eqref{eq:cumulant_second}. If $f_1=\cdots=f_m=f$, then
$\operatorname{Tr}(\A_\Gamma)=\operatorname{Tr}(\A_f^m)$ for every cycle
$\Gamma$. For $m=2$ there is a single cycle, whereas for $m\geq3$ there are $(m-1)!/2$ cycles on $m$ labelled vertices, corresponding to cyclic orderings modulo rotations and reversal of orientation. Substituting into
\eqref{eq:cumulant_second} yields \eqref{eq:1D_cumulant}. We finally prove \eqref{eq:change_second}. By
\autoref{cor:change_of_chaos},
\[
y^1\dia\cdots\dia y^m
\big|_{y^j=\I^2(f_j)}
=
\sum_{\substack{
D\in\mathcal D_{\mathrm G,\mathrm{NF},\mathrm C}
(f_1,\ldots,f_m)\\
|K|>0}}
\I(D).
\]
Every diagram containing at least one contraction collapses to a linear forest \(\Gamma(D)\in\mathcal M_{\mathrm F}(f_1,\ldots,f_m)\). If \(E(\Gamma(D))\) denotes the set of edges of \(\Gamma(D)\), then the
diagram contains exactly \(|E(\Gamma(D))|\) contractions and hence has
\(2m-2|E(\Gamma(D))|\) free nodes. By applying the contraction
\(\star\) successively along each connected component of
\(\Gamma(D)\) and symmetrising it, one obtains
\[
\I(D)
=
\I^{2m-2|E(\Gamma(D))|}
\bigl(
f_{\Gamma(D)}
\bigr).
\]
It remains to determine how many diagrams give rise to the same linear forest. Fix
\(\Gamma\in\mathcal M_{\mathrm F}(f_1,\ldots,f_m)\), and let
\(V(\Gamma)\) denote the set of vertices visited by \(\Gamma\), that is,
the vertices incident to at least one edge. At every vertex in
\(V(\Gamma)\), there are exactly two possible ways of assigning the
incident edge or edges to the two nodes of the corresponding row. These
choices are independent. By contrast, a vertex not visited by
\(\Gamma\) corresponds to a row with two free nodes and gives no
additional choice. Hence precisely \(2^{|V(\Gamma)|}
\) diagrams collapse to the same linear forest \(\Gamma\). Grouping the terms in \autoref{cor:change_of_chaos} according to their associated linear forest
therefore yields \eqref{eq:change_second}.
\end{proof}

\begin{example}
We provide some examples to compute cumulants and change of chaos for three variables $\I^2(f)$, $\I^2(g)$, $\I^2(h) \in \mathcal{W}^{2}$. We list the diagrams, the multigraphs and the final terms that contribute additively to each term. Each row is intended to be decorated with the corresponding kernel ($f$, $g$, $h$).
\tikzset{every picture/.append style={baseline={([yshift=-\the\fontdimen22\textfont2]current bounding box.center)}}}

\noindent $\kappa[\I^2(f), \I^2(g)]$:
	\begin{align*}
		&\begin{tikzpicture}
			\dotrows[rowsep = 5mm, colsep = 4mm]{2,2}
			\draw (\dotcoord{1,1}) to (\dotcoord{2,1});
			\draw (\dotcoord{1,2}) to (\dotcoord{2,2});
		\end{tikzpicture}\qquad
		\begin{tikzpicture}
			\dotrows[rowsep = 5mm, colsep = 4mm]{2,2}
			\draw (\dotcoord{1,1}) to (\dotcoord{2,2});
			\draw (\dotcoord{1,2}) to (\dotcoord{2,1});
		\end{tikzpicture}
		\quad\longrightarrow\quad
		2\,\begin{tikzpicture}
    \dotrows[rowsep = 5mm, colsep = 4mm]{1,1}
    \draw (\dotcoord{1,1}) to[bend left=50] (\dotcoord{2,1});
    \draw (\dotcoord{1,1}) to (\dotcoord{2,1});
\end{tikzpicture}
\quad\leadsto\quad
2\operatorname{Tr}\big(\A_{f}\A_{g}\big).
	\end{align*}

	\noindent $\kappa[\I^2(f), \I^2(g), \I^2(h)]$:
	\begin{align*}
		&\begin{tikzpicture}
			\dotrows[rowsep = 5mm, colsep = 4mm]{2,2,2}
			\draw (\dotcoord{1,1}) to (\dotcoord{2,1});
			\draw (\dotcoord{1,2}) to (\dotcoord{3,1});
			\draw (\dotcoord{2,2}) to (\dotcoord{3,2});
		\end{tikzpicture}\qquad
		\begin{tikzpicture}
			\dotrows[rowsep = 5mm, colsep = 4mm]{2,2,2}
			\draw (\dotcoord{1,1}) to (\dotcoord{2,1});
			\draw (\dotcoord{1,2}) to[bend right=30] (\dotcoord{3,2});
			\draw (\dotcoord{2,2}) to (\dotcoord{3,1});
		\end{tikzpicture}
		\qquad
		\begin{tikzpicture}
			\dotrows[rowsep = 5mm, colsep = 4mm]{2,2,2}
			\draw (\dotcoord{1,1}) to (\dotcoord{2,2});
			\draw (\dotcoord{1,2}) to (\dotcoord{3,1});
			\draw (\dotcoord{2,1}) to (\dotcoord{3,2});
		\end{tikzpicture}
		\qquad
		\begin{tikzpicture}
			\dotrows[rowsep = 5mm, colsep = 4mm]{2,2,2}
			\draw (\dotcoord{1,1}) to (\dotcoord{2,2});
			\draw (\dotcoord{1,2}) to[bend right=30] (\dotcoord{3,2});
			\draw (\dotcoord{2,1}) to (\dotcoord{3,1});
		\end{tikzpicture}\qquad
		\begin{tikzpicture}
			\dotrows[rowsep = 5mm, colsep = 4mm]{2,2,2}
			\draw (\dotcoord{1,2}) to (\dotcoord{2,1});
			\draw (\dotcoord{1,1}) to[bend left=30] (\dotcoord{3,1});
			\draw (\dotcoord{2,2}) to (\dotcoord{3,2});
		\end{tikzpicture}
\qquad	\begin{tikzpicture}
			\dotrows[rowsep = 5mm, colsep = 4mm]{2,2,2}
			\draw (\dotcoord{1,2}) to (\dotcoord{2,2});
			\draw (\dotcoord{1,1}) to[bend left=30] (\dotcoord{3,1});
			\draw (\dotcoord{2,1}) to (\dotcoord{3,2});
		\end{tikzpicture}
		\qquad
		\begin{tikzpicture}
			\dotrows[rowsep = 5mm, colsep = 4mm]{2,2,2}
			\draw (\dotcoord{1,1}) to (\dotcoord{3,2});
			\draw (\dotcoord{1,2}) to (\dotcoord{2,1});
			\draw (\dotcoord{2,2}) to (\dotcoord{3,1});
		\end{tikzpicture}
\qquad		\begin{tikzpicture}
			\dotrows[rowsep = 5mm, colsep = 4mm]{2,2,2}
			\draw (\dotcoord{1,1}) to (\dotcoord{3,2});
			\draw (\dotcoord{1,2}) to (\dotcoord{2,2});
			\draw (\dotcoord{2,1}) to (\dotcoord{3,1});
		\end{tikzpicture}
		\quad\longrightarrow\quad
		8\,\begin{tikzpicture}
		\dotrows[rowsep = 5mm, colsep = 4mm]{1,1,1}
		\draw (\dotcoord{1,1}) to (\dotcoord{2,1});
		\draw (\dotcoord{2,1}) to (\dotcoord{3,1});
		\draw (\dotcoord{1,1}) to[bend right=50] (\dotcoord{3,1});
	\end{tikzpicture}
	\quad\leadsto\quad
	8\operatorname{Tr}\big(\A_{f}\A_{g}\A_h\big)
	\end{align*}
	The connectedness condition is moot in the case of cumulants (keeping non-flatness into account), since for total diagrams it only becomes relevant with four or more rows. For partial diagrams, it already becomes relevant at level three. Passing to the change of chaos we also add the empty diagram corresponding to $\I^{2m}(f_1\odot\cdots\odot f_m)$.\\
	\noindent $\I^2(f) \dia \I^2(g)$:
	\begin{align*}
		&
		\begin{tikzpicture}
			\dotrows[rowsep = 5mm, colsep = 4mm]{2,2}
		\end{tikzpicture}
		\qquad
		\begin{tikzpicture}
			\dotrows[rowsep = 5mm, colsep = 4mm]{2,2}
			\draw (\dotcoord{1,1}) to (\dotcoord{2,1});
		\end{tikzpicture}
		\qquad
		\begin{tikzpicture}
			\dotrows[rowsep = 5mm, colsep = 4mm]{2,2}
			\draw (\dotcoord{1,1}) to (\dotcoord{2,2});
		\end{tikzpicture}
		\qquad
		\begin{tikzpicture}
			\dotrows[rowsep = 5mm, colsep = 4mm]{2,2}
			\draw (\dotcoord{1,2}) to (\dotcoord{2,1});
		\end{tikzpicture}
		\qquad
		\begin{tikzpicture}
			\dotrows[rowsep = 5mm, colsep = 4mm]{2,2}
			\draw (\dotcoord{1,2}) to (\dotcoord{2,2});
		\end{tikzpicture}
		\quad\longrightarrow\quad
		\begin{tikzpicture}
			\dotrows[rowsep = 5mm, colsep = 4mm]{1,1}
		\end{tikzpicture}\qquad
		4\,\begin{tikzpicture}
    \dotrows[rowsep = 5mm, colsep = 4mm]{1,1}
    \draw (\dotcoord{1,1}) to (\dotcoord{2,1});
\end{tikzpicture}
\quad\leadsto\quad
\I^4(f\odot g) + 4\I^2(\operatorname{Sym}_2(f\star g))\,.
	\end{align*}
	$\I^2(f) \dia \I^2(g) \dia \I^2(h)$:
	\begin{align*}
		&\begin{tikzpicture}
			\dotrows[rowsep = 5mm, colsep = 4mm]{2,2,2}
		\end{tikzpicture}
		\qquad
		\begin{tikzpicture}
			\dotrows[rowsep = 5mm, colsep = 4mm]{2,2,2}
			\draw (\dotcoord{1,1}) to (\dotcoord{2,1});
		\end{tikzpicture}
		\qquad
		\begin{tikzpicture}
			\dotrows[rowsep = 5mm, colsep = 4mm]{2,2,2}
			\draw (\dotcoord{1,1}) to (\dotcoord{2,2});
		\end{tikzpicture}
		\qquad
		\begin{tikzpicture}
			\dotrows[rowsep = 5mm, colsep = 4mm]{2,2,2}
			\draw (\dotcoord{1,1}) to[bend left=30] (\dotcoord{3,1});
		\end{tikzpicture}
		\qquad
		\begin{tikzpicture}
			\dotrows[rowsep = 5mm, colsep = 4mm]{2,2,2}
			\draw (\dotcoord{1,1}) to (\dotcoord{3,2});
		\end{tikzpicture}
		\qquad
		\begin{tikzpicture}
			\dotrows[rowsep = 5mm, colsep = 4mm]{2,2,2}
			\draw (\dotcoord{1,2}) to (\dotcoord{2,1});
		\end{tikzpicture}
		\qquad
		\begin{tikzpicture}
			\dotrows[rowsep = 5mm, colsep = 4mm]{2,2,2}
			\draw (\dotcoord{1,2}) to (\dotcoord{2,2});
		\end{tikzpicture}
		\qquad
		\begin{tikzpicture}
			\dotrows[rowsep = 5mm, colsep = 4mm]{2,2,2}
			\draw (\dotcoord{1,2}) to (\dotcoord{3,1});
		\end{tikzpicture}
		\qquad
		\begin{tikzpicture}
			\dotrows[rowsep = 5mm, colsep = 4mm]{2,2,2}
			\draw (\dotcoord{1,2}) to[bend right=30] (\dotcoord{3,2});
		\end{tikzpicture}
		\qquad
		\begin{tikzpicture}
			\dotrows[rowsep = 5mm, colsep = 4mm]{2,2,2}
			\draw (\dotcoord{2,1}) to (\dotcoord{3,1});
		\end{tikzpicture}
		\qquad
		\begin{tikzpicture}
			\dotrows[rowsep = 5mm, colsep = 4mm]{2,2,2}
			\draw (\dotcoord{2,1}) to (\dotcoord{3,2});
		\end{tikzpicture}
		\qquad
		\begin{tikzpicture}
			\dotrows[rowsep = 5mm, colsep = 4mm]{2,2,2}
			\draw (\dotcoord{2,2}) to (\dotcoord{3,1});
		\end{tikzpicture}
		\qquad
		\begin{tikzpicture}
			\dotrows[rowsep = 5mm, colsep = 4mm]{2,2,2}
			\draw (\dotcoord{2,2}) to (\dotcoord{3,2});
		\end{tikzpicture}\\[3ex]
		&\begin{tikzpicture}
			\dotrows[rowsep = 5mm, colsep = 4mm]{2,2,2}
			\draw (\dotcoord{1,1}) to (\dotcoord{2,1});
			\draw (\dotcoord{1,2}) to (\dotcoord{3,1});
		\end{tikzpicture}
		\qquad
		\begin{tikzpicture}
			\dotrows[rowsep = 5mm, colsep = 4mm]{2,2,2}
			\draw (\dotcoord{1,1}) to (\dotcoord{2,1});
			\draw (\dotcoord{1,2}) to[bend right=30] (\dotcoord{3,2});
		\end{tikzpicture}
		\qquad
		\begin{tikzpicture}
			\dotrows[rowsep = 5mm, colsep = 4mm]{2,2,2}
			\draw (\dotcoord{1,1}) to (\dotcoord{2,2});
			\draw (\dotcoord{1,2}) to (\dotcoord{3,1});
		\end{tikzpicture}
		\qquad
		\begin{tikzpicture}
			\dotrows[rowsep = 5mm, colsep = 4mm]{2,2,2}
			\draw (\dotcoord{1,1}) to (\dotcoord{2,2});
			\draw (\dotcoord{1,2}) to[bend right=30] (\dotcoord{3,2});
		\end{tikzpicture}
		\qquad
		\begin{tikzpicture}
			\dotrows[rowsep = 5mm, colsep = 4mm]{2,2,2}
			\draw (\dotcoord{1,2}) to (\dotcoord{2,1});
			\draw (\dotcoord{1,1}) to[bend left=30] (\dotcoord{3,1});
		\end{tikzpicture}
		\qquad
		\begin{tikzpicture}
			\dotrows[rowsep = 5mm, colsep = 4mm]{2,2,2}
			\draw (\dotcoord{1,2}) to (\dotcoord{2,1});
			\draw (\dotcoord{1,1}) to (\dotcoord{3,2});
		\end{tikzpicture}
		\qquad
		\begin{tikzpicture}
			\dotrows[rowsep = 5mm, colsep = 4mm]{2,2,2}
			\draw (\dotcoord{1,2}) to (\dotcoord{2,2});
			\draw (\dotcoord{1,1}) to[bend left=30] (\dotcoord{3,1});
		\end{tikzpicture}
		\qquad
		\begin{tikzpicture}
			\dotrows[rowsep = 5mm, colsep = 4mm]{2,2,2}
			\draw (\dotcoord{1,2}) to (\dotcoord{2,2});
			\draw (\dotcoord{1,1}) to (\dotcoord{3,2});
		\end{tikzpicture}\\[3ex]
		&\begin{tikzpicture}
			\dotrows[rowsep = 5mm, colsep = 4mm]{2,2,2}
			\draw (\dotcoord{1,1}) to (\dotcoord{2,1});
			\draw (\dotcoord{2,2}) to (\dotcoord{3,1});
		\end{tikzpicture}
		\qquad
		\begin{tikzpicture}
			\dotrows[rowsep = 5mm, colsep = 4mm]{2,2,2}
			\draw (\dotcoord{1,1}) to (\dotcoord{2,1});
			\draw (\dotcoord{2,2}) to (\dotcoord{3,2});
		\end{tikzpicture}
		\qquad
		\begin{tikzpicture}
			\dotrows[rowsep = 5mm, colsep = 4mm]{2,2,2}
			\draw (\dotcoord{1,2}) to (\dotcoord{2,1});
			\draw (\dotcoord{2,2}) to (\dotcoord{3,1});
		\end{tikzpicture}
		\qquad
		\begin{tikzpicture}
			\dotrows[rowsep = 5mm, colsep = 4mm]{2,2,2}
			\draw (\dotcoord{1,2}) to (\dotcoord{2,1});
			\draw (\dotcoord{2,2}) to (\dotcoord{3,2});
		\end{tikzpicture}
		\qquad
		\begin{tikzpicture}
			\dotrows[rowsep = 5mm, colsep = 4mm]{2,2,2}
			\draw (\dotcoord{2,1}) to (\dotcoord{3,1});
			\draw (\dotcoord{2,2}) to (\dotcoord{1,1});
		\end{tikzpicture}
		\qquad
		\begin{tikzpicture}
			\dotrows[rowsep = 5mm, colsep = 4mm]{2,2,2}
			\draw (\dotcoord{2,1}) to (\dotcoord{3,1});
			\draw (\dotcoord{2,2}) to (\dotcoord{1,2});
		\end{tikzpicture}
		\qquad
		\begin{tikzpicture}
			\dotrows[rowsep = 5mm, colsep = 4mm]{2,2,2}
			\draw (\dotcoord{2,1}) to (\dotcoord{3,2});
			\draw (\dotcoord{2,2}) to (\dotcoord{1,1});
		\end{tikzpicture}
		\qquad
		\begin{tikzpicture}
			\dotrows[rowsep = 5mm, colsep = 4mm]{2,2,2}
			\draw (\dotcoord{2,1}) to (\dotcoord{3,2});
			\draw (\dotcoord{2,2}) to (\dotcoord{1,2});
		\end{tikzpicture}\\[3ex]
		&\begin{tikzpicture}
			\dotrows[rowsep = 5mm, colsep = 4mm]{2,2,2}
			\draw (\dotcoord{3,1}) to[bend right=30] (\dotcoord{1,1});
			\draw (\dotcoord{3,2}) to (\dotcoord{2,1});
		\end{tikzpicture}
		\qquad
		\begin{tikzpicture}
			\dotrows[rowsep = 5mm, colsep = 4mm]{2,2,2}
			\draw (\dotcoord{3,1}) to[bend right=30] (\dotcoord{1,1});
			\draw (\dotcoord{3,2}) to (\dotcoord{2,2});
		\end{tikzpicture}
		\qquad
		\begin{tikzpicture}
			\dotrows[rowsep = 5mm, colsep = 4mm]{2,2,2}
			\draw (\dotcoord{3,1}) to (\dotcoord{1,2});
			\draw (\dotcoord{3,2}) to (\dotcoord{2,1});
		\end{tikzpicture}
		\qquad
		\begin{tikzpicture}
			\dotrows[rowsep = 5mm, colsep = 4mm]{2,2,2}
			\draw (\dotcoord{3,1}) to (\dotcoord{1,2});
			\draw (\dotcoord{3,2}) to (\dotcoord{2,2});
		\end{tikzpicture}
		\qquad
		\begin{tikzpicture}
			\dotrows[rowsep = 5mm, colsep = 4mm]{2,2,2}
			\draw (\dotcoord{3,1}) to (\dotcoord{2,1});
			\draw (\dotcoord{3,2}) to (\dotcoord{1,1});
		\end{tikzpicture}
		\qquad
		\begin{tikzpicture}
			\dotrows[rowsep = 5mm, colsep = 4mm]{2,2,2}
			\draw (\dotcoord{3,1}) to (\dotcoord{2,1});
			\draw (\dotcoord{3,2}) to[bend left=30] (\dotcoord{1,2});
		\end{tikzpicture}
		\qquad
		\begin{tikzpicture}
			\dotrows[rowsep = 5mm, colsep = 4mm]{2,2,2}
			\draw (\dotcoord{3,1}) to (\dotcoord{2,2});
			\draw (\dotcoord{3,2}) to (\dotcoord{1,1});
		\end{tikzpicture}
		\qquad
		\begin{tikzpicture}
			\dotrows[rowsep = 5mm, colsep = 4mm]{2,2,2}
			\draw (\dotcoord{3,1}) to (\dotcoord{2,2});
			\draw (\dotcoord{3,2}) to[bend left=30] (\dotcoord{1,2});
		\end{tikzpicture}
		\qquad
	\end{align*}
	\begin{align*}
	&\longrightarrow\qquad
	\begin{tikzpicture}
		\dotrows[rowsep=6mm,colsep=0mm]{1,1,1}
	\end{tikzpicture}
	\qquad
	4\,\begin{tikzpicture}
		\dotrows[rowsep=6mm,colsep=0mm]{1,1,1}
		\draw[bend left=50] (\dotcoord{1,1}) to (\dotcoord{2,1});
	\end{tikzpicture}
	\qquad
	4\,\begin{tikzpicture}
		\dotrows[rowsep=6mm,colsep=0mm]{1,1,1}
		\draw[bend left=80] (\dotcoord{1,1}) to (\dotcoord{3,1});
	\end{tikzpicture}
	\qquad
	4\,\begin{tikzpicture}
		\dotrows[rowsep=6mm,colsep=0mm]{1,1,1}
		\draw[bend right=50] (\dotcoord{2,1}) to (\dotcoord{3,1});
	\end{tikzpicture}\qquad
	8\,\begin{tikzpicture}
		\dotrows[rowsep=6mm,colsep=0mm]{1,1,1}
		\draw (\dotcoord{1,1}) -- (\dotcoord{2,1});
		\draw (\dotcoord{2,1}) -- (\dotcoord{3,1});
	\end{tikzpicture}
	\qquad
	8\,\begin{tikzpicture}
		\dotrows[rowsep=6mm,colsep=0mm]{1,1,1}
		\draw[bend right=40] (\dotcoord{2,1}) to (\dotcoord{1,1});
		\draw[bend right=40] (\dotcoord{1,1}) to (\dotcoord{3,1});
	\end{tikzpicture}
	\qquad
	8\,\begin{tikzpicture}
		\dotrows[rowsep=6mm,colsep=0mm]{1,1,1}
		\draw[bend right=40] (\dotcoord{1,1}) to (\dotcoord{3,1});
		\draw[bend right=40] (\dotcoord{3,1}) to (\dotcoord{2,1});
	\end{tikzpicture}
	\end{align*}

	\begin{align*}
	&\leadsto\quad
	\I^6(f\odot g\odot h)
	+ 4\I^4(\operatorname{Sym}_2(f\star g)\odot h)
	+ 4\I^4(\operatorname{Sym}_2(f\star h)\odot g)
	+ 4\I^4(\operatorname{Sym}_2(g\star h)\odot f )
	\\
	&\hphantom{\leadsto\quad}
	+ 8\I^2(\operatorname{Sym}_2(f\star g\star h))
	+ 8\I^2(\operatorname{Sym}_2(h\star f\star g))
	+ 8\I^2(\operatorname{Sym}_2(f\star h\star g))\,.
	\end{align*}
\end{example}

We now consider series of Appell polynomials w.r.t.\ second chaos variables. The following proposition shows that the conditions of \autoref{thm:martin} are satisfied with an explicit choice of the radius.

\begin{proposition}\label{prop_suff_cond_second_chaos}
	Let $\mu$ be the law of a random vector in the second Wiener chaos $X = (\I^2(f^1), \ldots, \I^2(f^d))$ with $f^k \in \mathcal{H}^{\odot 2}$. Then $\M$ is defined and non-zero on the open polydisc $D_R$ for any $R\in(0,\infty]^d$ such that
		\begin{equation}\label{eq_condition_R}
			\sum_{k=1}^d R_k \|f^k\|_{\mathcal H^{\odot2}} < \frac{1}{\sqrt{2}}\,.
		\end{equation}
\end{proposition}

\begin{proof}
	Let $R$ be a generic polyradius satisfying \eqref{eq_condition_R}. Fix $r<R$, and, for $\sigma=(\sigma_1,\dots,\sigma_d)\in\{\pm1\}^d$, define $\A_{\sigma,r}:=\sum_{k=1}^d \sigma_k r_k \A_{f^k}$. Using the hypothesis on $R$ and the identity $\|\A_f\|_{\mathrm{HS}}=2^{-1/2}\|f\|_{\mathcal H^{\odot2}}$, one has
	\[
	\|\A_{\sigma,r}\|_{\mathrm{op}}
	\le \|\A_{\sigma,r}\|_{\mathrm{HS}}
	\le \frac{1}{\sqrt{2}}\sum_{k=1}^d r_k \|f_k\|_{\mathcal H^{\odot2}}
	<\frac12,
	\]
	and hence all eigenvalues of the self-adjoint Hilbert--Schmidt operator
	$\A_{\sigma,r}$ belong to $(-1/2,1/2)$. By linearity one has
	$\sum_{k=1}^d \sigma_k r_k X^k=\I^2\!\left(\sum_{k=1}^d \sigma_k r_k f^k\right)$ and, using the spectral representation of second-chaos random variables (see, e.g.,
	\cite[Proposition~2.7.13]{NP}), there exist i.i.d.\ standard Gaussian random
	variables $(\xi_n)_{n\ge1}$ such that
	\[
	\sum_{k=1}^d \sigma_k r_k X^k
	=
	\sum_{n\ge1}\lambda_n(\sigma,r)(\xi_n^2-1),
	\]
	where $(\lambda_n(\sigma,r))_{n\ge1}$ are the eigenvalues of $\A_{\sigma,r}$. Hence we obtain
	\[
	\mathbb E \exp\Bigl(\sum_{k=1}^d r_k |X^k|\Bigr)
	\le
	\sum_{\sigma\in\{\pm1\}^d}
	\mathbb E
	\exp\Bigl(\sum_{k=1}^d \sigma_k r_k X^k\Bigr)
	=
	\sum_{\sigma\in\{\pm1\}^d}\prod_{n\ge1}
	\mathbb E
	\exp\bigl(\lambda_n(\sigma,r)(\xi_n^2-1)\bigr).
	\]
	Because $|\lambda_n(\sigma,r)|<1/2$, $\mathbb E\exp\bigl(\lambda_n(\sigma,r)(\xi_n^2-1)\bigr)=e^{-\lambda_n(\sigma,r)}(1-2\lambda_n(\sigma,r))^{-1/2}$,
	and thus
	\[
	\mathbb E \exp\Bigl(\sum_{k=1}^d r_k |X^k|\Bigr)\leq\sum_{\sigma\in\{\pm1\}^d}\prod_{n\ge1}
	e^{-\lambda_n(\sigma,r)}(1-2\lambda_n(\sigma,r))^{-1/2}.
	\]
	Using the standard approximation
	\[
	e^{-\lambda}(1-2\lambda)^{-1/2}=1+O(\lambda^2) \qquad (\lambda\to0),
	\]
	and the fact that $(\lambda_n(\sigma,r))_{n\geq1}\in\ell^2$, the infinite product converges to a finite positive number. Consequently, $\M$ is well defined and holomorphic on $D_R$.	Now let \(z\in D_R\), work on the complexification of \(\mathcal H\), and set
\[
\A_z:=\sum_{k=1}^d z_k\A_{f^k}.
\]
Then \(\A_z\) is Hilbert--Schmidt and
\[
\|\A_z\|_{\mathrm{op}}
\leq \|\A_z\|_{\mathrm{HS}}
\leq \frac{1}{\sqrt{2}}
\sum_{k=1}^d |z_k|\|f^k\|_{\mathcal H^{\odot2}}
<\frac12.
\]
Define
\[
\Delta(z)
:=
\exp\left\{
-\sum_{m=2}^{\infty}
\frac{2^m}{m}\operatorname{Tr}(\A_z^m)
\right\}.
\]
This is well defined and holomorphic on \(D_R\). Indeed, for \(m\geq2\),
\(\A_z^m\) is trace class and
\[
|\operatorname{Tr}(\A_z^m)|
\leq
\|\A_z^2\|_1\|\A_z\|_{\mathrm{op}}^{m-2}
\leq
\|\A_z\|_{\mathrm{HS}}^2
\|\A_z\|_{\mathrm{op}}^{m-2}.
\]
Hence, on every compact set \(K\subset D_R\), the defining series for \(\Delta\) converges uniformly. Since \(z\mapsto\operatorname{Tr}(\A_z^m)\) is a
polynomial for every \(m\geq2\), it follows that \(\Delta\) is holomorphic.
Moreover, \(\Delta(z)\neq0\) on \(D_R\). For \(z\in D_R\cap\mathbb R^d\), the operator \(\A_z\) is self-adjoint.
Writing \((\mu_n(z))_{n\geq1}\) for its eigenvalues, the spectral
representation gives
\[
\EuScript M(z)^2
=
\prod_{n\geq1}
e^{-2\mu_n(z)}(1-2\mu_n(z))^{-1}.
\]
On the other hand, since
\(\sup_n 2|\mu_n(z)|<1\) and
\((\mu_n(z))_{n\geq1}\in\ell^2\),
\[
\begin{aligned}
\Delta(z)
&=
\exp\left\{
-\sum_{m=2}^{\infty}
\frac{2^m}{m}\sum_{n\geq1}\mu_n(z)^m
\right\} =
\prod_{n\geq1}
(1-2\mu_n(z))e^{2\mu_n(z)}.
\end{aligned}
\]
Consequently, $
\EuScript M(z)^2\Delta(z)=1$ for $ z\in D_R\cap\mathbb R^d$. Since both sides are holomorphic on \(D_R\), the identity theorem, applied successively in each variable, yields $\EuScript M(z)^2\Delta(z)=1$, for $z\in D_R$.Since \(\Delta\) is nowhere zero, \(\EuScript M(z)\neq0\) for every \(z\in D_R\).
\end{proof}

\section{Processes in the second chaos and the Rosenblatt process}\label{sec:2chaos_processes}

We now return to stochastic processes, specifically ones that take values in the second chaos. Let $X\colon [0,T]\times\Omega\to\mathbb R^d$ be a $d$-dimensional stochastic process. We assume that there exists a centred isonormal Gaussian process over a real separable Hilbert space $\mathcal H$ such that, for every $i\in \{1,\ldots,d\}$ and $t\in[0,T]$, $X_t^i=\I^2(f_t^i)$, where $f^i\colon [0,T]\to\mathcal H^{\odot2}$. In what follows, we will denote by $\A_t^i$ the Hilbert--Schmidt operator associated with $f_t^i$ in \eqref{eq:contract_op}. Moreover, we will need to consider $\D X=(\D X^1,\ldots,\D X^d)\colon [0,T]\times\Omega\times \mathcal H \to \mathbb{R}^d$, the Malliavin derivative of the process $X$, see \cite[Definition 2.3.2]{NP}, which is explicitly characterised, for every $i\in\{1,\ldots,d\}$ and $h\in\mathcal H$, by the relation
\begin{equation*}
\langle \D X_t^i,h\rangle
=
2\,\I^1(\A_t^i h)\,,
\end{equation*}
see \cite[Proposition 2.7.4]{NP}. For $t\in[0,T]$ and $i,j\in\{1,\ldots,d\}$, we also introduce the Malliavin covariance kernel, given by the random variable
\[
\langle \D X_t^i,\D X_t^j\rangle
=
4\,\operatorname{Tr}(\A_t^i\A_t^j)
+
2\I^2\!\bigl(f_t^i\star f_t^j+f_t^j\star f_t^i\bigr)\,.
\]
We now provide sufficient conditions on $f= (f^1,\ldots,f^d)\colon [0,T]\to (\mathcal H^{\odot2})^d$ for \autoref{conditions} to be satisfied.
\begin{proposition}\label{thm_suff_second}
Assume that the following properties hold true:
\begin{enumerate}
\item There exist $C>0$ and $\alpha>\frac12$ such that
\begin{equation}\label{first_cond_chaos}
\sum_{i=1}^d
\|f_t^i-f_s^i\|_{\mathcal H^{\odot 2}}^2
\le C|t-s|^{2\alpha},
\qquad s,t\in[0,T].
\end{equation}

\item For every $m\ge2$ and every $i_1,\ldots,i_m\in\{1,\ldots,d\}$, the trace coefficient
\[
\tau_{i_1,\ldots,i_m}(t_1,\ldots,t_m)
:=
\operatorname{Tr}\bigl(\A_{t_1}^{i_1}\cdots \A_{t_m}^{i_m}\bigr)
\]
is continuously differentiable on $[0,T]^m$, and there exist $\lambda>0$ and constants $C_m\ge0$ such that
\begin{equation}\label{second_cond_chaos_2}
\begin{split}
\sup_{\substack{t_1,\ldots,t_m\in[0,T]\\i_1,\ldots,i_m\\1\le r\le m}}
\left|
\partial_{t_r}
\tau_{i_1,\ldots,i_m}(t_1,\ldots,t_m)
\right|
&\le C_m\,,
\qquad
\sum_{m\ge2}\frac{2^{m-1}}{m}\,C_m\lambda^m<\infty.
\end{split}
\end{equation}

\item For every $t\in]0,T]$ the Malliavin matrix
$
\Gamma_t
=
\bigl(
\langle \D X_t^i,\D X_t^j\rangle_{\mathcal H}
\bigr)_{1\le i,j\le d}
$
satisfies
\begin{equation}\label{third_cond_chaos}
\mathbb E\bigl[\det\Gamma_t\bigr]>0.
\end{equation}
\end{enumerate}
Then \autoref{conditions} are satisfied for any polyradius $R$ such that
\begin{equation}\label{eq_choice_poly_R}
\begin{split}
|R|_\infty=\sup_{i=1,\ldots,d}R_i
&<
\min\left\{
\frac{1}{2\sqrt{2}M},
\frac{\lambda}{2d}
\right\}
\,\quad \text{with} \;M
:=
\sup_{t\in[0,T]}
\sum_{i=1}^d
\|f_t^i\|_{\mathcal H^{\odot2}}\,.
\end{split}
\end{equation}
\end{proposition}
\begin{proof}
Each condition in the statement will imply the corresponding property in \autoref{conditions}. Concerning property $1$, the reasoning is perfectly classical. Using the isometry of multiple Wiener integrals and hypercontractivity (see, e.g., \cite[Theorem~2.7.2]{NP}), for every $q\ge2$ there exists a constant $C_q>0$ such that
\[
\E\bigl[|X_t-X_s|^q\bigr]
\leq
C_q
\left(
\sum_{i=1}^d
\|f_t^i-f_s^i\|_{\mathcal H^{\odot2}}^2
\right)^{\frac{q}{2}}
\leq
C_q C^{q/2}
|t-s|^{q\alpha}.
\]
Choosing $q$ sufficiently large and applying Kolmogorov's continuity theorem,
$X$ admits a modification which is $\gamma$-Hölder continuous for some
$\gamma>\frac12$. Hence this modification has finite $p$-variation for every
$p>1/\gamma$, and in particular for some $p<2$. We now verify condition $2$. First, notice that \eqref{first_cond_chaos} implies
\[
\sum_{i=1}^d\|f_t^i\|_{\mathcal H^{\odot2}}
\leq
\sum_{i=1}^d\|f_0^i\|_{\mathcal H^{\odot2}}
+
\sqrt{d}
\left(
\sum_{i=1}^d
\|f_t^i-f_0^i\|_{\mathcal H^{\odot2}}^2
\right)^{1/2}
\leq
\sum_{i=1}^d\|f_0^i\|_{\mathcal H^{\odot2}}
+
\sqrt{dC}\,T^\alpha,
\]
and therefore $M<\infty$. For any $s,t\in[0,T]$ and
$z,w\in\mathbb C^d$, write $z\cdot X_s+w\cdot X_t
=
\I^2(z\cdot f_s+w\cdot f_t)$,
and let $B_{s,t}(z,w):=\sum_{i=1}^d z_i\A_s^i+\sum_{i=1}^d w_i\A_t^i$
be the Hilbert--Schmidt operator associated with
$z\cdot f_s+w\cdot f_t$. Using
$\|\A_f\|_{\mathrm{HS}}=2^{-1/2}\|f\|_{\mathcal H^{\odot2}}$, we easily obtain, for any polyradius $\rho$ and any $z,w\in D_\rho$,
\[
\|B_{s,t}(z,w)\|_{\mathrm{HS}}
\leq
\sqrt{2}M\sup_{i=1,\ldots,d}\rho_i\,.
\]
We choose now a polyradius $R=(R_1,\ldots,R_d)$ according to \eqref{eq_choice_poly_R}. By \autoref{prop_suff_cond_second_chaos}, applied to the $2d$-dimensional random vector $(X_s,X_t)$, its moment generating function is defined and non-zero on $D_R\times D_R$, uniformly in $s,t\in[0,T]$. Indeed,
\[
\sum_{i=1}^d
R_i
\bigl(
\|f_s^i\|_{\mathcal H^{\odot2}}
+
\|f_t^i\|_{\mathcal H^{\odot2}}
\bigr)
\leq
2|R|_\infty M
<
\frac{1}{\sqrt2}.
\]
It remains to prove the required regularity of the corresponding cumulant-generating function. From the computations of the cumulants of the single random variable $\I^2(z\cdot f_s+w\cdot f_t)$ in \eqref{eq:1D_cumulant}, one has, initially in a neighbourhood of the origin,
\[
\K_{s,t}(z,w)
=
\sum_{m\ge2}
\frac{1}{m!}\,
\kappa_m\!\left(z\cdot X_s+w\cdot X_t\right)
=
\sum_{m\ge2}
\frac{2^{m-1}}{m}
\operatorname{Tr}\!\left(B_{s,t}(z,w)^m\right).
\]
This series converges normally on $D_R\times D_R$. Indeed, using the bound
$
\left|
\operatorname{Tr}\!\left(B_{s,t}(z,w)^m\right)
\right|
\le
\|B_{s,t}(z,w)\|_{\mathrm{HS}}^m,
$
we obtain
\[
\sum_{m\ge2}
\frac{2^{m-1}}{m}
\left|
\operatorname{Tr}\!\left(B_{s,t}(z,w)^m\right)
\right|
\le
\sum_{m\ge2}
\frac{2^{m-1}}{m}
(\sqrt{2}|R|_\infty M)^m
<\infty,
\]
because $2\sqrt{2}|R|_\infty M<1$. The holomorphic function defined by the series agrees with $\K_{s,t}$ in a neighbourhood of the origin and therefore, by analytic continuation, on the whole of $D_R\times D_R$. Thus $\K$ is holomorphic on $D_R\times D_R$. It remains to verify the $C^1$-regularity in the time variable. Expanding the
trace gives
\[
\operatorname{Tr}\!\left(B_{s,t}(z,w)^m\right)
=
\sum_{\varepsilon_1,\ldots,\varepsilon_m\in\{s,t\}}
\sum_{i_1,\ldots,i_m=1}^d
u_{\varepsilon_1}^{i_1}\cdots u_{\varepsilon_m}^{i_m}\,
\tau_{i_1,\ldots,i_m}
(\varepsilon_1,\ldots,\varepsilon_m),
\]
where $u_s^i=z_i$ and $u_t^i=w_i$. Since the trace coefficients $\tau_{i_1,\ldots,i_m}$ are $C^1$, the map
\[
(s,t,z,w)
\longmapsto
\partial_t
\operatorname{Tr}\!\left(B_{s,t}(z,w)^m\right)
\]
is continuous for every $m\ge2$. Moreover, $\partial_t$ can act only on the slots for which $\varepsilon_r=t$, hence on at most $m$ slots. Therefore, using the uniform bound on the derivatives of the trace coefficients in \eqref{second_cond_chaos_2}, we deduce
\begin{align*}
\sup_{\substack{s,t\in[0,T]\\z,w\in D_R}}
\left|
\partial_t
\operatorname{Tr}\!\left(B_{s,t}(z,w)^m\right)
\right|
&\le
mC_m
\sum_{\varepsilon_1,\ldots,\varepsilon_m\in\{s,t\}}
\sum_{i_1,\ldots,i_m=1}^d
|R|_\infty^m=
mC_m(2d|R|_\infty)^m.
\end{align*}
Consequently,
\begin{align*}
\sum_{m\ge2}
\frac{2^{m-1}}{m}
\sup_{\substack{s,t\in[0,T]\\z,w\in D_R}}
\left|
\partial_t
\operatorname{Tr}\!\left(B_{s,t}(z,w)^m\right)
\right|
&\le
\sum_{m\ge2}
2^{m-1}C_m(2d|R|_\infty)^m
\\&=
\sum_{m\ge2}
m
\left(
\frac{2d|R|_\infty}{\lambda}
\right)^m
\frac{2^{m-1}}{m}C_m\lambda^m
<\infty.
\end{align*}
The last series is finite because $2d|R|_\infty<\lambda$, the sequence
$
m(2d|R|_\infty/\lambda)^m
$
is bounded, and \eqref{second_cond_chaos_2} holds. Thus the series
\[
\sum_{m\ge2}
\frac{2^{m-1}}{m}
\partial_t
\operatorname{Tr}\!\left(B_{s,t}(z,w)^m\right)
\]
converges uniformly on $[0,T]^2\times D_R\times D_R$. Since each term is
continuous in $(s,t,z,w)$, the sum is continuous. Hence we may differentiate
the cumulant series term by term, and this proves that $\partial_t\K$ exists, is continuous in $(s,t,z,w)$, and is uniformly bounded on
$[0,T]^2\times D_R\times D_R$. Passing to condition $3$, by \eqref{third_cond_chaos} and
\cite[Theorem~3.1]{NNP13}, the law of $X_t$ is absolutely continuous with respect to Lebesgue measure on $\mathbb R^d$ for every $t\in[0,T]$, and therefore condition $3$ in \autoref{conditions} is satisfied.
\end{proof}
Thanks to this result and the explicit form of the cumulants in \autoref{prop:second-chaos-cumulants}, we can now apply directly the results in \autoref{sec:ints} to derive the It\^o-Stratonovich and the It\^o formula for second chaos processes.

\begin{theorem}[It\^o formula for second-chaos processes]
Assume that $X$ satisfies the hypotheses of \autoref{thm_suff_second}, and let
$g=(g_1,\ldots,g_d)$ satisfy the hypotheses of \autoref{thm:martin} with
polyradius $r<R$, where $R$ is given by \eqref{eq_choice_poly_R}. Then, for
every $(s,t)\in\triangle_T$,
\begin{equation}\label{eq:second-chaos-strat-young}
\begin{split}
&\int_s^t g_j(X_u)\,\dia \dif X^j_u \\
={}&
\int_s^t g_j(X_u)\,\dif X^j_u  -
\sum_{n\geq 1}\frac{2^n\omega_{n+1}}{n!}
\int_s^t
\sum_{\Gamma\in\mathcal M_{\mathrm C}
(f^{i_1}_u,\ldots,f^{i_n}_u,f^j_{v})}
\partial_{i_1\cdots i_n}g_j(X_u)
\partial_v
\operatorname{Tr}(
\A_\Gamma
)\big|_{v = u} \dif u.
\end{split}
\end{equation}
If $g$ is scalar-valued and satisfies the same hypotheses, then
\begin{equation}\label{eq:second-chaos-ito}
\begin{split}
g(X)_{s,t}
={}&
\int_s^t \D g(X_u)\,\dia \dif X_u +
\sum_{m\geq 2}
\frac{2^{m-1}\omega_m}{m!}
\int_s^t
\sum_{\Gamma\in\mathcal M_{\mathrm C}
(f^{i_1}_u,\ldots,f^{i_m}_u)}
\partial_{i_1\cdots i_m}g(X_u)\,
\dif_u
\operatorname{Tr}(
\A_\Gamma
).
\end{split}
\end{equation}
Throughout the statement, repeated indices are summed over. Moreover, there exists a polyradius $\sigma$ such that, if
$r<\sigma\wedge R$, then we have the following equalities:
\begin{equation}\label{eq:second-chaos-expect}
\begin{split}
\mathbb E\int_s^t g_j(X_u)\,\dif X^j_u
={}&
\sum_{n\geq 1}\frac{2^n\omega_{n+1}}{n!}
\int_s^t
\sum_{\Gamma\in\mathcal M_{\mathrm C}
(f^{i_1}_u,\ldots,f^{i_n}_u,f^j_{v})}
\mathbb E\!\left[
\partial_{i_1\cdots i_n}g_j(X_u)
\right]
\partial_v
\operatorname{Tr}(
\A_\Gamma
)\big|_{v = u} \dif u,
\\
\mathbb{E}[g(X)_{s,t}]
={}&
\sum_{m\geq 2}
\frac{2^{m-1}\omega_m}{m!}
\int_s^t
\sum_{\Gamma\in\mathcal M_{\mathrm C}
(f^{i_1}_u,\ldots,f^{i_m}_u)}
\mathbb E\!\left[
\partial_{i_1\cdots i_m}g(X_u)
\right]\,
\dif_u \operatorname{Tr}\!\left(
\A_\Gamma
\right).
\end{split}
\end{equation}
\end{theorem}

\begin{proof}
The identities \eqref{eq:second-chaos-strat-young} and \eqref{eq:second-chaos-ito} follow by applying \autoref{thm:itostrat} and \autoref{thm:ito}, together with the explicit cumulant identities in \autoref{prop:second-chaos-cumulants}. To prove the identities in
\eqref{eq:second-chaos-expect}, it remains to verify the integrability assumptions in \autoref{thm:itostrat}. Fix $i\in\{1,\ldots,d\}$. From \eqref{first_cond_chaos}, the isometry of multiple Wiener integrals under our normalization, Nelson's hypercontractivity estimate \cite[Corollary~2.8.14]{NP}, and monotonicity of $L^q$-norms, for every $q\geq 1$ and every $(s,t)\in\triangle_T$ one has
\begin{equation}\label{eq_pointwise_estimates}
\|X^i_t-X^i_s\|_{L^q}
\lesssim q |t-s|^\alpha \,,
\qquad
\|X^i_0\|_{L^q}\lesssim q.
\end{equation}
Indeed, for $q\geq2$,
\[
\|X_t^i-X_s^i\|_{L^q}
\leq
(q-1)\|X_t^i-X_s^i\|_{L^2}
=
(q-1)\|f_t^i-f_s^i\|_{\mathcal H^{\odot2}},
\]
and similarly
$\|X_0^i\|_{L^q}\leq(q-1)\|f_0^i\|_{\mathcal H^{\odot2}}$.

We now pass from increment estimates to a supremum estimate. Since $\alpha>1/2$, choose $p\in(1,2)$ such that $1/p<\alpha$, then choose
$\beta\in(1/p,\alpha)$, and finally choose $r_0\geq1$ sufficiently large that $(\alpha-\beta)r_0>1$. By the GRR inequality (see \cite[Appendix~A2]{FV10}), applied pathwise to the continuous version of $X^i$, one has
\[
[X^i]_{\beta\text{-}\mathrm{Hl};[0,T]}
\lesssim
\left(
\int_0^T\int_0^T
\frac{|X^i_t-X^i_s|^{r_0}}{|t-s|^{\beta r_0+2}}
\,\dif s\,\dif t
\right)^{1/r_0}\,,
\]
where $[\cdot]_{\beta\text{-}\mathrm{Hl};[0,T]}$ is the $\beta$-H\"older seminorm of a path. For $q\geq r_0$, Minkowski's integral inequality and
\eqref{eq_pointwise_estimates} give
\begin{align*}
\left\|
[X^i]_{\beta\text{-}\mathrm{Hl};[0,T]}
\right\|_{L^q}
&\lesssim
\left(
\int_0^T\int_0^T
\frac{\|X^i_t-X^i_s\|_{L^q}^{r_0}}
{|t-s|^{\beta r_0+2}}
\,\dif s\,\dif t
\right)^{1/r_0}
\\
&\lesssim
q
\left(
\int_0^T\int_0^T
|t-s|^{(\alpha-\beta)r_0-2}
\,\dif s\,\dif t
\right)^{1/r_0}.
\end{align*}
The last integral is finite precisely because $(\alpha-\beta)r_0>1$. Thus, for each $i=1,\ldots,d$ and $q\geq r_0$,
\begin{equation}\label{eq:bound_Hoelder}
\left\|
[X^i]_{\beta\text{-}\mathrm{Hl};[0,T]}
\right\|_{L^q}
\lesssim q\,.
\end{equation}
For $1\leq q<r_0$, the same bound follows from monotonicity of $L^q$-norms, after enlarging the constant. Condition \eqref{eq:bound_Hoelder} is then enough to prove the finiteness of $\|X\|_{p\text{-}\mathrm{var};[0,T]}$ for some $p<2$. Indeed, using the deterministic embedding of H\"older paths into paths of finite $p$-variation, for every $q\geq1$ one has
\[
\E\left[
\|X^i\|_{p\text{-}\mathrm{var};[0,T]}^q
\right]
\leq
T^{\beta q}
\E\left[
[X^i]_{\beta\text{-}\mathrm{Hl};[0,T]}^q
\right]
<\infty .
\]
It follows, by the triangle inequality, that $\|X\|_{p\text{-}\mathrm{var};[0,T]}$ has moments of all orders for every $p\in(1,2)$ such that $1/p<\alpha$. From the same bound \eqref{eq:bound_Hoelder}, we can also show that $\E\exp\!\left(\sigma\cdot X^*\right)<\infty$ for a sufficiently small polyradius $\sigma$. Since
\[
\sup_{t\in[0,T]} |X^i_t|
\leq
|X^i_0|+
T^\beta [X^i]_{\beta\text{-}\mathrm{Hl};[0,T]},
\]
we deduce from \eqref{eq_pointwise_estimates} and
\eqref{eq:bound_Hoelder} that there exist constants $C_i>0$ such that the random variable $(X^i)^*= \sup_{t\in[0,T]}|X^i_t|
$ satisfies
\[
\left\|
(X^i)^*
\right\|_{L^q}
\leq
C_i q,
\qquad q\geq1.
\]
Therefore, for any polyradius $\sigma=(\sigma_1,\ldots,\sigma_d)\in(0,\infty)^d$ and any $q\geq 1$, one has on the random vector $X^*= ((X^1)^*,\ldots (X^d)^*) $ 
\[
\|\sigma\cdot X^*\|_{L^q}
\leq
\sum_{i=1}^d
\sigma_i
\left\|
(X^i)^*\right\|_{L^q}
\leq
q\sum_{i=1}^d \sigma_i C_i.
\]
Consequently, for every integer $k\geq1$,
\[
\E\big[(\sigma\cdot X^*)^k\big]
\leq
k^k
\left(
\sum_{i=1}^d \sigma_i C_i
\right)^k .
\]
Using Stirling's estimate $k!\geq(k/e)^k$, we obtain
\begin{align*}
\E\exp\!\left(\sigma\cdot X^*\right)
&=
\sum_{k=0}^{\infty}
\frac{\E[(\sigma\cdot X^*)^k]}{k!}
\leq
\sum_{k=0}^{\infty}
\left(
e\sum_{i=1}^d \sigma_i C_i
\right)^k .
\end{align*}
Therefore $
\E\exp\!\left(\sigma\cdot X^*\right)<\infty$
whenever $e\sum_{i=1}^d\sigma_iC_i<1$. For example, choosing any $\sigma$ such that
$|\sigma|_{\infty}<\bigl(e\sum_{i=1}^d C_i\bigr)^{-1}$ yields the conclusion.
\end{proof}

From the compact expression in \eqref{eq:1D_cumulant}, for a scalar process we obtain a simplified formula.
\begin{corollary}[The scalar second-chaos case]\label{prop:1d_Ito_second_chaos}
	 For any one-dimensional process $X$ satisfying the hypotheses of \autoref{thm_suff_second} and any $g\colon \mathbb{R}\to \mathbb{R}$ satisfying the hypotheses of \autoref{thm:martin} with
radius $r<R$, where $R$ is given by \eqref{eq_choice_poly_R} one has
	\begin{equation*}
		\begin{split}
			g(X)_{s,t}
			&=
			\int_s^t g'(X_u) \dia \dif X_u
			+
			\sum_{n = 2}^\infty
			\frac{2^{n-1}}{n}
			\int_s^t g^{(n)}(X_u)\,
			\dif_u \operatorname{Tr}(\A_u^n)\,.
		\end{split}
	\end{equation*}
\end{corollary}
\subsection{The Rosenblatt process}

We conclude the section by applying the preceding results to the
one-dimensional Rosenblatt process. Let $H\in(1/2,1)$. The Rosenblatt
process $X^H=(X_t^H)_{t\in[0,T]}$ is a centred second-chaos process over $\mathcal H=L^2(\mathbb R)$, namely $X_t^H=\I^2(f_t^H)$, $ t\in[0,T]$
where $f_t^H\in L^2(\mathbb R^2)$ is given by the symmetric function
\[
f_t^H(x_1,x_2)
=
c_H\int_0^t
(s-x_1)_+^{\frac H2-1}
(s-x_2)_+^{\frac H2-1}
\,\dif s .
\]
Here $x_+^\alpha=x^\alpha$ if $x>0$, and $x_+^\alpha=0$ otherwise. The constant $c_H>0$ is chosen so that $\E\big[(X_t^H)^2\big]=t^{2H}$. The Rosenblatt process was introduced by Taqqu \cite{Taqqu1975} as a
fundamental example of a non-Gaussian self-similar process with stationary
increments arising in non-central limit theorems. With the above
normalization, $X^H$ is $H$-self-similar and has stationary increments.
More precisely, when the process is constructed on $[0,\infty)$, for every
$a>0$,
\begin{equation*}
\bigl(X_{at}^H\bigr)_{t\ge0}
\stackrel{\mathrm{law}}{=}
\bigl(a^H X_t^H\bigr)_{t\ge0},
\end{equation*}
and, for every $h\ge0$,
\begin{equation*}
\bigl(X_{t+h}^H-X_h^H\bigr)_{t\ge0}
\stackrel{\mathrm{law}}{=}
\bigl(X_t^H\bigr)_{t\ge0}.
\end{equation*}
Together with the normalization $\E[(X_t^H)^2]=t^{2H}$, these properties imply that $X^H$ has the same covariance structure as fractional Brownian motion with Hurst parameter $H$. Nevertheless, $X^H$ is non-Gaussian, since it belongs to the second Wiener chaos. We refer to \cite{Tudor23} for a detailed construction of the Rosenblatt process and for its standard analytic and probabilistic properties. In our context, the Rosenblatt process is a very simple example of process satisfying the conditions of \autoref{thm_suff_second}.

\begin{proposition}
\label{prop:cum_rosenblatt}
The process $X^H$ satisfies the hypotheses of \autoref{thm_suff_second}.
Moreover, its $n$-th order cumulant function satisfies $\kappa_n(t)=t^{nH}\kappa_n(1)$ for every $n\geq2$.
\end{proposition}
\begin{proof}
We first verify the H\"older condition on the kernels. By the Wiener chaos
isometry, for $(s,t)\in\triangle_T$,
\[
\|f_t^H-f_s^H\|_{\calH^{\odot2}}^2
=
\E\bigl[(X_t^H-X_s^H)^2\bigr].
\]
Using stationarity of the increments, self-similarity, and the normalization
$\E[(X_1^H)^2]=1$, we get
\[
\|f_t^H-f_s^H\|_{\calH^{\odot2}}^2
=
\E\bigl[(X_{t-s}^H)^2\bigr]
=
(t-s)^{2H}.
\]
Thus \eqref{first_cond_chaos} holds with $\alpha=H>1/2$. We now prove the regularity of the trace coefficients. Let $\A_t^H$ be the
Hilbert--Schmidt operator on $\calH$ associated with $f_t^H$, i.e.
\begin{equation}\label{operator_rosenblatt}
\A_t^Hh(x_1)
=
\int_{\mathbb R}f_t^H(x_1,x_2)h(x_2)\,\dif x_2,
\end{equation}
and, for $n\geq2$, set
$\tau_n^H(t_1,\ldots,t_n)\coloneqq
\operatorname{Tr}(\A_{t_1}^H\cdots \A_{t_n}^H)$.
Writing the trace as an integral, we obtain
\begin{equation}\label{eq_cumulant}
\begin{split}
\tau_n^H(t_1,\ldots,t_n)
={}&
\int_{\mathbb R^n}
f_{t_1}^H(x_1,x_2)
f_{t_2}^H(x_2,x_3)
\cdots
f_{t_n}^H(x_n,x_1)
\,\dif x_1\cdots\dif x_n .
\end{split}
\end{equation}
Using the identity
\[
\int_{\mathbb R}
(u-x)_+^{\alpha-1}(v-x)_+^{\alpha-1}\,\dif x
=
B(\alpha,1-2\alpha)|u-v|^{2\alpha-1},
\qquad
0<\alpha<\frac12,
\]
with $\alpha=H/2$ and $B$ Euler's beta function, we get
\[
\tau_n^H(t_1,\ldots,t_n)
=
C_{H,n}
\int_0^{t_1}\cdots\int_0^{t_n}
\prod_{i=1}^n|s_i-s_{i+1}|^{H-1}
\,\dif s_1\cdots\dif s_n,
\]
where $s_{n+1}=s_1$ and
$
C_{H,n}
=
c_H^nB\!\left(\frac H2,1-H\right)^n
$. We now prove that this function belongs to $C^1([0,T]^n)$. Fix
$k\in\{1,\ldots,n\}$ and, for
$x\in[0,T]$ and $t_{\widehat k}:=(t_i)_{i\ne k}$, define
\[
G_k(x;t_{\widehat k})
:=
C_{H,n}
\int_{\prod_{i\ne k}[0,t_i]}
\left.
\prod_{i=1}^n|s_i-s_{i+1}|^{H-1}
\right|_{s_k=x}
\prod_{i\ne k}\dif s_i.
\]
Then
\[
\tau_n^H(t_1,\ldots,t_n)
=
\int_0^{t_k}G_k(x;t_{\widehat k})\,\dif x,
\]
so it is enough to prove that $G_k$ is finite and continuous. Introduce
$\rho(x):=|x|^{H-1}\mathbf 1_{\{0<|x|\leq T\}}$. By non-negativity and
extension of the integration domain,
\[
0\leq G_k(x;t_{\widehat k})
\leq
C_{H,n}(\rho^{*n})(0)
\leq
C_{H,n}\|\rho\|_2^2\|\rho\|_1^{n-2}.
\]
Since $H>1/2$,
\[
\rho\in L^1(\mathbb R)\cap L^2(\mathbb R),
\qquad
\|\rho\|_1=\frac{2T^H}{H},
\qquad
\|\rho\|_2^2=\frac{2T^{2H-1}}{2H-1},
\]
and hence $G_k$ is uniformly finite. To prove continuity, let $\rho_m\in C_c(\mathbb R)$ be such that
$\rho_m\to\rho$ in $L^1(\mathbb R)\cap L^2(\mathbb R)$, and define
\[
G_{k,m}(x;t_{\widehat k})
:=
C_{H,n}
\int_{\prod_{i\ne k}[0,t_i]}
\prod_{i=1}^n\rho_m(s_i-s_{i+1})
\prod_{i\ne k}\dif s_i,
\qquad s_k=x.
\]
Each $G_{k,m}$ is continuous. Using
\[
\prod_{i=1}^na_i-\prod_{i=1}^nb_i
=
\sum_{j=1}^n
(a_j-b_j)
\prod_{i<j}a_i
\prod_{i>j}b_i
\]
and the same convolution estimate as above, we obtain
\[
\sup_{x,t_{\widehat k}}
|G_{k,m}(x;t_{\widehat k})-G_k(x;t_{\widehat k})|
\leq
C_{H,n}n\,
\|\rho_m-\rho\|_2\,M_2\,M_1^{n-2},
\]
where
\[
M_1:=\|\rho\|_1+\sup_m\|\rho_m\|_1,
\qquad
M_2:=\|\rho\|_2+\sup_m\|\rho_m\|_2.
\]
Thus $G_{k,m}\to G_k$ uniformly, and therefore $G_k$ is continuous. It
follows that $\tau_n^H\in C^1([0,T]^n)$ and
\[
\partial_{t_k}\tau_n^H(t_1,\ldots,t_n)
=
G_k(t_k;t_{\widehat k}).
\]
Moreover,
\[
\sup_{t_1,\ldots,t_n\in[0,T]}
|\partial_{t_k}\tau_n^H(t_1,\ldots,t_n)|
\leq
C_{H,n}\|\rho\|_2^2\|\rho\|_1^{n-2}.
\]
Thus, setting
\[
C_n
\coloneqq
c_H^nB\!\left(\frac H2,1-H\right)^n
\|\rho\|_2^2\|\rho\|_1^{n-2},
\]
the sequence $(C_n)_{n\geq2}$ has at most exponential growth, and hence
\[
\sum_{n\geq2}
\frac{2^{n-1}}{n}C_n(\lambda_{T,H})^n<\infty
\]
for all sufficiently small $\lambda_{T,H}>0$ depending on $H$ and $T$. This proves
\eqref{second_cond_chaos_2}. Concerning condition \eqref{third_cond_chaos}, for every $t>0$ the random
variable $X_t^H$ admits a density. Hence, by
\cite[Theorem~3.1]{NNP13}, one has
$\mathbb E[\det\Gamma_t]>0$, and therefore
\eqref{third_cond_chaos} holds for every $t\in(0,T]$. It remains to prove the scaling identity for the one-point cumulants. Since
the Rosenblatt process is $H$-self-similar and cumulants are homogeneous under multiplication by constants, for every $n\geq2$ we obtain
\[
\kappa_n[X_t^H]
=
\kappa_n[t^HX_1^H]
=
t^{nH}\kappa_n[X_1^H].
 \qedhere\]
\end{proof}
It then follows from
\autoref{prop:1d_Ito_second_chaos}  the following It\^o formula for the Rosenblatt process.

\begin{corollary}[It\^o formula for the Rosenblatt process]\label{corollary_Ito_Rosenblatt}
There exists a radius $R_{T,H}>0$ such that, for any
$f\colon\mathbb R\to\mathbb R$ satisfying the hypotheses of
\autoref{thm:martin} with radius $r<R_{T,H}$, one has
\begin{equation*}
\begin{split}
f(X^H)_{s,t}
&=
\int_s^t f'(X_u^H) \dia \dif X_u^H
+
\sum_{n = 2}^\infty
\frac{H\,\kappa_n(1)}{(n-1)!}
\int_s^t f^{(n)}(X_u^H)\,
u^{nH-1}\,\dif u,
\end{split}
\end{equation*}
with $\kappa_n(1) = 2^{n-1}(n-1)! \operatorname{Tr}((\A_{1}^H)^n)$ given by \eqref{eq_cumulant}.
\end{corollary}

\begin{remark}
The strategy used in the proof of \autoref{prop:cum_rosenblatt} suggests a possible extension to Hermite processes of arbitrary order. These processes share the self-similarity and stationary-increment properties of fractional Brownian motion and the Rosenblatt process, and arise in Dobrushin--Major type non-central limit theorems \cite{DobMaj}; see also \cite{Tudor23}. The higher-order case, however, presents two main difficulties. First, the cumulant formula no longer reduces to the trace-type expressions available in the Rosenblatt case, and secondly the only admissible integrands will be polynomials, as explained in \autoref{rem:notexp}.
\end{remark}

We end by comparing the Wick integral to the existing literature, specifically the work of \cite{Arras15}. We also mention that a related change-of-variable formula for systems involving the Rosenblatt process and fractional Brownian motion was obtained in \cite{coupek_2022}.
The framework of \cite{Arras15} is different from ours. Instead of using the Wick product associated with the law of $X^H$, which we continue denoting by $\dia$, the author works in the white-noise space generated by the underlying Gaussian noise $W$. The corresponding Wick product will be denoted by $\dia_W$. This has the advantage that one can use the full machinery of white-noise analysis and Hida distributions, at the price of keeping track of objects which are defined through the underlying Gaussian noise rather than through the process $X^H$ alone. More precisely, in \cite{Arras15} one works with the Hida triple $ \mathcal S \subset L^2(\Omega) \subset \mathcal S^*$, see, for instance, \cite{Kuo96} for proper definition. In this setting the Wick product $\dia_W$ and the time integral extend naturally to suitable $\mathcal S^*$-valued processes. The Rosenblatt noise is then interpreted as a Hida-distribution-valued derivative $\dot X^H_t$, and the stochastic integral is defined by
\begin{equation}\label{eq:arras_int}
 \int_s^t \varphi_r \,\dia_W\, \dot X^H_r\,\dif r .
 \end{equation}
The same formalism and integration also applied to the auxiliary family of second-chaos processes
\begin{equation}\label{def_aux}
 X^{H,n}_t=\I^2(\underbrace{f_t^H \star \cdots \star f_t^H}_{n \text{ times}})\,.
 \end{equation}
Thus $X^{H,1}=X^H$, whereas the higher processes $X^{H,n}$ are constructed from the same underlying Gaussian noise and are, in general, not measurable with respect to the process $X^H$ alone. The choice of considering the one-dimensional kernel $f^H$, implies that in \eqref{def_aux} there is no need of symmetrisation.  In the polynomial case the white-noise integral in \eqref{eq:arras_int} admits a concrete interpretation as the limit of Wick--Riemann sums. This observation allows one to connect the Hida-distribution formulation of \cite{Arras15} with the Riemann-sum point of view used below.

\begin{proposition} Fix $k\geq1$. For every polynomial $p:\mathbb{R}\to\mathbb{R}$, the process $r\to p(X^H_r)\diamond_W \dot X^{H,k}_r$ is $(\mathcal S)^*$-integrable on $[s,t]$. Moreover, the integral belongs to $L^2(\Omega)$, and for every sequence of partitions $\pi_m$ of $[s,t]$  with $|\pi_m|\to0$, one has
\begin{equation}\label{eq:convergence_rosenblatt}
\int_s^t p(X^H_r)\diamond_W \dot X^{H,k}_r\,\dif r = \lim_{m\to\infty} \sum_{[u,v]\in\pi_m} p(X^H_u)\diamond_W X^{H,k}_{u,v} \qquad\text{in }L^2(\Omega)\,.
\end{equation}
This limit coincides with the Wick integral with respect to $\diamond_W$ and we will denote it by $\int_s^t p(X^H_r)\diamond_W  \dif X^{H,k}_r\,$.
\end{proposition}

\begin{proof}
Write $X^H_t=\I^2(f^H_t)$ and $X^{H,k}_t=\I^2(f^{H,k}_t)$, so that $X^{H,k}_{u,v}=\I^2(f^{H,k}_{u,v})$. From the proof of \autoref{prop:cum_rosenblatt}, we already know that $t\mapsto f^H_t$ is $H$-Hölder continuous with values in $\mathcal H^{\odot 2}$. Since $\star$ is a continuous map $\mathcal H^{\odot 2}\times\mathcal H^{\odot 2}\to\mathcal H^{\odot 2}$, it follows that $t\mapsto f^{H,k}_t$ is also $H$-Hölder continuous. We first show that the right-hand side of \eqref{eq:convergence_rosenblatt} is well-defined. Since $X^H_t=\I^2(f^H_t)$, the random variable $p(X^H_t)$ belongs to a finite sum of Wiener chaoses. Hence, by the product formula for multiple Wiener integrals with respect to $\dia_W$, there exist an integer $N\geq0$ and symmetric kernels $g^q_t\in\mathcal H^{\odot q}$, for $0\leq q\leq N$, such that $p(X^H_t)=\sum_{q=0}^{N}\I^q(g_t^q)\,$. The kernels $g_t^q$ are finite linear combinations of iterated contractions of copies of $f^{H,l}_t$, with $l\geq1$. Therefore the maps $t\mapsto g_t^q$ are bounded and $H$-Hölder continuous. Now fix a partition $\pi_m$ of $[s,t]$. Using the defining property of the Gaussian Wick product, $ \I^q(g)\diamond_W \I^2(h)=\I^{q+2}(g\odot h), $ we obtain
\[\sum_{[u,v]\in\pi_m} p(X^H_u)\diamond_W X^{H,k}_{u,v} = \sum_{[u,v]\in\pi_m} \sum_{q=0}^{N} \I^q(g^q_u)\diamond_W \I^2(f^{H,k}_{u,v})= \sum_{q=0}^{N} \I^{q+2} \left( \sum_{[u,v]\in\pi_m} g^q_u\odot f^{H,k}_{u,v} \right). \]
 Since $g^q$ and $f^{H,k}$ are $H$-Hölder continuous, $H+H>1$, and $\odot$ is continuous, for each $q$ the preceding Riemann sums converge to the Young integral
 \[ \sum_{[u,v]\in\pi_m} g^q_u\odot f^{H,k}_{u,v} \longrightarrow \int_s^t g^q_r\odot \dif f^{H,k}_r \qquad\text{in }\mathcal H^{\odot(q+2)} . \]
  Consequently, by the Wiener isometry and the orthogonality of different chaoses,
 \[ \sum_{[u,v]\in\pi_m} p(X^H_u)\diamond_W X^{H,k}_{u,v} \longrightarrow \sum_{q=0}^{N} \I^{q+2} \left( \int_s^t g^q_r\odot \dif f^{H,k}_r \right) \qquad\text{in }L^2(\Omega). \]
 In particular, the limit belongs to $L^2(\Omega)$. It remains to identify this $L^2(\Omega)$-limit with the Hida integral in the left-hand side of \eqref{eq:convergence_rosenblatt}.  The $(\mathcal S)^*$-integrability of $r\mapsto p(X^H_r)\diamond_W \dot X^{H,k}_r$ follows from \cite[Theorem 3.16 and Remark 3.17]{Arras15}. Hence the $(\mathcal S)^*$-valued integral $\int_s^t p(X^H_r)\diamond_W\dot X^{H,k}_r\,\dif r$
  is well-defined and is characterized by \[ S\left( \int_s^t p(X^H_r)\diamond_W\dot X^{H,k}_r\,\dif r \right)(\xi) = \int_s^t S\bigl(p(X^H_r)\diamond_W\dot X^{H,k}_r\bigr)(\xi)\,\dif r , \]
   see \cite[Theorem 13.4]{Kuo96}. Here $S\colon(\mathcal S)^*\to (\mathcal S(\mathbb R)\to\mathbb R)$ denotes the $S$-transform, where $\mathcal S(\mathbb R)$ is the space of Schwartz functions; see \cite[Definition 5.8]{Kuo96}. We shall also use the characterization of the Gaussian Wick product by the $S$-transform:
   \[ S(\Phi\diamond_W\Psi)(\xi) = S(\Phi)(\xi)\,S(\Psi)(\xi), \qquad \xi\in\mathcal S(\mathbb R), \]
   for any $\Phi,\Psi\in(\mathcal S)^*$; see, for instance, \cite[Theorem 1.3]{Arras15}. Let $Y$ be the $L^2(\Omega)$-limit of the Wick--Riemann sums obtained above. Since $L^2(\Omega)$-convergence implies convergence of $S$-transforms, for every $\xi\in\mathcal S(\mathbb R)$ we have
    \begin{align*} S(Y)(\xi) &= \lim_{m\to\infty} \sum_{[u,v]\in\pi_m} S\bigl(p(X^H_u)\diamond_W X^{H,k}_{u,v}\bigr)(\xi) = \lim_{m\to\infty} \sum_{[u,v]\in\pi_m} S(p(X^H_u))(\xi)\, S(X^{H,k}_{u,v})(\xi).
     \end{align*}
      By definition of the Hida-distributional derivative $\dot X^{H,k}_r$, $X^{H,k}_{u,v} = \int_u^v \dot X^{H,k}_r\,\dif r$ in $(\mathcal S)^*$. Applying the $S$-transform yields $S(X^{H,k}_{u,v})(\xi) = \int_u^v S(\dot X^{H,k}_r)(\xi)\,\dif r$.
      Therefore \[ S(Y)(\xi) = \lim_{m\to\infty} \sum_{[u,v]\in\pi_m} S(p(X^H_u))(\xi) \int_u^v S(\dot X^{H,k}_r)(\xi)\,\dif r . \]
      The map $r\mapsto S(p(X^H_r))(\xi)$ is continuous on $[s,t]$. Indeed, writing $p(X^H_r)=\sum_{q=0}^{N}\I^q(g^q_r)$, we have
     \[
S(p(X^H_r))(\xi)
=
\sum_{q=0}^{N}
\big\langle g^q_r,\xi^{\otimes q}\big\rangle_{\mathcal H^{\otimes q}}
=
\sum_{q=0}^{N}
\frac{1}{q!}
\big\langle g^q_r,\xi^{\otimes q}\big\rangle_{\mathcal H^{\odot q}},
\]
      and each map $r\mapsto g^q_r$ is continuous. Moreover, by the definition of the Hida-distributional derivative and by the explicit formulae for $\dot X^{H,k}$ in \cite[Lemma 3.7]{Arras15}, the map $r\mapsto S(\dot X^{H,k}_r)(\xi)$ belongs to $L^1([s,t])$. Hence the preceding weighted Riemann sums converge to
      \begin{align*}
      \lim_{m\to\infty} \sum_{[u,v]\in\pi_m} S(p(X^H_u))(\xi) \int_u^v S(\dot X^{H,k}_r)(\xi)\,\dif r &= \int_s^t S(p(X^H_r))(\xi)\, S(\dot X^{H,k}_r)(\xi)\,\dif r \\ &= S\left( \int_s^t p(X^H_r)\diamond_W\dot X^{H,k}_r\,\dif r \right)(\xi)\,.
      \end{align*}
     Thus the $S$-transforms of $Y$ and of the Hida integral coincide for every $\xi\in\mathcal S(\mathbb R)$. Since the $S$-transform is injective on $(\mathcal S)^*$, the $L^2(\Omega)$-limit of the Wick--Riemann sums coincides with the Hida integral.
 \end{proof}
 Combining the change of chaos identity \eqref{eq:change_second} with the convergence in \eqref{eq:convergence_rosenblatt}, we derive the It\^o formula \cite[Theorem 3.16]{Arras15} as a conseguence of ours; since this result is not proved for integrands satisfying our assumptions we only handle the polynomial case.
 
\begin{theorem}\label{prop:second-chaos-integrals}
For every polynomial $p:\mathbb R\to\mathbb R$ one has
\begin{equation}\label{change_of_chaos_integrals}
\begin{split}
 \int_s^t p(X^H_u)\dia \dif X^H_u
 ={}& \int_s^t p(X^H_u)\dia_W \dif X^H_u + \sum_{n=2}^{\infty} 2^{n-1} \int_s^t p^{(n-1)}(X^H_u)\dia_W \dif X^{H,n}_u \,.
\end{split}
 \end{equation}
As a conseguence, we obtain \cite[Theorem 3.16]{Arras15}
\begin{equation}\label{change_of_chaos_integrals_ito}
\begin{split}
 p(X^H)_{s,t} ={}& \int_s^t p'(X^H_u)\dia_W \dif X^H_u + \sum_{n=2}^{\infty} 2^{n-1} \int_s^t p^{(n)}(X^H_u)\dia_W \dif X^{H,n}_u \\
 &+ \sum_{n=2}^{\infty} \frac{H\kappa_n(1)}{(n-1)!} \int_s^t p^{(n)}(X^H_u)u^{Hn-1}\,\dif u \,.
\end{split}
\end{equation}
\end{theorem}

\begin{proof}
It is enough to prove \eqref{change_of_chaos_integrals}. We first consider
the Appell monomials \(p(x)=x^{\dia k}\), with \(k\geq0\), and prove that
\begin{equation}\label{thesis_polynomials}
\begin{split}
\int_s^t (X^H_u)^{\dia k}\dia \dif X^H_u
={}&
\int_s^t (X^H_u)^{\dia k}\dia_W \dif X^H_u
+
\sum_{n=2}^{k+1}
2^{n-1}\frac{k!}{(k-n+1)!}
\int_s^t
(X^H_u)^{\dia(k-n+1)}
\dia_W \dif X^{H,n}_u .
\end{split}
\end{equation}
The goal is to derive an identity among the Riemann sums defining the integrals in \eqref{thesis_polynomials} up to a remainder that will converge to $0$;
By Fixing  \(s\leq u<v\leq t\) and writing $(X_u^H)^{\dia k}\dia X_{u,v}^H$ according to \eqref{eq:change_second} we have

\begin{equation}\label{eq:forest_increment_expansion}
\begin{split}
(X_u^H)^{\dia k}\dia X_{u,v}^H
={}&
\sum_{\substack{
\Gamma\in\mathcal M_{\mathrm F}
(f_u^H,\ldots,f_u^H,f^H_{u,v})}}
2^{|V(\Gamma)|}
\I^{2(k+1)-2|E(\Gamma)|}(f_\Gamma^H),
\end{split}
\end{equation}
where \(V(\Gamma) \,, E(\Gamma)\) denote the set of non-isolated vertices and edges of \(\Gamma\). We organise the terms according to the connected component containing the distinguished vertex \(f^H_{u,v}\). If \(f^H_{u,v}\) is isolated, then the remaining
forest involves only the \(k\) copies of \(f_u^H\). Applying
\eqref{eq:change_second} to these remaining vertices shows that the
corresponding terms, together with the first term in
\eqref{eq:forest_increment_expansion}, sum to
\[
(X_u^H)^{\dia k}\dia_W X_{u,v}^H.
\]
We now fix \(n\in\{2,\ldots,k+1\}\) and consider the forests for which
the component containing \(f^H_{u,v}\) has \(n\) vertices. Introduce the
symmetric kernel
\begin{equation}\label{eq:linear_kernel_increment}
L_{u,v}^{H,n}
:=
\sum_{j=0}^{n-1}
\operatorname{Sym}_2\left(
\underbrace{f_u^H\star\cdots\star f_u^H}_{j\text{ factors}}
\star f^H_{u,v}\star
\underbrace{f_u^H\star\cdots\star f_u^H}_{n-1-j\text{ factors}}
\right),
\end{equation}
where an empty block is omitted. Since transposition reverses a
\(\star\)-product, the summands corresponding to \(j\) and \(n-1-j\)
are transposes of one another. Consequently,
\begin{equation}\label{eq:linear_kernel_without_sym}
L_{u,v}^{H,n}
=
\sum_{j=0}^{n-1}
\underbrace{f_u^H\star\cdots\star f_u^H}_{j\text{ factors}}
\star f^H_{u,v}\star
\underbrace{f_u^H\star\cdots\star f_u^H}_{n-1-j\text{ factors}} .
\end{equation}
To compute the multiplicity, first choose the \(n-1\) copies of \(f_u^H\) belonging to the distinguished component. For any fixed choice, the oriented paths are obtained by choosing the position of \(f^H_{u,v}\) and
ordering the remaining \(n-1\) labelled vertices. Since a path and its reverse determine the same symmetric contraction, the sum over all unoriented paths on these vertices is
\[
\frac{(n-1)!}{2}\,L_{u,v}^{H,n}.
\]
There are \(\binom{k}{n-1}\) possible choices of the \(n-1\) vertices. Hence the sum of the contractions associated with all possible distinguished components is
\[
\binom{k}{n-1}\frac{(n-1)!}{2}\,
L_{u,v}^{H,n}
=
\frac{k!}{2(k-n+1)!}\,L_{u,v}^{H,n}.
\]
Each such component contains \(n\) non-isolated vertices and therefore
carries the factor \(2^n\). The remaining \(k-n+1\) copies of \(f_u^H\) form an arbitrary forest. Summing over these remaining forests and applying \eqref{eq:change_second} once more gives \((X_u^H)^{\dia(k-n+1)}\). Since symmetric tensor products correspond
to Gaussian Wick products, the total contribution at level \(n\) is
therefore
\[
2^{n-1}\frac{k!}{(k-n+1)!}
(X_u^H)^{\dia(k-n+1)}
\dia_W \I^2(L_{u,v}^{H,n}).
\]
We have thus obtained the exact identity
\begin{equation}\label{eq:exact_linear_identity}
\begin{split}
(X_u^H)^{\dia k}\dia X_{u,v}^H
={}&
(X_u^H)^{\dia k}\dia_W X_{u,v}^H+
\sum_{n=2}^{k+1}
2^{n-1}\frac{k!}{(k-n+1)!}
(X_u^H)^{\dia(k-n+1)}
\dia_W \I^2(L_{u,v}^{H,n}).
\end{split}
\end{equation}
We next compare \(L_{u,v}^{H,n}\) with the increment of
\(f^{H,n}\), the kernel of $X^{H,n}$. By associativity and bilinearity of \(\star\),
\[
f_v^{H,n}
=
\underbrace{(f_u^H+f^H_{u,v})\star\cdots\star(f_u^H+f^H_{u,v})}_{n\text{ factors}}.
\]
After expanding this product, the term containing no copy of \(f^H_{u,v}\) is
\(f_u^{H,n}\), while the sum of the terms containing exactly one copy
of \(f^H_{u,v}\) is \(L_{u,v}^{H,n}\), by
\eqref{eq:linear_kernel_without_sym}. It follows that
\begin{equation}\label{eq:increment_kernel_decomposition}
f_{u,v}^{H,n}
=
L_{u,v}^{H,n}+Q_{u,v}^{H,n},
\end{equation}
where \(Q_{u,v}^{H,n}\) is the sum of all words containing at least two
copies of \(f_{u,v}^H\), which is   is symmetric. Substituting
\[
\I^2(L_{u,v}^{H,n})
=
X_{u,v}^{H,n}-\I^2(Q_{u,v}^{H,n})
\]
into \eqref{eq:exact_linear_identity}, we obtain
\begin{equation}\label{eq:exact_identities}
\begin{split}
(X_u^H)^{\dia k}\dia X_{u,v}^H
={}&
(X_u^H)^{\dia k}\dia_W X_{u,v}^H
+
\sum_{n=2}^{k+1}
2^{n-1}\frac{k!}{(k-n+1)!}
(X_u^H)^{\dia(k-n+1)}
\dia_W X_{u,v}^{H,n}
+
R_{u,v}^{(k)},
\end{split}
\end{equation}
where
\[
R_{u,v}^{(k)}
:=
-
\sum_{n=2}^{k+1}
2^{n-1}\frac{k!}{(k-n+1)!}
(X_u^H)^{\dia(k-n+1)}
\dia_W \I^2( (Q_{u,v}^{H,n})).
\]
We now estimate the remainder. Recall that
\[
\sup_{r\in[0,T]}
\|f_r^H\|_{\mathcal H^{\odot 2}}<\infty,
\qquad
\|f_{u,v}^H\|_{\mathcal H^{\odot 2}}
\lesssim |v-u|^H.
\]
Every term in \(Q_{u,v}^{H,n}\) is an \(n\)-fold \(\star\)-product
containing at least two copies of \(f_{u,v}^H\). By continuity of \(\star\) on \(\mathcal H^{\otimes2}\) and since \(n\) is fixed,
\[
\|Q_{u,v}^{H,n}\|_{\mathcal H^{\odot 2}}
\lesssim |v-u|^{2H}.
\]
Then the Wiener isometry gives
\[
\|\I^2(Q_{u,v}^{H,n})\|_{L^2(\Omega)}
\lesssim |v-u|^{2H}.
\]
For fixed \(k\) and \(n\), the family
\[
\left\{
(X_u^H)^{\dia(k-n+1)}:u\in[s,t]
\right\}
\]
belongs to a fixed finite sum of Gaussian chaoses. Moreover, the
corresponding kernels are uniformly bounded in \(u\). Therefore there
exists a constant \(C_{k,n}\) such that, for every
\(q\in\mathcal H^{\odot2}\),
\[
\sup_{u\in[s,t]}
\left\|
(X_u^H)^{\dia(k-n+1)}
\dia_W\I^2(q)
\right\|_{L^2(\Omega)}
\leq
C_{k,n}\|q\|_{\mathcal H^{\odot2}}.
\]
Consequently,\(
\|R_{u,v}^{(k)}\|_{L^2(\Omega)}
\lesssim |v-u|^{2H}\). Let \(\pi_m\) be a sequence of partitions of \([s,t]\) such that \(|\pi_m|\to0\). Then
\[
\begin{aligned}
\left\|
\sum_{[u,v]\in\pi_m}R_{u,v}^{(k)}
\right\|_{L^2(\Omega)}
&\leq
\sum_{[u,v]\in\pi_m}
\|R_{u,v}^{(k)}\|_{L^2(\Omega)}
\lesssim
\sum_{[u,v]\in\pi_m}|v-u|^{2H}
\leq
|\pi_m|^{2H-1}(t-s).
\end{aligned}
\]
Since \(H>1/2\), this converges to \(0\). Passing to Wick--Riemann sums
in \eqref{eq:exact_identities} and using
\eqref{eq:convergence_rosenblatt} proves
\eqref{thesis_polynomials}. We finally pass to a generic polynomial \(p\). Let \(N=\deg p\) and set $a_k(u):=\E[p^{(k)}(X_u^H)]$. By \autoref{cor_stroock},
\begin{equation}\label{eq:appell_expansion_p}
p(X_u^H)
=
\sum_{k=0}^N
\frac{a_k(u)}{k!}(X_u^H)^{\dia k}.
\end{equation}
Moreover, applying the same expansion to \(p^{(n-1)}\) yields
\begin{equation}\label{eq:appell_expansion_derivative}
p^{(n-1)}(X_u^H)
=
\sum_{k=n-1}^N
\frac{a_k(u)}{(k-n+1)!}
(X_u^H)^{\dia(k-n+1)}.
\end{equation}
The functions \(u\mapsto a_k(u)\) are continuous, and hence bounded,
on \([s,t]\). Multiply \eqref{eq:exact_identities} by \(a_k(u)/k!\), sum over
\(k=0,\ldots,N\), and use
\eqref{eq:appell_expansion_p}. For every fixed \(n\geq2\), the
corresponding correction term becomes
\[
\begin{aligned}
\sum_{k=n-1}^N
\frac{a_k(u)}{k!}
\frac{k!}{(k-n+1)!}
(X_u^H)^{\dia(k-n+1)}
&=
p^{(n-1)}(X_u^H),
\end{aligned}
\]
by \eqref{eq:appell_expansion_derivative}. Since only finitely many
values of \(k\) occur and the coefficients \(a_k\) are bounded, the
resulting remainder still satisfies an estimate of order
\(|v-u|^{2H}\). Passing to Riemann sums therefore gives
\eqref{change_of_chaos_integrals}.
\end{proof}

The preceding argument also suggests a multidimensional change-of-chaos
formula. Let
\[
X_t=(X_t^1,\ldots,X_t^d),
\qquad
X_t^i=\I^2(f_t^i),
\qquad
f_t^i\in\mathcal H^{\odot2}.
\]
The associativity of \(\star\) removes the need to choose a
parenthesization for iterated contractions. The remaining difficulty is its noncommutativity: in general, the contraction associated with a path
depends on the order in which its vertices are visited. Let \(J\) be a multiset with elements in \(\{1,\ldots,d\}\), and let
\(\mathfrak S(J)\) denote the set of its distinct listings. For
\(|J|=n\), define
\[
f_t^J
:=
\operatorname{Sym}_2
\left(
\sum_{(j_1,\ldots,j_n)\in\mathfrak S(J)}
f_t^{j_1}\star\cdots\star f_t^{j_n}
\right),
\qquad
X_t^J:=\I^2(f_t^J).
\]
Since \(\mathfrak S(J)\) is invariant under reversal and transposition
reverses a \(\star\)-product, the kernel inside the preceding
symmetrization is already symmetric. Thus the symmetrization may
equivalently be omitted. With this convention, the natural multidimensional counterpart of
\eqref{change_of_chaos_integrals}, for a polynomial
\(p\colon\mathbb R^d\to\mathbb R\), is
\[
\int_s^t
\partial_i p(X_u)\dia \dif X_u^i
=
\int_s^t
\partial_ip(X_u)\dia_W \dif X_u^i
+
\sum_{|J|\geq2}
2^{|J|-1}
\int_s^t
\partial_J p(X_u)\dia_W \dif X_u^J,
\]
where the first index is summed according to the Einstein convention,
the second sum runs over multisets \(J\), and \(\partial_J\) is
independent of the chosen listing of \(J\). The reason multisets suffice in the gradient case is that the mixed derivatives of \(p\) are symmetric. After summing over the
possible components carrying the increment, the terms containing
exactly one increment kernel combine into the linear part of
\(g_{u,v}^J\). By multilinearity of \(\star\), the remaining terms
contain at least two increment kernels and are therefore of order
\(|v-u|^{2H}\), under the corresponding uniform \(H\)-Hölder
assumptions. We do not state this identity as a theorem, since a complete formulation would require precise multidimensional convergence assumptions and a careful bookkeeping of the ordered path contractions. For a general
non-gradient vector field, the component of the vector field carrying the increment remains distinguished. The correction terms should then be indexed by pointed words, rather than by multisets, since this distinguished index cannot in general be symmetrized with the derivative
indices.

\appendix
\section{Several complex variables}\label{app:proof}

The purpose of this appendix is to extend to the case some results on entire functions exponential type to the case of several complex variables; the techniques used in the proof are the same as in the one-dimensional case.

The growth of the derivatives of a generic holomorphic function $f$ can be estimated by an application of Cauchy's formula centred at the point $z$ with arbitrary positive radius. This estimate, however, generally grows factorially with the order of differentiation. The following lemma leverages the exponential type of $f$ to obtain a much sharper estimate. It works by optimising the choice of the radius based on the order of differentiation, cf.\ \cite[Lemma 2.1]{ButerinFreilingYurko2014} and \cite{boasEntire}. Recall that we write $\lesssim_a$ to mean boundedness up to a multiplicative constant which depends on $a$.

\begin{lemma}\label{lem:growth}
	Let $f \colon \mathbb C^d \to \mathbb C$ be entire. Assume that there exist
	$r \in [0,\infty)^d$ and $C>0$ such that
	\[
	|f(z)| \leq C \exp(r\cdot |z|), \qquad z\in \mathbb C^d .
	\]
	Then, for every polyradius $s>r$, there exists
	a constant $B_{C,r,s}>0$ such that, for every multiindex $K\in \mathbb N^d$ and
	every $z\in \mathbb C^d$,
	\[
	|\partial_K f(z)| \leq B_{C,r,s} s^K \exp(r\cdot |z|).
	\]
\end{lemma}

\begin{proof}
	Fix $z\in \mathbb C^d$, a multiindex $K\in\mathbb N^d$, and a polyradius
	$a\in(0,\infty)^d$. Applying Cauchy's formula in each coordinate, we have
	\[
	\partial_K f(z)
	=
	\frac{K!}{(2\pi i)^d}
	\int_{|\zeta_1-z_1|=a_1}\cdots
	\int_{|\zeta_d-z_d|=a_d}
	\frac{f(\zeta)}
	{\prod_{j=1}^d(\zeta_j-z_j)^{K_j+1}}
	\,\dif \zeta_d\cdots \dif \zeta_1 .
	\]
	Therefore, using the growth assumption on $f$
	\begin{align*}
		|\partial_K f(z)|
		&\leq
		K!a^{-K}
		\sup_{|\zeta-z|=a}|f(\zeta)| \lesssim_C \exp({r\cdot |z|}) K!a^{-K}\exp({r\cdot a}).
	\end{align*}
	Now choose $a$ depending on $K$ by
	\[
	a_j=
	\begin{cases}
		K_j/s_j, & K_j\geq 1,\\
		1, & K_j=0.
	\end{cases}
	\]
	Then, by Stirling's estimate applied factor-wise
	\begin{align*}
		K!a^{-K}\exp({r\cdot a})
		&=
		\prod_{K_j=0} \exp({r_j})
		\prod_{K_j\geq 1}
		K_j!\left(\frac{s_j}{K_j}\right)^{K_j}
		\exp\left(\frac{r_j}{s_j}K_j\right) \\
		&\lesssim \prod_{K_j=0} \exp({r_j})
		\prod_{K_j\geq 1}
		\sqrt{K_j} \bigg(\frac{K_j}{e}\bigg)^{K_j} \left(\frac{s_j}{K_j}\right)^{K_j}
		\exp\left(\frac{r_j}{s_j}K_j\right) \\
		&= \prod_{K_j=0} \exp({r_j})
		\prod_{K_j\geq 1}\sqrt{K_j} s_j^{K_j} \exp\!\left[K_j \left(\frac{r_j}{s_j} - 1\right)\right] \\
		&\lesssim_{r,s} s^K
	\end{align*}
	since $s > r$ implies $\sup_k\sqrt k \exp [ - k (1 - r_j/s_j)] < \infty$.
\end{proof}

\begin{proof}[Proof of \autoref{thm:martin}]
	
	We first prove the claim about holomorphy of $\EuScript{B}f$ on $E_r$. By \autoref{lem:growth} (suppressing dependence on $C,r$ which are fixed)
	\[
	|\partial_K f(0)| \lesssim_s s^K.
	\]
	Then for any $t > s$ and $|\theta| \geq t$ we have
	\[
	\bigg| \frac{\partial_K f(0)}{\theta^{K+1}} \bigg| \lesssim_s \frac{s^K}{t^{K+1}} \implies \sum_K \bigg| \frac{\partial_K f(0)}{\theta^{K+1}} \bigg| \lesssim_s \frac{1}{t^1} \sum_K \bigg(\frac{s}{t}\bigg)^K < \infty.
	\] 
	Since $s,t$ were arbitrary polyradii with $r < s < t$, the series defining $\EuScript{B}f$ then converges absolutely, uniformly on compacts in $E_r$ and by the Weierstrass theorem $\EuScript{B}f$ is holomorphic there.
	
	Now, substituting $\EuScript{B}f$ for its Laurent series and the Taylor series in $\theta = 0$ for $\theta \mapsto \exp({z \cdot \theta})$, both converging absolutely uniformly on the torus $|\theta| = \rho$ for any $\rho > r$, we obtain the Borel inversion formula:
	\begin{align*}
		\frac{1}{(2\pi i)^d} \int_{|\theta| = \rho} \exp({z \cdot \theta}) \EuScript{B}f(\theta) \dif \theta &= \sum_{H, K} \frac{\partial_K f(0)}{H!} z^H \frac{1}{(2\pi i)^d} \int_{|\theta|= \rho} \theta^{H - K - 1} \dif \theta = \sum_{K} \frac{\partial_K f(0)}{K!} z^K = f(z).
	\end{align*}
	Furthermore if $r < \rho < R$, $\EuScript{M}$ does not vanish on $\{|\theta| = \rho\}$ and
	\begin{align*}
		f(z) &= \frac{1}{(2\pi i)^d} \int_{|\theta| = \rho} \frac{\exp({z \cdot \theta})}{\EuScript{M}(\theta)} \EuScript{M}(\theta) \EuScript{B}f(\theta) \dif \theta \\
		&= \frac{1}{(2\pi i)^d} \int_{|\theta| = \rho} \sum_K \frac{z^{\dia K}}{K!} \theta^K \EuScript{M}(\theta) \EuScript{B}f(\theta) \dif \theta \\
		&= \sum_K  \frac{1}{(2\pi i)^d} \bigg[ \int_{|\theta| = \rho} \theta^K  \EuScript{M}(\theta) \EuScript{B}f(\theta) \dif \theta \bigg] \frac{z^{\dia K}}{K!},
	\end{align*}
	where we have exchanged integration and summation thanks to the fact that the Appell generating function \eqref{eq:wickGenerating} is the Taylor expansion in $\theta = 0$ of a function that is holomorphic on $D_R$ and thus converges uniformly on $\{|\theta| = \rho\} \Subset D_R$. This proves the claim regarding the expansion with $z$ fixed, and independence of the coefficients of $\rho$ is a consequence of the integrands being holomorphic in the polyannulus $\{ r < |\theta| < R\}$, by successive applications of the one-variable Cauchy theorem.
	
	We now proceed to estimate the coefficients and consequent uniform convergence. Fix $r < \rho < \sigma < R$. A Cauchy estimate on the polydisc $\{|\theta| \leq \sigma \}$ (similar to the one performed earlier for $\partial_K f(0)$) gives
	\[
	\bigg| \frac{z^{\dia K}}{K!} \bigg| \leq \max_{|\theta| = \sigma} \bigg| \frac{\exp({\theta \cdot z})}{\EuScript{M}(\theta)} \bigg| \sigma^{-K}.
	\]
	Then 
	\begin{align*}
		\bigg| L_K(f) \frac{z^{\dia K}}{K!} \bigg| \leq \rho^1 \max_{|\theta| = \rho} |\EuScript{M}(\theta) \EuScript{B}f(\theta)| \cdot \max_{|\theta| = \sigma} \bigg| \frac{\exp({\theta \cdot z})}{\EuScript{M}(\theta)}\bigg| \cdot \bigg( \frac{\rho}{\sigma} \bigg)^K .
	\end{align*}
	with $\rho^1 = \rho_1 \cdots \rho_d$ as usual coming from the volume of the torus (and we have temporarily suppressed the obvious dependence of $c_{\rho, \sigma}$ on $\mu, f$). Absolute convergence, uniformly on compacts follows by taking the supremum over $z$ in a compact $U$ and by summability of the geometric series. Absolute convergence in $L^1$ follows similarly, by restricting to $\Rd$ and integrating the exponential bound, given that $\sigma < R$.
	
	For the second statement, let $s_f,s_g,\sigma$ such that $s_f > r_f$, $s_g > r_g$ and $s_f + s_g < \sigma < R$. By the first part of the statement we have
	\[
	|a_I| \lesssim_{s_f} s^I_f, \qquad |b_J| \lesssim_{s_g} s^J_g, \qquad
	\bigg| \frac{z^{\dia K}}{K!} \bigg| \lesssim_\sigma \sigma^{-K}.
	\]
	Therefore
	\[
	\bigg|a_I b_J\frac{z^{\dia I \sqcup J}}{I! J!}\bigg| \lesssim_{s_f, s_g, \sigma, U} \binom{I + J}{I} \bigg( \frac{s_f}{\sigma} \bigg)^I \bigg( \frac{s_g}{\sigma} \bigg)^J,
	\]
	which is summable since
	\[
	\sum_{I, J}  \binom{I + J}{I} \bigg( \frac{s_f}{\sigma} \bigg)^I \bigg( \frac{s_g}{\sigma} \bigg)^J = \sum_K \sum_{I \leq K} \binom{K}{I} \bigg( \frac{s_f}{\sigma} \bigg)^I \bigg( \frac{s_g}{\sigma} \bigg)^{K - I} = \sum_K \bigg(\frac{s_f + s_g}{\sigma} \bigg)^K < \infty,
	\]
	and absolute convergence in the two modes follows.
\end{proof}

\bibliographystyle{alpha}
\bibliography{refs}

\end{document}